\documentclass[11pt,a4paper]{article}
\usepackage[arxiv]{mymacros}
\usepackage[normalem]{ulem}

\author{Roland Bauerschmidt\thanks{University of Cambridge, Statistical Laboratory, DPMMS. Email: {\tt rb812@cam.ac.uk}.}
  \and Tyler Helmuth\thanks{University of Bristol, School of Mathematics. Email: {\tt th17948@bristol.ac.uk}.}
  \and Andrew Swan\thanks{University of Cambridge, Statistical Laboratory, DPMMS. Email: {\tt acks2@cam.ac.uk}.}
  }
\title{The geometry of random walk isomorphism theorems}
\date{July 2, 2020}

\let\originalleft\left
\let\originalright\right
\renewcommand{\left}{\mathopen{}\mathclose\bgroup\originalleft}
\renewcommand{\right}{\aftergroup\egroup\originalright}
\usepackage{tocloft}
\begin{document}

\maketitle

\begin{abstract}
  The classical random walk isomorphism theorems relate the
  local times of a continuous-time random walk to the square of a
  Gaussian free field. A Gaussian free field is a spin system that
  takes values in Euclidean space, and this article generalises the
  classical isomorphism theorems to spin systems taking values in
  hyperbolic and spherical geometries. The corresponding random walks
  are no longer Markovian: they are the vertex-reinforced and
  vertex-diminished jump processes. We also investigate supersymmetric
  versions of these formulas.

  Our proofs are based on exploiting the continuous symmetries of the
  corresponding spin systems. The classical isomorphism theorems use
  the translation symmetry of Euclidean space, while in hyperbolic and
  spherical geometries the relevant symmetries are Lorentz boosts and
  rotations, respectively. These very short proofs are new even in the
  Euclidean case.

  Isomorphism theorems are useful tools, and to illustrate this we
  present several applications. These include simple proofs of
  exponential decay for spin system correlations, exact formulas for
  the resolvents of the joint processes of random walks together
  with their local times,
  and a new derivation of the Sabot--Tarr\`{e}s formula for
  the limiting local time of the vertex-reinforced jump process.
\end{abstract}

\setcounter{tocdepth}{1}
\tableofcontents

\section{Introduction}

Random walk isomorphism theorems refer to a class of
distributional identities that relate the local times of Markov
processes to the squares of Gaussian fields. These theorems, which
connect two different types of probabilistic objects, have their
origins in the work of the physicist K.\ Symanzik~\cite{Syma69}.
Isomorphism theorems have been useful in the
investigation of a variety of phenomena, and they can be used in
two directions: to study field theoretic questions in terms
of random walks, and to study random walks in terms of field
theory.  An incomplete list of topics
investigated via isomorphism theorems includes: local times of
Markov processes~\cite{MR2250510} and their large
deviations~\cite{MR1100240,MR2330973}; cover times and thick points of the
simple random walk~\cite{MR2912708,1809.04369,1903.04045};
four-dimensional self-avoiding walk~\cite{MR1143413,MR3339164};
$\phi^4$ field theory~\cite{MR719815,BrydgesHelmuthHolmes, MR643591};
and random walk loop soups~\cite{MR2932978, MR941982}.

The purpose of this article is to expand the scope of
isomorphism theorems beyond Gaussian fields. Namely, we describe,
and make use of, isomorphism theorems that relate \emph{non-Markovian}
stochastic processes to \emph{non-Gaussian} spin systems. Our
proofs also provide a new perspective on isomorphism theorems: they are
consequences of the symmetries of the underlying spin systems.

In Section~\ref{sec:isom-theor-hyperb} below we give an introduction
to isomorphism theorems and the processes this article
is concerned with. Before doing this, we briefly summarise the new results
contained in this article:
\begin{itemize}
\item New and efficient proofs of the
  Brydges--Fr\"ohlich--Spencer--Dynkin (BFS--Dynkin), Eisenbaum, and
  second generalised Ray--Knight isomorphism theorems for the simple
  random walk (SRW).  These results are all derived in a few pages
  from a more general Ward identity for the Gaussian free field.
\item New and efficient proofs of supersymmetric versions of the
  isomorphism theorems for the SRW. In particular, we prove a previously
  unknown supersymmetric version of the generalised second Ray--Knight
  isomorphism.  For the reader's convenience we also present an
  introduction to supersymmetry directed towards probabilists in an
  appendix.
\item New isomorphism theorems connecting the vertex-reinforced jump
  process (VRJP) with hyperbolic sigma models, and supersymmetric
  versions of these theorems.  The analogue of the BFS--Dynkin
  isomorphism previously appeared in \cite{1802.02077}, and here we
  also establish analogues of the Eisenbaum and Ray--Knight
  isomorphism theorems.  Our proofs are geometric and do not
  rely on any particular set of coordinates. In particular, we do
  not use horospherical coordinates.
\item New isomorphism theorems for the vertex-diminished jump process
  (VDJP).  The VDJP is connected to a spin model taking values in the
  hemisphere.  It previously appeared in the context of the
  Ray--Knight isomorphism theorem for SRW in~\cite{MR3520013}.
\end{itemize}

We also give several applications of these isomorphism theorems.  In
Section~\ref{sec:ST} we show that the Sabot--Tarr\`{e}s limit formula for
the local time of the VRJP~\cite{MR3420510} is a direct consequence of our
supersymmetric Ray--Knight theorem for the $\HH^{2|2}$ model. In Section~\ref{sec:TC-HK} we show
how isomorphism theorems yield fixed-time formulas and representations
of the resolvents for the joint processes of the random walks
together with their local times.
Lastly, we prove some results
concerning exponential decay of correlation functions for the
associated spin models in Section~\ref{sec:applications}.

\subsection{Isomorphism theorems for hyperbolic and spherical
  geometries}
\label{sec:isom-theor-hyperb}

Let $X_{t}$ be a continuous-time stochastic process on a finite state
space $\Lambda$ with associated local times
$\V{L}_{t} = (L_{t}^{i})_{i\in\Lambda}$.  The processes considered in
this paper are all of the form
\begin{equation}
  \label{eq:RW-models}
  \P\cb{X_{t+dt}=j \mid (X_{s})_{s\leq t}, X_{t}=i} =
  \beta_{ij}(1+\epsilon L^{j}_{t})\,dt, \qquad \epsilon\in \{-1,0,1\},
\end{equation}
where $\beta_{ij}\geq 0$ and $\beta_{ij}=\beta_{ji}$ for all $i,j\in\Lambda$.

The random walk models defined by \eqref{eq:RW-models} are
defined more precisely below. The models have all appeared
previously, though they have received varying amounts of
attention. When $\epsilon=0$ the model is the \emph{continuous-time simple
  random walk}; for $\epsilon=1$ it is the \emph{vertex-reinforced
  jump process} (VRJP) first studied in~\cite{MR1900324,MR2027294};
for $\epsilon=-1$ it is the \emph{vertex-diminished jump process}
(VDJP) which appeared in~\cite{MR3520013}. As the names suggest, the
VRJP is a random walk that is encouraged to revisit vertices it has
visited in the past, while the VDJP is discouraged from doing so.

Let $\R^{n}$ denote $n$-dimensional Euclidean space, $\HH^{n}$ denote
$n$-dimensional hyperbolic space, and let $\bbS^{n}_{+}$ denote the
upper hemisphere of the $n$-dimensional sphere. Below we will
introduce spin systems that take values in these spaces, and then
link these to the aforementioned random walks. The
spin systems are the $\R^{n}$-valued \emph{Gaussian free field} (GFF),
corresponding to the SRW; the $\HH^{n}$-valued \emph{hyperbolic spin
  model}, corresponding to the VRJP; and the $\bbS^{n}_{+}$-valued
\emph{hemispherical spin model}, corresponding to the VDJP.

To give a flavour of the relationships that we will establish, recall
Dynkin's formulation of an isomorphism linking the SRW and the
$\R$-valued GFF~\cite{MR734803}. Let $G=(\Lambda,E)$ be a finite
graph, $h>0$, and let $\avg{\cdot}$ denote the expectation of a GFF
$(u_{i})_{i\in\Lambda}$ with covariance $(-\Delta + h)^{-1}$. This is
often called the \emph{massive GFF} with \emph{mass} $m=\sqrt{h}$.
Let $\E_{i}$ denote the expectation of a continuous-time SRW $X_{t}$
with associated local time field $\V{L}_t = (L_t^i)_{i\in\Lambda}$,
started from $i\in\Lambda$, with $X_{t}$ independent of the GFF. Then
for all bounded $g\colon \R^{\Lambda}\to \R$,
\begin{equation}
  \label{eq:intro-dynkin}
  \spin{u_{i}u_{j} g(\frac{1}{2}\V{u}^{2})} = \spin{
    \int_{0}^{\infty} \E_{i}(g(\frac{1}{2} \V{u}^{2} +    \V{L}_{t}) 1_{X_{t}=j}) \, e^{-ht}\, dt},
  \qquad \V{u}^2 \bydef (u_i^2)_{i\in\Lambda}.
\end{equation}
The left-hand side is a generalization of the spin-spin correlation
between the spins $u_{i}$ and $u_{j}$ of the GFF. In particular,
taking $g=1$ in \eqref{eq:intro-dynkin} reveals the well-known fact
that the second moments of the massive GFF are given by 
the Green's function of a SRW killed at rate $h$.

In Theorems~\ref{thm:BHS} and \ref{thm:BHS-hemi} we establish
analogues of~\eqref{eq:intro-dynkin} for the hyperbolic and
hemispherical spin models; the hyperbolic case first appeared
in~\cite{1802.02077}. Our methods also allow us to establish other
isomorphism theorems. In particular, we give new proofs of the
Eisenbaum isomorphism theorem~\cite{MR1459468} and of the generalised
second Ray--Knight theorem~\cite{MR1813843} for the GFF, and we
establish analogues of these results for hyperbolic and hemispherical
spin models. Our proofs apply to $n$-component spin systems for
general $n \in \N = \{1,2,\dots\}$ in all cases, and even for
the GFF some of these results are new when $n> 1$. To ease our
exposition we will refer to the generalised second Ray--Knight theorem
as the Ray--Knight isomorphism in what follows.

\subsection{Supersymmetric isomorphism theorems}
\label{sec:supersymm-isom-theor}

There is another type of isomorphism that relates the simple random walk
to a spin system, in which the GFF is replaced by the
\emph{supersymmetric Gaussian free field (SUSY GFF)}. These
isomorphisms originated in work of McKane~\cite{MR594576} and Parisi
and Sourlas~\cite{PhysRevLett.43.744}. Supersymmetry has played a role
in several interesting probabilistic
problems~\cite{MR2031859,MR2728731,MR2525670,MR3622573}, and several of the
applications we mentioned in the opening paragraph of this article
involve the SUSY
GFF~\cite{MR1143413,MR3339164,MR1100240,MR2330973,MR941982}.

The most important aspect of the SUSY isomorphism for the SRW is
immediately apparent from the statement of the result, and hence we
defer a careful definition of the SUSY GFF to Section~\ref{sec:SUSY}.
Let $\avg{\cdot}$ now denote the expectation with respect to
the SUSY GFF. The SUSY isomorphism theorem is that for all smooth and
bounded $g\colon \R^{\Lambda}\to \R$,
\begin{equation}
  \label{eq:intro-susy}
  \avga{u_{i}^{1}u_{j}^{1} g(\frac{1}{2} \V{|u|}^{2})} = 
  \int_{0}^{\infty} \E_{i} (g(\V{L}_{t}) 1_{X_{t}=j}) \, e^{-ht}\, dt,
  \qquad \V{|u|}^2 \bydef (|u_i|^2)_{i\in\Lambda}.
\end{equation}
The key point of~\eqref{eq:intro-susy} is that the right-hand side
only involves the simple random walk, while the left-hand side involves
only the components $(u_{i})_{i\in\Lambda}$ of the SUSY GFF. Thus
questions about the local time of random walk can be rephrased purely
in terms of the SUSY GFF.

The viewpoint that isomorphism theorems arise as a consequence of
continuous symmetries applies equally well to supersymmetric spin
systems. Beyond proving \eqref{eq:intro-susy}, Section~\ref{sec:SUSY}
also establishes results analogous to \eqref{eq:intro-susy} for the
supersymmetric $\HH^{2|2}$ and $\bbS^{2|2}_{+}$ models, and
moreover we prove a SUSY variant of the
Ray--Knight isomorphism. This is new even for the simple random walk. We
emphasise that these theorems give direct access to the local times
of the non-Markovian VRJP and VDJP in terms of the spin models. The
analogue of \eqref{eq:intro-susy} for $\HH^{2|2}$ first appeared
in~\cite{1802.02077}.

\subsection{Proof ideas}
\label{sec:proof-ideas}

Our proofs of isomorphism theorems all follow a common strategy. The
spin systems we consider possess continuous symmetries, and as a
result satisfy integration by parts formulas that are called
\emph{Ward identities} in the physics literature. Isomorphism theorems
are a direct  consequence of these Ward identities. 

A key step is to consider a random walk $X_{t}$ to be a marginal of
the joint process $(X_{t},\V{L}_{t})$ of the walk and its local times
together. Our Ward identities can be rephrased in terms of the
infinitesimal generator of this joint process, and \emph{all} of our
isomorphism theorems follow quite quickly by choosing appropriate
specializations of the Ward identities. In particular, this gives a
unified set of proofs of the BFS--Dynkin, Eisenbaum,
and Ray--Knight isomorphism theorems for the SRW.

\subsection{Structure of this article}
\label{sec:struct-this-article}

Section~\ref{sec:euclidean} gives our new proofs of the classical
isomorphism theorems that link random walks
to Gaussian fields. We present our arguments in detail in this
familiar context as very similar ideas are used in
Sections~\ref{sec:hyperbolic} and~\ref{sec:spherical}, which derive
isomorphism theorems for the VRJP and VDJP. We derive supersymmetric
isomorphisms for the SRW, the VRJP, and the VDJP in
Section~\ref{sec:SUSY}, and
Sections~\ref{sec:ST} through \ref{sec:applications} concern applications
of our new isomorphisms.

To keep this article self-contained, Appendix~\ref{sec:SUSY-intro}
contains an introduction to the parts of supersymmetry needed to
understand our supersymmetric isomorphisms and their
applications. In Appendix~\ref{sec:fSUSY} we discuss some
further aspects of symmetries and supersymmetries that are not
needed for our results, but that help place the results of this
article in context.

\subsection{Related literature and future directions}
\label{sec:related-literature}

\paragraph{Related literature}
For monograph-length treatments of isomorphism theorems and related
topics, e.g., loop soups, see~\cite{MR2932978,MR2250510}.
Many proofs of various isomorphism theorems have been given; here we
mention only the recent~\cite{MR3520013,1607.05201}. The major
innovation in the present work is that we do not rely on Gaussian
calculations. This is important both for obtaining results for
$\HH^{n}$ and $\bbS^{n}_{+}$, and for obtaining supersymmetric
variants.

\paragraph{Future directions}

This article describes isomorphism theorems that link spin systems on
$\R^{n}$, $\HH^{n}$, and $\bbS^{n}_{+}$ (and the supersymmetric
versions when $n=2$) to random walks. This provides a partial answer
to a question of Kozma~\cite{MR3469136}, who asked if there are other
spin models (beyond the $\HH^{2|2}$ model) with associated random
walks. The development of a more systematic connection between spin
models and random walks would be very interesting. In particular, it
is natural to wonder if there are geometric spaces beyond $\R^{n}$,
$\HH^{n}$, and $\bbS^{n}_{+}$ that have associated isomorphism
theorems.

Another interesting future direction would be to clarify the relation
between our new isomorphism theorems and loop soups. In the setting of the
SRW this connection is well-developed~\cite{MR2932978,MR2250510} ---
do these connections extend to the VRJP and VDJP? Similar
questions can be asked about random interlacements; for recent
  progress in this direction see~\cite{1903.07910}.

\subsection{Notation and conventions}
\label{sec:NC}

$\Lambda$ will be a finite set and
$\beta = (\beta_{ij})_{i,j\in\Lambda}$ will be a set of edge weights,
i.e., $\beta_{ij}=\beta_{ji}\geq 0$.  The edge weights induce a graph
with vertices $\Lambda$ and edge set $\{\{i,j\} \mid \beta_{ij}>0\}$,
and we will assume that this graph is connected. We also let
$\V{h} = (h_{i})_{i\in\Lambda}$ denote a set of non-negative vertex
weights; here we are setting a convention that bold symbols denote
objects indexed by $\Lambda$. Both $\beta$ and $\V{h}$ will play the
role of parameters in our models.  For typographical reasons we will
sometimes write $h$ in place of $\V{h}$ when there is no risk of confusion.

Suppose $V$ is a set equipped with a binary operation $(x,y)\mapsto x\cdot y$.
We write $V^{\Lambda}$ for the set of
maps from $\Lambda$ to $V$, denote elements of this set by
$\V{u}=(u_{i})_{i\in\Lambda}$, and let $|\V{u}|^{2} = (u_{i}\cdot u_{i})_{i\in\Lambda}$.
If elements of $V$ are vectors, e.g.,
$u_{i}=(u_{i}^{1},\dots, u_{i}^{n})\in\R^{n}$, then we write
$\V{u}^{\alpha} = (u_{i}^{\alpha})_{i\in\Lambda}$ for the collection
of $\alpha^{\text{th}}$ components.

For a function $f\colon \R \to \R$ we often impose that $f$ is smooth
and has \emph{rapid decay}.  A sufficient condition is that $f$ and
its derivatives decay faster than any polynomial: for every $p$ and
$k$, there are constants $C_{p,k}$ such that the $k$th derivative
satisfies $|f^{(k)}(u)|\leq C_{p,k}|u|^{-p}$. If
$f\colon \R^{n}\to \R$,
$(u_{1},\dots, u_{n})\mapsto f(u_{1},\dots, u_{n})$, then we say $f$
has \emph{rapid decay in $u_{1}$} if $f(\cdot,u_{2},\dots, u_{n})$ has
rapid decay with constants uniform in $u_{2},\dots, u_{n}$. Rapid
decay in $u_{j}$ is defined analogously, and we say such an $f$ has
\emph{rapid decay} if it has rapid decay in some coordinate.  For a
non-smooth function $f$, we say that $f$ has rapid decay if the the
above holds with $k=0$.

Similarly, we often impose that $f\colon \R^{n}\to\R^{m}$ has
\emph{moderate growth}. A sufficient condition is that $f$ has at most
polynomial growth, i.e., there exists $q$ and $C_{k}$ such that
$|\nabla^kf(u)|\leq C_{k}|u|^q$ for all $k$.

Given a function $f\colon\Lambda\times\R^{\Lambda}\to \R$,
$(i,\V{\ell})\mapsto f(i,\V{\ell})$ we say $f$ is \emph{smooth},
\emph{rapidly decaying}, etc.\ if it has this property with respect to its
second coordinate $\V{\ell}$. Throughout we will assume
functions are Borel measurable without making this explicit.

\section{Isomorphism theorems for flat geometry}
\label{sec:euclidean}

In this section we introduce the simple random walk, the corresponding
Gaussian free field, and several well-known isomorphism theorems
relating these objects. The method of proof will be used repeatedly in
the remainder of the paper when we consider other spin systems.  An
important aspect of the proofs is that they do not rely on explicit
Gaussian computations; this is essential for the generalization of
these theorems to non-Gaussian spin systems. Our proofs also show that
these results are true for GFFs with any number of components.

\subsection{Simple random walk and Gaussian free field}
\label{sec:basic-definitions}

\paragraph{Simple random walk}
\label{sec:SRW}

The continuous-time \emph{simple random walk} (SRW) on $\Lambda$
with symmetric edge weights $\beta \bydef (\beta_{ij})_{i,j\in\Lambda}$,
i.e., $\beta_{ij}=\beta_{ji} \geq 0$, is the
Markov jump process $(X_{t})_{t\geq 0}$ with transition rates
\begin{equation}
  \label{eq:SRW-J}
  \P\cb{X_{t+dt}=i\mid X_{t}=j} = \beta_{ij}\,dt.
\end{equation}
We write $\P_{i}$ and $\E_{i}$ for the law and expectation of $X_t$ when
it is started from the vertex $i$.  Formally, $X_{t}$ is a
continuous-time Markov process with generator $\Delta_{\beta}$, where
the Laplacian $\Delta_{\beta}$ is the matrix indexed by
$\Lambda$ that acts on $f \colon \Lambda \to \R$ by
\begin{equation}
  \label{eq:Lact}
  (\Delta_{\beta}f)(i)
  \bydef
  \sum_{j\in\Lambda}\beta_{ij}(f(j)-f(i)).
\end{equation}

In what follows it will be useful to view $X_t$ as a marginal of the
Markov process $(X_{t},\V{L}_{t})_{t\geq 0}$ consisting of $X_t$ and its
\emph{local times} $\V{L}_{t} \bydef (L^{i}_{t})_{i\in\Lambda}$, which
are defined by 
\begin{equation}
  \label{eq:local}
  L_{t}^{i} \bydef L^{i}_{0}+\int_{0}^{t} 1_{X_{s}=i}\,ds, \qquad i\in\Lambda,
\end{equation}
where the vector $\V{L}_{0}$ is a collection of free parameters called
the \emph{initial local time}. A short computation shows that the
generator of $(X_t,\V{L}_t)$ acts on smooth functions
$f \colon \Lambda\times\R^{\Lambda}\to \R$ by
\begin{equation}
  \label{eq:SRW-gen}
  (\cL f)(i,\V{\ell}) = (\Delta_{\beta}f)(i,\ell) + \ddp{f(i,\V{\ell})}{\ell_{i}}
  ,\quad \text{i.e.,} \quad
  \cL \V{f} = \Delta_\beta \V{f} + \partial \V{f},
\end{equation}
where $\Delta_\beta$ only acts on the first argument and the last
equation uses the vector notation
\begin{equation} \label{e:vector-ell}
  \V{f} \bydef
  (f(i,\V{\ell}))_{i\in \Lambda},
  \quad
  \partial\V{f} \bydef
  (\ddp{f(i,\V{\ell})}{\ell_i})_{i\in \Lambda}.
\end{equation}
We write $\P_{i,\V{\ell}}$ for the law of $(X,\V{L})$ started at
$(i,\V{\ell})\in\Lambda\times\R^{\Lambda}$, and $\E_{i,\V{\ell}}$ for
its expectation. Note that
$\E_{i,\V{\ell}}f(X_t,\V{L}_t) = \E_{i,\V{0}}f(X_t,\V{\ell}+\V{L}_t)$,
and in particular that $f_t(i,\V{\ell}) \bydef \E_{i,\V{\ell}}f(X_t,\V{L}_t)$
is a smooth function with rapid decay
in $\V{\ell}$ if $f$ is smooth with rapid decay.

\paragraph{Gaussian free field}
\label{sec:GFF}

The ($n$-component) \emph{Gaussian free field} (\emph{GFF} or
\emph{$\R^{n}$ model}) is a spin system
taking values in $\R^{n}$. Its configurations are elements
$\V{u}\in (\R^{n})^{\Lambda}$; by an abuse of notation we will write
$\R^{n\Lambda}$ in place of $(\R^{n})^{\Lambda}$.
Let $\V{h} = (h_{i})_{i\in\Lambda}$, and assume $h_{i}\geq 0$.
To define the probability of a configuration, let
\begin{equation}
  \label{eq:GFF-H}
  H_{\beta}(\V{u})
  \bydef
  \frac{1}{2}(\V{u},-\Delta_\beta \V{u}),
  \qquad H_{\beta,h}(\V{u})\bydef
   H_{\beta}(\V{u}) + {\frac{1}{2}(\V{h},|\V{u}|^2)},
\end{equation}
where $(\V{f},\V{g}) \bydef \sum_{i\in\Lambda}f_{i}g_{i}$, 
$|\V{u}|^{2}\bydef (u_{i}\cdot u_{i})_{i\in\Lambda}$, and
$\cdot$ is the Euclidean inner product. In
\eqref{eq:GFF-H} the Laplacian acts diagonally on the $n$ components
of $\V{u}$, i.e.,
$\Delta_{\beta}\V{u} = (\Delta_{\beta}\V{u}^{\alpha})_{\alpha =1}^{n}$,
and hence \eqref{eq:GFF-H} can be rewritten using
\begin{equation}
  (\V{u},-\Delta_\beta \V{u})
  = 
  \frac{1}{2}\sum_{i,j\in \Lambda}\beta_{ij}(u_{i}-u_{j})^{2}
,\qquad
   (\V{h},|\V{u}|^2) =
   \sum_{i\in \Lambda}h_{i}u_{i}\cdot u_{i},
\end{equation}
where $(u_{i}-u_{j})^{2}$ is shorthand for $(u_{i}-u_{j})\cdot (u_{i}-u_{j})$.
Note that another common notation is $h_{i}=m_{i}^{2}\geq 0$,
and $m_{i}$ is called the \emph{mass} at the vertex $i$.  Define the
unnormalised expectation $\cb{\cdot}_{\beta,h}$ on functions
$F\colon \R^{n\Lambda}\to \R$ by
\begin{equation}
  \label{eq:GFF-un}
  \cb{F}_{\beta,h} \bydef \int_{\R^{n\Lambda}}F(\V{u})e^{-H_{\beta,h}(\V{u})}d\V{u},
\end{equation}
where the integral is with respect to Lebesgue measure $d\V{u}$ on $\R^{n\Lambda}$. 
We set $\cb{\cdot}_{\beta} \bydef\cb{\cdot}_{\beta,0}$.  

The \emph{Gaussian free field}
is the probability measure on $\R^{n\Lambda}$ defined by the
normalised expectation
\begin{equation}
  \label{eq:GFF-n}
  \avg{F}_{\beta,h} \bydef \frac{1}{Z_{\beta,h}}\cb{F}_{\beta,h}
 =\frac{\q{F e^{-{\frac{1}{2}(\V{h},|\V{u}|^2)}}}_{\beta}} 
  {\q{ e^{-{\frac{1}{2}(\V{h},|\V{u}|^2)}}}_{\beta}}
  ,
  \qquad Z_{\beta,h} \bydef \cb{1}_{\beta,h}.
\end{equation}
Note that for the expectation in \eqref{eq:GFF-n} to be
well-defined we must have $Z_{\beta,h} <\infty$; this is the case
if and only if $h_{i}>0$ for some $i$. The divergence if $\V{h}=\V{0}$ is due
to the invariance of $H_{\beta}(\V{u})$ under the simultaneous
translation $u_{i}\mapsto u_{i}+s$ for any $s\in\R^{n}$.

\subsection{Fundamental integration by parts identity}
\label{sec:ibp}

For any differentiable $f\colon \R^{n\Lambda}\to \R$ we write
\begin{equation}
  \label{eq:gen-SRW-1}
  T_{j}f \bydef \ddp{f}{u_{j}^{1}},
  \quad \quad
  \V{T} f \bydef (T_if)_{i\in\Lambda}
  .
\end{equation}
Thus $T_j$ is the infinitesimal generator of translations of the $j^{\text{th}}$
coordinate in the direction $e^{1} =(1,0,\dots,0) \in\R^{n}$.  The
following lemma is a consequence of the translation invariance of
Lebesgue measure, and we will derive all of our
isomorphism theorems from this
identity.  In later sections
of this paper we will derive analogous
results by replacing the translation symmetry by different
symmetries.

\begin{lemma}
  \label{lem:gen-SRW}
  Let $[\cdot]_\beta$ be the unnormalised expectation of the $\R^n$
  model, and let $\E_{i,\V{\ell}}$ be the expectation of the SRW.  Let
  $f\colon \Lambda\times \R^{\Lambda} \to \R$ be smooth with rapid
  decay, and let $\rho\colon \R^{n\Lambda}\to \R$ be smooth with
  moderate growth.  Then:
  \begin{equation}
    \label{eq:gen-lem}
    -\sum_{j\in\Lambda}\cb{\rho(\V{u})u_{j}^{1} \cL f(j, \frac{1}{2} |\V{u}|^{2})}_{\beta}
    =
    \sum_{j\in\Lambda} \cb{(T_j\rho)(\V{u})f(j,\frac{1}{2}|\V{u}|^{2})}_{\beta}
    .
  \end{equation}
  In particular, the following integrated version holds for all
  $f\colon\Lambda\times\R^{\Lambda}\to\R$ with rapid decay:
  \begin{equation}
    \label{eq:generic-iso}
    \sum_{j\in\Lambda}\uspin{\rho(\V{u})u_{j}^{1} f(j, \frac{1}{2} |\V{u}|^{2})}_{\beta}
    =
    \sum_{j\in\Lambda} \uspin{(T_j\rho)(\V{u})\int_0^\infty
      \E_{j,\frac{1}{2}|\V{u}|^2}(f(X_t,\V{L}_t))\,dt}_{\beta}
    .
\end{equation}
\end{lemma}

\begin{remark}
  \label{rem:srw-compact}
  Using \eqref{e:vector-ell} and with
  $(\V{T},\V{f}) \bydef \sum_{i\in\Lambda} T_{i}f_{i}$,
  \eqref{eq:gen-lem} can be restated compactly as
  \begin{equation} 
    \label{e:TLrho}
    -\uspin{(\rho(\V{u})\V{u}^1, (\mathcal{L} \V{f})(\frac{1}{2}|\V{u}|^{2}))}_\beta
    =
    \uspin{(\V{T}\rho(\V{u}), \V{f}(\frac{1}{2}|\V{u}|^{2}))}_\beta
    .
  \end{equation}
\end{remark}

\begin{proof}
  We first prove~\eqref{eq:gen-lem} by integration by parts.  If
  $f_{1},f_{2}\colon \R^{n\Lambda}\to \R$ are differentiable and have
  rapid decay, then integration by parts implies
  \begin{equation}
    \label{eq:gen-SRW-2}
    \cb{(T_{j}f_{1}) f_{2}}_{\beta} = \cb{f_{1} (T^{\star}_{j}f_{2})}_{\beta},
  \end{equation}
  where, for $f\colon\R^{n\Lambda}\to\R$ differentiable,
  \begin{equation}
    \label{eq:gen-SRW-3}
    T^{\star}_{j}f(\V{u}) \bydef -T_{j}f(\V{u}) + (T_{j}H_{\beta}(\V{u}))f(\V{u}).
  \end{equation}
  
  We now compute the right-hand side of \eqref{eq:gen-SRW-3}.
  To simplify notation, let $x_{i}\bydef u^{1}_{i}$ and $\V{x} \bydef (x_{i})_{i\in\Lambda}$.
  By \eqref{eq:GFF-H}, \eqref{eq:Lact}, and using that $T_j$ is the derivative in the $x$-component,
  \begin{equation}
    \label{eq:gen-SRW-4}
    T_{j}H_{\beta}(\V{u})
    =\frac12 T_j \sum_{i\in\Lambda} u_i \cdot (-\Delta u)_i
    =  (-\Delta_{\beta}\V{x})_{j},
  \end{equation}
  so that for a function of the form $f(\frac{1}{2}|\V{u}|^{2})$,
  \begin{equation}
    \label{eq:gen-SRW-5}
    -T^{\star}_{j}f(\frac{1}{2}|\V{u}|^{2}) =
    (\Delta_{\beta}\V{x})_{j}f(\frac{1}{2}|\V{u}|^{2}) +
    x_{j}\ddp{f(\frac{1}{2}|\V{u}|^{2})}{\ell_{j}},
  \end{equation}
  where the last term denotes a partial derivative with respect to the
  $j$th coordinate of the function $f$. By applying
  \eqref{eq:gen-SRW-5} to each of the functions
  $f(j,\frac{1}{2}|\V{u}|^{2})$ and using
  $(\V{f}_{1},\Delta_{\beta}\V{f}_{2})=(\Delta_{\beta}\V{f}_{1},\V{f}_{2})$,
  \begin{equation}
    \label{eq:gen-SRW-6}
    -\sum_{j\in\Lambda}T^{\star}_{j}f(j,\frac{1}{2}|\V{u}|^{2}) 
    =
    \sum_{j\in\Lambda}x_{j}
    \cb{ \Delta_{\beta}f(j,\frac{1}{2}|\V{u}|^{2}) +
      \ddp{f(j,\frac{1}{2}|\V{u}|^{2})}{\ell_{j}}}
    =
    \sum_{j\in\Lambda}x_{j}(\cL f)(j,\frac{1}{2}|\V{u}|^{2}).
  \end{equation}
  To verify \eqref{eq:gen-lem},
  multiply
  \eqref{eq:gen-SRW-6} by $\rho$ and use the result to rewrite
  the left-hand side of \eqref{eq:gen-lem}. The desired equation
  then follows by applying \eqref{eq:gen-SRW-2}:
  \begin{equation*}
    -\sum_{j\in\Lambda}\cb{\rho x_{j} \cL f(j, \frac{1}{2} |\V{u}|^{2})}_{\beta}
    = 
    \sum_{j\in\Lambda}\cb{ \rho T^{\star}_{j}f(j,\frac{1}{2} |\V{u}|^{2})}_{\beta}
    = 
    \sum_{j\in\Lambda}\cb{ (T_{j}\rho)
      f(j,\frac{1}{2}|\V{u}|^{2})}_{\beta}.
  \end{equation*}

  We now prove \eqref{eq:generic-iso}; it suffices to consider $f$
  smooth with rapid decay.  Indeed, if $f_\epsilon$ is the convolution
  of $f$ with a smooth mollifier in the second argument, one has
  $f_\epsilon \to f$ pointwise and the $f_\epsilon$ are bounded
  uniformly in $\epsilon$ by a function with rapid decay, so by
  dominated convergence the result for $f$ follows from the result for
  the $f_\epsilon$.  Let
  $f_{t}(i,\V{\ell}) \bydef \E_{i,\V{\ell}} (f(X_{t},\V{L}_{t}))$, and
  note that $f_t$ is a smooth function with rapid decay since $f$ has
  this property (see below \eqref{e:vector-ell}).  Apply
  \eqref{eq:gen-lem} to $f_{t}$ and rewrite the left-hand side using
  Kolmogorov's backward equation, i.e.,
  $\cL f_{t} = \partial_{t}f_{t}$. The result is
  \begin{equation}
    \label{eq:cor-1}
    -\ddp{}{t}\sum_{j\in\Lambda}\uspin{\rho(\V{u}) u_{j}^1
      \E_{j,\frac{1}{2}|\V{u}|^2}(f(X_{t},\V{L}_{t}))}_{\beta}
    =
    \sum_{j\in \Lambda}\uspin{(T_j\rho)(\V{u})\, \E_{j,\frac{1}{2}|\V{u}|^{2}}f(X_{t},\V{L}_{t})}_{\beta}
    .
  \end{equation}
  To conclude, integrate \eqref{eq:cor-1} over $(0,\infty)$.
  The result follows since
  the boundary term at infinity on the
  left-hand side vanishes. To see this last claim, recall that the graph
  induced by $\beta$ is finite and connected, so
  $L^{i}_{t}\to\infty$ in probability for all vertices
  $i\in\Lambda$. When $f$ has sufficient decay this implies
  \begin{equation}
    \label{eq:van}
    \lim_{T\to\infty} \E_{j,\frac{1}{2}|\V{u}|^{2}} f(X_{T},\V{L}_{T})=0
  \end{equation}
  for all $\V{u}$. If $f$ has sufficient decay and $\rho$ has moderate
  growth then \eqref{eq:van} implies 
  \begin{equation}
    \label{eq:van2}
    \lim_{T\to\infty} [ \rho(\V{u}) \E_{j,\frac{1}{2} | \V{u}|^{2}} f(X_T,\V{L}_T)]_{\beta} = 0
  \end{equation}
  by dominated convergence, as desired.
  This completes the proof of \eqref{eq:generic-iso}.
\end{proof}

Our proofs of the classical isomorphism theorems will apply
Lemma~\ref{lem:gen-SRW} with
the following choices of $\rho$ and $f$;
further details will be given in the proofs.
\begin{itemize}
\item BFS--Dynkin isomorphism: $\rho(\V{u}) = u_a$ and
  $f(j,\V{\ell}) = g(\V{\ell})1_{j=b}$ with $a,b\in\Lambda$;
\item Ray--Knight isomorphism:
  $T_a\rho(\V{u}) \to \delta(u_a)-\delta(u_a-s)$ and
  $f(j,\V{\ell}) \to g(\V{\ell})\delta(\ell_a -\frac{s^2}{2})1_{j=a}$;
\item Eisenbaum isomorphism:
  $\rho(\V{u}) =  \exp (s(\V{h},\V{u}) -  \frac{s^2}{2}(\V{h},\V{1}) )$ 
  and
  $f(j,\V{\ell}) = g(\V{\ell})e^{-(h,\V{\ell})} 1_{j=a}$.
\end{itemize}

\subsection{BFS--Dynkin isomorphism theorem}
\label{sec:bfs-dynk-isom}

We now prove the BFS--Dynkin isomorphism theorem.

\begin{theorem}
  \label{thm:eucl-dynkin}
  Let $[\cdot]_\beta$ be the unnormalised expectation of the $\R^n$ model, and let $\E_{i,\V{\ell}}$ be the expectation of the SRW.
  Let $g\colon \R^\Lambda \to \R$ have rapid decay, and
  let $a,b\in\Lambda$. Then:
  \begin{equation}
    \label{eq:BFSD-g}
    \cb{u_{a}^{1}u_{b}^{1} \,g(\frac{1}{2}|\V{u}|^{2})}_{\beta}
    =
    \cb{\int_0^\infty \E_{a,\frac{1}{2}|\V{u}|^{2}} \pa{g(\V{L}_t)}1_{X_t=b}\, dt }_{\beta}.
  \end{equation}
\end{theorem}
\begin{proof}
  Apply Lemma~\ref{lem:gen-SRW} with $\rho(\V{u}) = u_a^1$,
  $f(j,\V{\ell}) = g(\V{\ell})1_{j=b}$, and use
  $T_j\rho(\V{u}) = 1_{j=a}$. 
\end{proof}

If $\V{h} \neq \V{0}$, after replacing $g(\V{\ell})$ by
$g(\V{\ell})e^{-(\V{h},\V{\ell})}$ in \eqref{eq:BFSD-g} the
unnormalised expectation can be normalised using
\eqref{eq:GFF-n}. Since
$\E_{a,\V{\ell}}(g(\V{L}_t)) = \E_a(g(\V{L}_t+\V{\ell}))$ for the simple
random walk, we immediately obtain Dynkin's formulation of this
theorem as stated, e.g., in~\cite[Theorem~2.8]{MR2932978}.

\begin{corollary}
  \label{cor:BFSD-n}
  Let $\avg{\cdot}_\beta$ be the expectation of the $\R^n$ model, and
  let $\E_{i,\V{\ell}}$ be the expectation of the SRW. Let
  $g\colon \R^{\Lambda}\to \R$ be bounded, 
  $a,b\in\Lambda$, and suppose $\V{h}\neq \V{0}$. Then
  \begin{equation}
    \label{eq:BFSD-c}
    \avga{u_{a}^{1}u_{b}^{1}\, g(\frac{1}{2}|\V{u}|^2)}_{\beta,h} =
    \avga{\int_0^\infty \E_{a} \pa{g(\V{L}_t +
        \frac{1}{2}|\V{u}|^2) e^{-(\V{h},\V{L}_{t})} 
        1_{X_{t}={ b}}
        } \, dt }_{\beta,h}.
  \end{equation}
\end{corollary}

We have rebranded this the BFS--Dynkin isomorphism because a version
of Corollary~\ref{cor:BFSD-n} first appeared in the work of Brydges,
Fr\"{o}hlich, and Spencer~\cite[Theorem~2.2]{MR648362}.

\subsection{Ray--Knight isomorphism}
\label{sec:ray-knight-identity}

The Ray--Knight isomorphism (i.e., the generalised second Ray--Knight
theorem) is also a quick consequence of Lemma~\ref{lem:gen-SRW}.
Several other proofs of this identity exist for the $1$-component GFF,
see \cite{MR1813843,MR3520013} and references therein. 
For an explanation of the name, see~\cite[Remark~2.19]{MR2932978}.

We introduce the following notation for translations to emphasise the
analogy between the classical Ray--Knight isomorphism and its hyperbolic and
spherical versions.  Let $\theta_{s}$ be the translation of all
coordinates by $s\in \R$ in the direction
$e^{1} = (1,0,\dots, 0)\in \R^n$, i.e.,
$\theta_s f(\V{u}) \bydef f(\V{u}+s \V{e}^1)$ for
$\V{e}^1= (e^1, \dots, e^1) \in \R^{n\Lambda}$.  In particular,
$\theta_{s}\V{u} = \V{u} + s \V{e}^{1}$. Note that
$\theta_{s}$ is the group action associated to the diagonal
translation symmetry, which has
infinitesimal generator $\sum_{j\in\Lambda}T_{j}$.

We will write
\begin{equation}
  [\delta_{u_0}(u_a)F]_\beta
\end{equation}
for the expectation of the spin model in which the spin at vertex $a$
is fixed to $u_0=(0,\dots, 0) \in \R^n$.

\begin{theorem}
  \label{thm:eucl-rk}
  Let $[\cdot]_\beta$ be the unnormalised expectation of the $\R^n$ model, and let $\E_{i,\V{\ell}}$ be the expectation of the SRW.
  Let $g\colon\R^{\Lambda}\to\R$ be a smooth compactly supported function, let
  $a\in\Lambda$, and let $s\in\R$.  Then
  \begin{equation}
    \uspin{ g(\frac{1}{2}|\theta_s\V{u}|^2)\delta_{u_0}(u_a)}_\beta 
    = \uspin{  \E_{a,\frac{1}{2}|\V{u}|^2}g(\V{L}_{\tau(\frac{s^2}{2})})\delta_{u_0}(u_a)}_\beta
  \end{equation}
  where $\tau(\gamma) \bydef \inf \{t \,| \,L_a^t \ge \gamma\}$ and $u_0 = (0,\dots,0) \in \R^{n}$.
\end{theorem}

\begin{proof}[Proof of Theorem~\ref{thm:eucl-rk}]
  Since the identity is trivial if $s=0$, assume $s\neq 0$.  The proof
  is by applying Lemma~\ref{lem:gen-SRW} with
  $\rho_{\epsilon}(\V{u}) \equiv \rho_{\epsilon}(u_{a})$,
  $f(j,\V{\ell}) \bydef g(\V{\ell})\eta_{\epsilon}(\ell_a) 1_{j=a}$, and the functions
  $\rho_{\epsilon} \colon \R^n \to \R$ and $\eta_{\epsilon} \colon\R\to\R$ chosen such that
  $T_{a}\rho_{\epsilon}$ and $\eta_{\epsilon}$ are smooth compactly supported approximations
  to $\delta_{ u_0}-\delta_{\theta_s u_0}$ and
  $\delta_{\frac{1}{2}s^{2}}$ subject to
  $\rho_{\epsilon}(v)\eta_{\epsilon}(\frac{1}{2}|v|^{2})=0$ for all $v\in \R^n$. Explicitly,
  with $\delta^{(k)}_{\tilde u, \epsilon}(x)$ denoting a smooth
  approximation to a delta function at $\tilde u\in \R^{k}$ with
  support in $|x|<\varepsilon/2$, we may take
  \begin{equation}
    \rho_{\epsilon}(u_a) = \int_0^{s-\varepsilon}\delta_{u_{0},\varepsilon}^{(n)}(\theta_{-r}u_a)\, dr,
    \quad \eta_{\epsilon}(\ell) =
    \delta_{0,\varepsilon}^{(1)}\left(\ell- \frac12 s^2 - \frac{\varepsilon}{2}\right).
  \end{equation}
  By Lemma~\ref{lem:gen-SRW}, since $\rho_{\epsilon}(u_a)\eta_{\epsilon}(\frac{1}{2}|u_a|^2) = 0$,
  \begin{equation}
    \label{eq:tlim}
    \uspin{T_a\rho_{\varepsilon}(u_a)\int_0^\infty \E_{a,\frac{1}{2}|\V{u}|^2}(g(\V{L}_t)\eta_\varepsilon(L_t^a)1_{X_t=a}) \, dt}_{\beta}
    =
    \uspin{\rho_\varepsilon(u_a)u_{a}^{1} g( \frac{1}{2}
      |\V{u}|^{2})\eta_{\varepsilon} 
      (\frac{1}{2}|u_a|^2)}_{\beta} = 0.
  \end{equation}

  Let $dL^a=1_{X_t=a} \, dt$.
  By the continuity\footnote{
    To see continuity, since $g$ is compactly supported, it suffices to show that for a sufficiently large $T$,
    $s\mapsto \E_{a,\V{\ell}} g(\V{L}_{\tau(\frac12 s^2) \wedge T})$ is continuous.
    Since $g$ is Lipschitz, it suffices to show
    $\E_{a,\V{\ell}}|\V{L}_{\tau(\frac12 s^2-\delta)\wedge
      T}-\V{L}_{\tau(\frac12 s^2 + \delta)\wedge T}|_{1} \to 0$ as
    $\delta \to 0$, $|\cdot|_{1}$ the $1$-norm.
    Let $J_{\delta}$ be the event that a jump occurs in the interval $[\frac12
    s^2-\delta, \frac12 s^{2} + \delta]$. Then
    \begin{equation*}
      \E_{a,\V{\ell}}|\V{L}_{\tau(\frac12 s^2-\delta)\wedge T}-\V{L}_{\tau(\frac12 s^2 + \delta)\wedge T}|_1
      \leq \delta + T \P_{a,\ell}(J_{\delta}) = O_{T}(\delta).
    \end{equation*}
  }
  of $s\mapsto \E_{a,\V{\ell}} g(\V{L}_{\tau(\frac12 s^2)})$
  and the definition of $\eta_\epsilon$,
    \begin{align}
    \label{eq:rklimpre}
    \lim_{\epsilon\to 0}
    \E_{a,\V{\ell}}\int_{0}^{\infty}  g(\V{L}_{t})
    \eta_{\epsilon}(L^{a}_{t})1_{X_{t}=a} \, dt
    &=
    \lim_{\epsilon\to 0} 
    \E_{a,\V{\ell}}\int_{0}^{\infty}  g(\V{L}_{ \tau(L^{a})})
      \eta_{\epsilon}(L^{a}) \, dL^{a}
      \nnb
    &=
    \lim_{\epsilon\to 0} 
    \int_{0}^{\infty} \E_{a,\V{\ell}}( g(\V{L}_{\tau(\gamma)}))
    \eta_{\epsilon}(\gamma) \, d\gamma
    = \E_{a,\V{\ell}}g(\V{L}_{\tau(\frac12 s^2)}),
  \end{align}
  uniformly in $\V{\ell}$ with $\ell_a \leq \frac12 s^2$.
  Since $T_a\rho_{\epsilon}(u_a) = \delta_{u_0,\varepsilon}^{(n)}(u_a) -
  \delta_{u_0\varepsilon}^{(n)}(\theta_{-(s-\varepsilon)}u_a)$,
  taking the
  limit $\varepsilon \rightarrow 0$  in \eqref{eq:tlim}
  yields, by \eqref{eq:rklimpre},
  \begin{equation}
    \uspin{ \E_{a,\frac{1}{2}|\V{u}|^2}(g(\V{L}_{\tau(\frac{s^2}{2})}))\delta_{u_0}(u_a)}_{\beta}
    =
    \uspin{
      \E_{a,\frac{1}{2}|\theta_s\V{u}|^2}(g(\V{L}_{\tau(\frac{s^2}{2})}))\delta_{u_0}(u_a)}_{\beta}
  \end{equation}
  where we have used the invariance of $[\cdot]_\beta$ under
  $\theta_s$, i.e., $[F]_{\beta} = [\theta_{s}F]_{\beta}$.
  To conclude, observe
  \begin{equation}
   \uspin{
      \E_{a,\frac{1}{2}|\theta_s\V{u}|^2}(g(\V{L}_{\tau(\frac{s^2}{2})}))\delta_{u_0}(u_a)}_{\beta} = \uspin{ g(\frac{1}{2}|\theta_s\V{u}|^2)\delta_{u_0}(u_a)}_\beta
  \end{equation}
  since $\tau(\frac{1}{2}s^{2})=0$ if $L_0^a=\frac{s^2}{2}$.
\end{proof}

\subsection{Eisenbaum isomorphism theorem}
\label{sec:eisenb-isom-theor}

The Eisenbaum isomorphism theorem involves a continuous-time random
walk with killing.  Thus let $X_{t}$ be a killed random walk with
killing rates $\V{h}$, and let $\V{L}_{t}$ be its local times.  To be
precise, the generator of the joint process
$(X_{t},\V{L}_{t})_{t\geq 0}$ is given by
\begin{equation}
  \label{eq:SRW-gen-k}
  (\cL^{h} f)(i,\V{\ell})
  \bydef \cL
  f(i,\V{\ell}) 
  - h_{i}f(i,\V{\ell}) 
  ,\quad \text{i.e.,}\quad
  \cL^h = \cL - \V{h}.
\end{equation}
for $f\colon \Lambda\times\R^{\Lambda}\to \R$ smooth.  We let
$\E^{h}_{i,\V{\ell}}$ denote the
corresponding (deficient) expectation, i.e., integration with respect to the
density of the killed random walk, which may have measure less than $1$.
Note that the killing does not depend on the initial local times, i.e.,
\begin{equation}
  \label{eq:SRW-gen-k-1}
  \E^{h}_{i,\V{\ell}} \pb{g(X_{t},\V{L}_{t})} 
  = 
  \E_{i,\V{\ell}}\pb{g(X_{t},\V{L}_{t})e^{-\sum_{j\in\Lambda}h_{j}(L^{j}_{t}-\ell_j)}},
\end{equation}
and we can hence write
\begin{equation}
  \label{eq:SRW-nokill}
  \E_{i,\V{\ell}}(g(X_t,\V{L}_t)e^{-\sum_{j\in\Lambda} h_jL_t^j}) =
  \E_{i,\V{\ell}}^h(g(X_t,\V{L}_t))e^{-\sum_{j\in\Lambda} h_j
    \ell_j}.
\end{equation}

Probabilistically, the deficient law can be realised as a Markov
process with state space
$(\Lambda\cup \{\cemetery\})\times \R^{\Lambda\cup\{\cemetery\}}$,
where $\cemetery \notin\Lambda$ is an absorbing `cemetery' state.  The
walk jumps from $i$ to $\cemetery$ with rate $h_{i}$. The generator
acts on functions that are identically zero at $\cemetery$, and we
identify such functions with functions on
$\Lambda\times \R^{\Lambda}$.  We denote the time of the one and only
jump to $\cemetery$ by $\death$.

The following theorem is a version of Eisenbaum's isomorphism~\cite{MR1459468}.
\begin{theorem}
  \label{thm:eucl-eis}
  Suppose $\V{h}\neq \V{0}$. Let $\spin{\cdot}_{\beta,h}$ be the
  expectation of the $\R^n$ model, and let $\E_{i,\V{\ell}}^h$ be the
  expectation of the killed SRW.  Let $g\colon\R^{\Lambda}\to \R$ have
  moderate growth, let $a\in\Lambda$, and let $s\in\R$. Then
  \begin{equation}
    \label{eq:EIS-g}
    \spin{(\theta_{s}u_{a}^{1})g(\frac{1}{2}|\theta_{s}\V{u}|^{2})}_{\beta,h} 
    =
    s\sum_{i\in\Lambda}h_{i}\spin{ \int_{0}^{\infty}
      \E^{h}_{i,\frac{1}{2}|\theta_{s}\V{u}|^{2}}(g(\V{L}_{t})1_{X_{t}=a}) \, dt}_{\beta,h}.
  \end{equation}
\end{theorem}
\begin{proof}
 We apply Lemma~\ref{lem:gen-SRW} with
  \begin{align}
    \label{eq:eis-rho}
      \rho(\V{u}) 
      &\bydef e^{s(\V{h},\V{u}) - \frac{s^2}{2}(\V{h},\V{1})} =
        e^{\frac{1}{2}(\V{h},|\V{u}|^2)}
        (e^{-\frac{1}{2}(\V{h},|\theta_{-s}\V{u}|^2)}), \\
      f(j,\V{\ell}) &\bydef
      g(\V{\ell})e^{-(\V{h},\V{\ell})} 1_{j=a}.
  \end{align}
  While $\rho$ does not have moderate growth in the sense of our
  conventions, the very rapid (Gaussian) decay of $f$ is sufficient
  for the lemma to hold. We then
  use that $(T_j\rho)(\V{u}) = sh_j \rho(\V{u})$ to obtain
  \begin{align}
    \label{eq:flat-eis-1}
      s\sum_{j\in\Lambda} h_j\uspin{\rho(\V{u})\int_0^\infty
        \E_{j,\frac{1}{2}|\V{u}|^2}(
        g(\V{L}_t)1_{X_{t}=a}e^{-(\V{h},\V{L}_t)})\,dt}_{\beta} 
      &=
      \sum_{j\in\Lambda}\uspin{\rho(\V{u})u_{j}^{1} 
        g(\frac{1}{2} |\V{u}|^{2}) 1_{j=a}
        e^{-\frac{1}{2}(\V{h},|\V{u}|^2)}}_{\beta}
      \nnb
      &= 
      \uspin{u_{a}^{1} 
        g(\frac{1}{2} |\V{u}|^{2})e^{-\frac{1}{2}(\V{h},|\theta_{-s}\V{u}|^2)}}_{\beta},
  \end{align}
  by inserting the definition \eqref{eq:eis-rho}.
  Using \eqref{eq:SRW-nokill} to substitute
  \begin{equation}
    \rho(\V{u})\E_{j,\frac{1}{2}|\V{u}|^2}(g(\V{L}_t)e^{(-\V{h},\V{L}_t)}) 
    = \E_{j,\frac{1}{2}|\V{u}|^2}^h(g(\V{L}_t))e^{-\frac{1}{2}(\V{h},|\theta_{-s}\V{u}|^2)}
  \end{equation}
  and by the translation invariance of
  $\cb{\cdot}_{\beta}$, i.e., 
  $[\theta_sF]_\beta = [F]_\beta$, we can rewrite
  \eqref{eq:flat-eis-1} as
  \begin{equation}
    s\sum_{j\in\Lambda}h_j\uspin{ \ob{\int_0^\infty
      \E_{j,\frac{1}{2}|\theta_s\V{u}|^2}^h(g(\V{L}_t)1_{X_{t}=a})\,dt}e^{-\frac{1}{2}(\V{h},|\V{u}|^2)}
    }_{\beta} 
    = 
    \uspin{(\theta_s u_a^1) g(\frac{1}{2}|\theta_s \V{u}|^2) e^{-\frac{1}{2}(\V{h},|\V{u}|^2)}}_{\beta}.
  \end{equation}
  This can be re-written in terms of $\cb{\cdot}_{\beta,h}$ as 
    \begin{equation}
    s\sum_{j\in\Lambda}h_j\uspin{ \int_0^\infty
      \E_{j,\frac{1}{2}|\theta_s\V{u}|^2}^h(g(\V{L}_t)1_{X_{t}=a})\,dt}_{\beta,h} 
    = 
    \uspin{(\theta_s u_a^1) g(\frac{1}{2}|\theta_s \V{u}|^2)}_{\beta,h},
  \end{equation} 
  and normalising gives \eqref{eq:EIS-g}.
\end{proof} 

We will now derive the usual formulation of the Eisenbaum isomorphism
as a corollary. For notational simplicity, suppose $n=1$, and let
$u_{i}=u_{i}^{1}$.  Writing the translations explicitly,
Theorem~\ref{thm:eucl-eis} yields, for
$\V{s}=(s,s,\dots, s)\in\R^{\Lambda}$, $s\neq 0$,
\begin{align}
  \spin{\frac{u_{a}+s}{s}g( \frac{1}{2}|\V{u}+\V{s}|^{2})}_{\beta,h} 
  &=
  \sum_{i\in\Lambda}h_{i}\spin{\E^{h}_{i,\frac{1}{2}|\V{u}+\V{s}|^{2}} \int_{0}^{\infty} g(\V{L}_{t}) 
  1_{X_{t}=a}\, dt}_{\beta,h}\nnb
  &=
  \sum_{i\in\Lambda}h_{i}\spin{\E_{i} \int_{0}^{\infty}g(\frac{1}{2}|\V{u}+\V{s}|^{2}+\V{L}_{t}) 
  1_{X_{t}=a}e^{-\sum_{j\in\Lambda}h_{j}L_{t}^{j}}\, dt}_{\beta,h} \nnb
  &=
  \sum_{i\in\Lambda}h_{i}\spin{\E_{a} \int_{0}^{\infty}g(\frac{1}{2}|\V{u}+\V{s}|^{2}+\V{L}_{t}) 
  1_{X_{t}=i}e^{-\sum_{j\in\Lambda}h_{j}L_{t}^{j}}\, dt}_{\beta,h}
\end{align}
where in the last line we have used the reversibility of the killed random
walk. Bringing the sum inside the Gaussian expectation, we recognise
the conditional density that $X$ jumps from $i$ to $\cemetery$ at time $t$,
proving the following corollary.
Recall $\death$ is the time of the jump to the cemetery state.

\begin{corollary}
  \label{cor:EIS-c}
  Suppose $\V{h}\neq \V{0}$. Let $\spin{\cdot}_{\beta,h}$ be the
  expectation of the $\R^n$ model, and let $\E_{i,\V{\ell}}^h$ be the
  expectation of the killed SRW.  Suppose $g\colon\R^{\Lambda}\to\R$
  has moderate growth, $a\in\Lambda$, and
  $\V{s}=(s,s,\dots, s) \in \R^\Lambda$ with $s\neq 0$. Then
  \begin{equation}
    \label{eq:EIS-c}
    \avga{\frac{u_{a}+s}{s}g( \frac{1}{2}|\V{u}+\V{s}|^{2})}_{\beta,h} = \avga{
      \E_{a}^{h} \pb{ g( \frac{1}{2}|\V{u}+\V{s}|^{2} + \V{L}_{\death})}}_{\beta,h}.
  \end{equation}
\end{corollary}

\section{Isomorphism theorems for hyperbolic geometry}
\label{sec:hyperbolic}

In this section we describe spin models with hyperbolic symmetry, the
associated vertex-reinforced jump processes, and isomorphism theorems
that link these objects. The proofs follow closely those of
Section~\ref{sec:euclidean}, but with the translation symmetry of
$\R^{n}$ replaced by the boost symmetry of $\HH^{n}$.

\subsection{The vertex-reinforced jump process}
\label{sec:vert-reinf-jump}

The \emph{vertex-reinforced jump process} (VRJP) $X_{t}$ with
\emph{initial local time $\V{L}_{0}\in (0,\infty)^{\Lambda}$} and
\emph{initial vertex $v\in\Lambda$} is the process $X_{t}$ with
$X_{0}=v$ and jump rates
\begin{equation}
  \label{eq:VRJP}
  \P_{v,\V{L}_{0}}\cb{X_{t+dt}=j \mid (X_{s})_{s\leq t}, X_{t}=i} =
  \beta_{ij}L^{j}_{t}\,dt,
\end{equation}
where the local times $\V{L}_{t}$ of $X_{t}$ are defined as in
\eqref{eq:local}.  Note that \eqref{eq:RW-models} with $\epsilon=1$ is
the special case of \eqref{eq:VRJP} in which $\V{L}_{0}=\V{1}$.  The
construction of a VRJP with given initial local times is
straightforward, see~\cite[Section~2]{MR2027294}.
Our assumption that the graph induced by the edge weights $\beta$ is connected
implies that $L^{j}_{t}\to\infty$ as $t\to\infty$ in probability for
all $j$ and all sets of initial local times, see
\cite[Lemma~1]{MR2027294}.

As in Section~\ref{sec:euclidean}, it will be helpful to view $X_{t}$
as the marginal of the process $(X_{t},\V{L}_{t})$ that includes the
local times $\V{L}_{t}$. For convenience we will also call this joint
process a VRJP. Unlike $X_{t}$, the joint process $(X_{t},\V{L}_{t})$
is a Markov process.  The generator $\cL$ of the joint process
acts on smooth functions $g\colon\Lambda\times\R^{\Lambda}\to\R$ by
\begin{equation}
  \label{eq:VRJP-generator}
  (\cL g)(i,\V{\ell}) = \sum_{j\in\Lambda}\beta_{ij}\ell_{j}
  (g(j,\V{\ell}) - g(i,\V{\ell})) 
  + \ddp{g(i,\V{\ell})}{\ell_{i}}.
\end{equation}
We note that $g_t(i,\V{\ell})=\E_{i,\V{\ell}}\,g(X_t,\V{L}_t)$ is smooth in $\ell$
for any $t>0$ if $g$ is smooth.  This can be seen, for example, from
the explicit construction of the VRJP in \cite[Section~2]{MR2027294}.

\subsection{Hyperbolic symmetry}
\label{sec:hyperb-spin-model}

The VRJP will be seen to be closely related with hyperbolic symmetry,
i.e., the Lorentz group $O(n,1)$.  In this subsection we discuss the
relevant aspects of this group and its action on Minkowski and
hyperbolic space.

\begin{figure}
\begin{center}
  \input{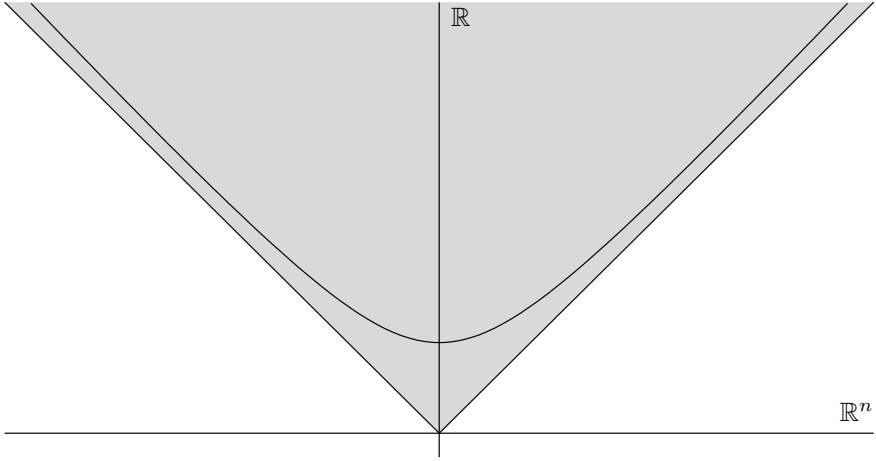}
\end{center}
\caption{Minkowski space $\R^{n,1}$. The shaded area is the causal future 
  and the hyperboloid is $\HH^n$.\label{fig:minkowski}}
\end{figure}

\paragraph{Minkowski space}

\emph{Minkowski space $\R^{n,1}$} is the vector space
$\R^{n+1}$ equipped with the indefinite \emph{Minkowski inner product}
\begin{equation}
  \label{eq:Mip}
  u_1 \cdot u_2
  \bydef
  - u_1^{0} u_2^{0}  + \sum_{\alpha=1}^n u_1^\alpha u_2^\alpha,
\end{equation}
where each
$u_{i} = (u^{0}_{i},u^1_{i}, \dots, u_{i}^{n})\in \R^{n,1}$.  The
points $u \in \R^{n,1}$ with $u\cdot u < 0$ are called
\emph{time-like}. The set of time-like vectors with $u^0>0$ is called
the \emph{causal future}; schematically this is the shaded area in
Figure~\ref{fig:minkowski}.  In what follows, for $u\in \R^{n,1}$ it
will be notationally convenient to write $z=u^{0}$ and $x=u^1$.

The group preserving the quadratic form $u\cdot u$ given by
\eqref{eq:Mip} is the \emph{Lorentz group} $O(n,1)$.  The
\emph{restricted Lorentz group} $SO^+(n,1)$ is the subgroup of
$T \in O(n,1)$ with $\det T = 1$ and $T_{00} >0$.  $SO^{+}(n,1)$
preserves the causal future, see Figure~\ref{fig:minkowski}.  The
elements of $SO^+(n,1)$ can be written as compositions of
rotations and boosts. We briefly review the aspects of
these transformations needed for what follows. Rotations act on the
coordinates $u^{1},\dots, u^{n}$ exactly as in Euclidean space, while
a boost $\theta_{s}$ by $s\in\R$ in the $xz$-plane acts by
\begin{equation}
  \label{eq:boost}
  \theta_{s}z = x\sinh s + z \cosh s,
  \quad
  \theta_{s}x = x\cosh s + z \sinh s,
  \quad
  \theta_{s}u^{\alpha} = u^{\alpha}, \quad (\alpha=2,\dots, n),
\end{equation}
and similarly for boosts in other planes. From \eqref{eq:boost} it
follows that the infinitesimal generator $T$ of boosts in the
$xz$-plane is the linear differential operator satisfying
\begin{equation}
  \label{eq:boost-inf}
  T z = x, \quad T x = z, \quad T u^{\alpha} =0, \quad (\alpha=2,\dots, n),
\end{equation}
i.e.,
\begin{equation}
  T \bydef z \ddp{}{x}+ x\ddp{}{z}.
\end{equation}

\paragraph{Hyperbolic space}

When given the metric induced by the Minkowski inner product,
the set
\begin{equation}
  \label{eq:hn-def}
  \mathbb{H}^n \bydef \{ u \in \R^{n,1}
  \mid u \cdot u = -1, z >0\}
\end{equation}
is a model for $n$-dimensional hyperbolic space.  Note that
\eqref{eq:hn-def} implies $z\geq 1$. For $u,v\in\HH^{n}$,
$-u\cdot v = \cosh (d(u,v))$, where $d(u,v)$ is the geodesic distance
from $u$ to $v$. In particular, $-u\cdot v \geq 1$. For details on why
this is indeed hyperbolic space see, e.g.~\cite{MR1491098}.

$\HH^n$ is the orbit under $SO^+(n,1)$ of the point
 $u_0 = (1,0,\dots,0)$, and the stabiliser of $u_0$ is the subgroup
$SO(n)$.  Thus $\HH^{n}$ can be identified with $SO^{+}(n,1)/ SO(n)$.
It is parameterised by $(u^{1},\dots, u^{n}) \in \R^n$:
\begin{equation}
  \label{eq:hyp-coord}
  \HH^n = \{ u\in \R^{n,1}
  \mid (u^{1},\dots, u^{n})\in\R^{n},
  z=\sqrt{1+(u^{1})^{2} + \dots (u^{n})^{2}}\}.
\end{equation}
In these coordinates, the $SO^+(n,1)$-invariant Haar measure on
$\HH^n$ can be written as
\begin{equation}
  \label{eq:hn-meas}
  du = \frac{du^{1}\dots du^{n}}{z(u)},
  \quad z(u) \bydef \sqrt{1+(u^1)^2 + \cdots + (u^n)^2}.
\end{equation}

Note that the Lorentz boost \eqref{eq:boost} maps $\HH^{n}$ to
$\HH^{n}$, and that in the parameterization of $\HH^{n}$ by
$(u^1, \dots, u^n)$, the infinitesimal Lorentz boost in the $xz$-plane
is given by
\begin{equation}
  \label{eq:gen-hyp-1}
  T \bydef z\ddp{}{x}.
\end{equation}
This is because $T$ satisfies the defining
equations~\eqref{eq:boost-inf}:
$Tz=x$, $Tx = z$, and $Tu^{\alpha}=0$ for $\alpha\geq 2$.
In the last calculation we have used the
definition \eqref{eq:hyp-coord} of $z(u)$.  The invariance of the
measure $du$ under Lorentz boosts implies that for differentiable
$f\colon\HH^{n}\to \R$ with sufficient decay,
\begin{equation}
  \label{eq:IBP-hyp}
  \int_{\HH^{n}}Tf\,du = 0. 
\end{equation}

\subsection{Hyperbolic sigma model}

Hyperbolic spin models are analogues of the Gaussian free field
defined in terms of the Minkowski inner product instead of the
Euclidean inner product. While it is possible to define a spin model
associated to the entire causal future of Minkowski space, see
Figure~\ref{fig:minkowski}, for now we restrict ourselves to the
\emph{sigma model} version of this model in which spins are constrained
to lie in $\HH^n$.  We will later consider (the
supersymmetric version of) a spin model taking values in the
causal future in Section~\ref{sec:HK}.

In the $\HH^{n}$ sigma model there is a spin $u_i\in \HH^n$ for each
$i \in\Lambda$. We again let $\beta$ be a non-negative
collection of edge weights and $\V{h}\geq \V{0}$ be a collection of
non-negative vertex weights. For a spin configuration $\V{u}$ we consider the
energy
\begin{equation}
  \label{eq:H-Hn}
  H_{\beta}(\V{u}) \bydef \frac{1}{2} (\V{u},-\Delta_\beta \V{u})
  = \frac{1}{4}\sum_{i,j\in\Lambda}\beta_{ij} (u_{i}-u_{j})^{2},
  \quad
  H_{\beta,h}(\V{u}) = H_{\beta}(\V{u}) + (\V{h},\V{z}-\V{1}),
\end{equation}
analogous to \eqref{eq:GFF-H}, except that the inner product in
$(u_{i}-u_{j})^{2} = (u_{i}-u_{j})\cdot (u_{i}-u_{j})$ is now given by
the Minkowski inner product.
The mass term has also been replaced by the term $(\V{h},\V{z}-\V{1})$
since $z_{i}\geq 1$ for all $i$.

Note that $H_{\beta}(\V{u})$ is invariant under the diagonal action of
$SO^+(n,1)$, analogous to the invariance of \eqref{eq:GFF-H} by the
Euclidean group.  Moreover, since $u_i\cdot u_i=-1$, we have
$(u_{i}-u_{j})^{2} = -2-2u_{i}\cdot u_{j}$, we can thus rewrite
$H_{\beta}(\V{u})$ in terms of
$\tilde u \bydef (u^{1},\dots, u^{n})\in\R^{n}$ as
\begin{equation}
  \label{eq:H-Hn1}
  H_{\beta}(\V{u}) = -\frac{1}{2}
  \sum_{i,j\in\Lambda}\beta_{ij}\left(\sum_{\alpha=1}^{n}u_{i}^{\alpha}u_{j}^{\alpha}-z_{i}z_{j}+1\right),
\end{equation}
where we recall that $z_{i}=z_{i}(\tilde u_{i})$ is given by
\eqref{eq:hyp-coord}. Define an unnormalised expectation
$\cb{\cdot}_{\beta,h}$ on functions $F\colon \HH^{n\Lambda}\to \R$ by
\begin{equation}
  \label{eq:unex-hn}
  \cb{F}_{\beta,h} \bydef
  \int_{\HH^{n\Lambda}}F(\V{u})e^{-H_{\beta,h}(\V{u})}\,d\V{u} =
  \int_{\R^{n\Lambda}} F(\V{u})e^{-H_{\beta,h}(\V{u})}\,
  \prod_{i\in\Lambda}\frac{d\tilde u_{i}}{z(\tilde u_{i})},
\end{equation}
where $d\V{u}$ is the $\Lambda$-fold product of the invariant measure on
$\HH^{n}$. In the second equality we have written this integral using
the parametrization by $\R^{n}$ in~\eqref{eq:hn-meas}.  When $\V{h}=\V{0}$ we
set $\cb{\cdot}_{\beta}\bydef \cb{\cdot}_{\beta,h}$.

The \emph{$\HH^{n}$-model} is the probability measure on
$\HH^{n\Lambda}$ defined by the normalised expectation
\begin{equation}
  \label{eq:Hn-model}
  \avg{F}_{\beta,h} \bydef \frac{1}{Z_{\beta,h}}\cb{F}_{\beta,h},
  \qquad Z_{\beta,h} \bydef \cb{1}_{\beta,h}.
\end{equation}
Note that for \eqref{eq:Hn-model} to be well-defined we must have
$Z_{\beta,h}<\infty$.  This is the case if and only if $h_{i}>0$ for
some $i$ due to the invariance of $H_{\beta}(\V{u})$ under the
non-compact boost symmetry of $\HH^{n}$.

\begin{remark}
  \label{rem:Hn-history}
  This model was studied in~\cite{MR2104878} as a toy model
  for some aspects of random band matrices. See
  Remark~\ref{rem:H22-history} below for further details on this
  connection.
\end{remark}

\subsection{Fundamental integration by parts identity}

The statement of the following lemma is formally identical to that of
Lemma~\ref{lem:gen-SRW}.  However, the objects in its statement are
now hyperbolic versions: $\cL$ is the generator of the VRJP,
$[\cdot]_\beta$ is the unnormalised expectation from
\eqref{eq:unex-hn}, $T_j$ is the infinitesimal Lorentz boost in the
$xz$-plane in the $j$th coordinate specified by \eqref{eq:boost-inf},
and $\frac12 |\V{u}|^2$ is replaced by $\V{z}$.

\begin{lemma}
  \label{lem:gen-hyp}
  Let $[\cdot]_\beta$ be the unnormalised expectation of the $\HH^n$
  model, and let $\E_{i,\V{\ell}}$ be the expectation of the VRJP.
  Let $f\colon \Lambda\times \R^\Lambda \to \R$ be a smooth function
  with rapid decay, and let $\rho\colon \HH^{n\Lambda}\to \R$ be
  smooth with moderate growth.  Then:
  \begin{equation}
    \label{eq:gen-hyp}
    -\sum_{j\in\Lambda}\cb{\rho(\V{u})x_{j} \cL f(j,\V{z})}_{\beta}
    =
    \sum_{j\in\Lambda} \cb{(T_j\rho)(\V{u})f(j,\V{z})}_{\beta}
    .
  \end{equation}
  In particular, the following  integrated version holds for all
  $f\colon\Lambda\times\R^{\Lambda}\to\R$ with rapid decay:
  \begin{equation}
    \label{eq:generic-iso-hyp}
    \sum_{j\in\Lambda}\uspin{\rho(\V{u})x_{j} f(j, \V{z})}_{\beta}
    =
    \sum_{j\in\Lambda} \uspin{(T_j\rho)(\V{u})\int_0^\infty \E_{j,\V{z}}(f(X_t,\V{L}_t))\,dt}_{\beta}
    .
  \end{equation}
\end{lemma}

\begin{proof}
  The proof is again by integration by parts and closely follows that of
  Lemma~\ref{lem:gen-SRW}.
  Indeed, using that $[\cdot]_\beta$ has density $e^{-H_\beta}$ with respect to the Lorentz invariant measure on $\HH^{n\Lambda}$,
  the identity~\eqref{eq:IBP-hyp} implies that
  for $f_{1},f_{2}\colon \HH^{n\Lambda}\to \R$ smooth and with sufficient decay,
  \begin{equation}
    \label{eq:gen-hyp-4}
    \cb{(T_{i}f_{1})f_{2}}_{\beta} = \cb{f_{1}(T^{\star}_{i}f_{2})}_{\beta},
  \end{equation}
  where
  \begin{equation}
    \label{eq:gen-hyp-5}
    T^{\star}_{i}f(\V{u}) = -T_{i}f(\V{u}) + (T_{i}H_{\beta}(\V{u})) f(\V{u}).
  \end{equation}
  Using \eqref{eq:H-Hn1} and \eqref{eq:boost-inf} yields
  \begin{equation}
    \label{eq:gen-hyp-6}
    T_{i}H_{\beta}(\V{u}) 
    = -\frac{1}{2}\sum_{j,k\in\Lambda}\beta_{jk} T_{i}( x_{j}x_{k} - z_{j}z_{k}) 
    = \sum_{j\in\Lambda}\beta_{ij}(x_{i}z_{j} - x_{j}z_{i})
  \end{equation}
  and hence, using \eqref{eq:boost-inf} and the chain rule to compute
  $T_{i}f$,
  \begin{equation}
    \label{eq:gen-hyp-8}
    -T_{i}^{\star}f(\V{z})  =
    \sum_{j\in\Lambda}\beta_{ij}(x_{j}z_{i}-x_{i}z_{j})f(\V{z}) + x_{i}\ddp{f(\V{z})}{\ell_{i}}.
  \end{equation}
  Applying \eqref{eq:gen-hyp-8} to each function $f(i,\V{z})$ and
  summing over $i$ yields
  \begin{equation}
    \label{eq:gen-hyp-9}
    -\sum_{i\in\Lambda}T_{i}^{\star}f(i,\V{z}) 
    =
    \sum_{i\in\Lambda}x_{i}\pb{\sum_{j\in\Lambda}\beta_{ij}z_{j}(
      f(j,\V{z})-f(i,\V{z})) + \ddp{f(i,\V{z})}{\ell_{i}}} 
      = \sum_{i\in\Lambda}x_{i}(\cL f)(i,\V{z})
  \end{equation}
  by the formula \eqref{eq:VRJP-generator} for $\cL$. The remainder of
  the proof follows the proof of Lemma~\ref{lem:gen-SRW}.
\end{proof}

\subsection{Hyperbolic isomorphism theorems}

The following theorems are analogues of the BFS--Dynkin, Ray--Knight,
and Eisenbaum isomorphism theorems.  Their proofs are analogous to
those in Section~\ref{sec:euclidean}, using Lemma~\ref{lem:gen-hyp} in
place of Lemma~\ref{lem:gen-SRW}, and using hyperbolic versions of
$\rho$ and $f$.  We begin with the hyperbolic version of the
BFS--Dynkin isomorphism, i.e., Theorem~\ref{thm:eucl-dynkin}. It first
appeared in~\cite{1802.02077} and was proven there using horospherical
coordinates.  Here we give a more intrinsic proof that avoids
horospherical coordinates.

\begin{theorem}
  \label{thm:BHS}
  Let $[\cdot]_\beta$ be the unnormalised expectation of the $\HH^n$
  model, and let $\E_{i,\V{\ell}}$ be the expectation of the VRJP.
  Let $g \colon \R^{\Lambda}\to \R$ have rapid decay, and let
  $a, b\in\Lambda$. Then
  \begin{equation}
    \label{eq:BHS}
    \uspin{x_{a}x_{b}
      g(\V{z})}_{\beta} =
    \uspin{z_{a}\int_{0}^{\infty} \E_{a,\V{z}}(g(\V{L}_{t})1_{X_t = b})\,dt}_{\beta}.
  \end{equation}
\end{theorem}

\begin{proof}
  Apply Lemma~\ref{lem:gen-hyp} with $\rho(\V{u}) = x_a$,
  $f(j,\V{\ell}) = g(\V{\ell})1_{j=b}$, and use
  $T_j\rho(\V{u}) = 1_{j=a}z_j$.
\end{proof}

The next theorem is a hyperbolic version of the Ray--Knight
isomorphism, i.e., Theorem~\ref{thm:eucl-rk}.  Recall the definition
of a boost $\theta_{s}$ by $s\in\R$ in the $xz$-plane from
\eqref{eq:boost}. In what follows we let $\theta_{s}$ act diagonally
on $\V{u}\in \HH^{n\Lambda}$, and we write $\theta_{s}\V{z}$ to
denote the first component of $\theta_{s}\V{u}$.
We also write $\cb{f\delta_{u_{0}}(u_{a})}_{\beta}$ for the expectation of the
spin model in which the spin $u_{a}$ is fixed at $u_{0}\in\HH^{n}$.

\begin{theorem}
  \label{thm:hyp-rk}
  Let $[\cdot]_\beta$ be the unnormalised expectation of the $\HH^n$ model, and let $\E_{i,\V{\ell}}$ be the expectation of the VRJP.
  Let $g\colon\R^{\Lambda}\to\R$ be a smooth compactly supported function,
  let $a\in\Lambda$, and let $s\in\R$.  Then
  \begin{equation}
    \label{eq:hyp-rk}
    \uspin{ g(\theta_s\bm{z})\delta_{ u_0}(u_a)}_\beta = \uspin{  \E_{a,\V{z}}g(\bm{L}_{\tau(\cosh{s})})\delta_{u_0}(u_a)}_\beta
  \end{equation}
  where $\tau(\gamma) = \inf \{t \,| \,L_a^t \ge \gamma\}$ and
  $u_0=(1,0,\dots,0) \in \HH^n$.
\end{theorem}

\begin{proof}[ Proof of Theorem~\ref{thm:hyp-rk}]
  Since the identity is trivial if $s=0$, assume $s\neq 0$.  We begin
  by applying Lemma~\ref{lem:gen-hyp} with
  $\rho_{\epsilon}(\V{u}) = \rho_{\epsilon}(u_{a})$,
  $f(j,\V{\ell}) = g(\V{\ell})\eta_{\epsilon}(\ell_a)1_{j=a}$,
  with the functions $\rho_{\epsilon} \colon \HH^n \to \R$ and
  $\eta_{\epsilon} \colon\R\to\R$ chosen such that $T_{a}\rho_{\epsilon}$
  and $\eta_{\epsilon}$ are smooth compactly supported approximations to
  $\delta_{u_{0}}(u_a)-\delta_{\theta_{s}{u_{0}}}(u_a)$ and
  $\delta_{\cosh{s}}(\ell_a)$ subject to $\rho_{\epsilon}(u_a)\eta_{\epsilon}(z_a)=0$ for
  all $u_a\in \HH^n$.  Since $s\neq 0$, these conditions can be shown
  to be satisfiable by explicit construction. Exactly as in the proof
  of Theorem~\ref{thm:eucl-rk} this yields
  \begin{equation}
    \label{eq:RKhyp1a}
    \uspin{T_{a}\rho_{\epsilon}(u_{a})\int_{0}^{\infty}\E_{a,\V{z}}(g(\V{L}_{t})\eta_{\epsilon}(L_{t}^{a}){1_{X_t=a}})\,dt}_{\beta}=0,
  \end{equation}
  i.e.,
  \begin{equation}
 \label{eq:RKhyp1}
 \uspin{\delta_{\theta_{s-\varepsilon} u_0,\varepsilon}(u_0)\int_{0}^{\infty}\E_{a,\V{z}}(g(\V{L}_{t})\eta_{\epsilon}(L_{t}^{a}){ 1_{X_t=a}}\,dt}_{\beta}= \uspin{\delta_{u_0,\varepsilon}(u_0)\int_{0}^{\infty}\E_{a,\V{z}}(g(\V{L}_{t})\eta_{\epsilon}(L_{t}^{a}){ 1_{X_t=a}})\,dt}_{\beta}.
\end{equation}

As in \eqref{eq:rklimpre}, by the continuity\footnote{ Continuity
  can be proven by an argument similar to the one we gave for simple
  random walk near~\eqref{eq:rklimpre}: after restricting to times at
  most $T$ using compact support, the claim follow from the fact that
  $\P(J_\delta) = O_T(\delta)$ since the jump rates up to time $T$ are bounded
  by $O(T)$.} of $s\mapsto \E_{a,\V{\ell}} g(\V{L}_{\tau(\cosh s)})$
and the definition of $\eta_\epsilon$,
\begin{equation} \label{eq:RKhyp-delta}
  \lim_{\epsilon\to 0} 
  \E_{a,\V{\ell}}\int_{0}^{\infty}  g(\V{L}_{t})
  \eta_{\epsilon}(L^{a}_{t})1_{X_{t}=a} \, dt 
  =
  \lim_{\epsilon\to 0} 
  \int_{0}^{\infty} \E_{a,\V{\ell}}( g(\V{L}_{\tau(\gamma)})
  \eta_{\epsilon}(\gamma) \, d\gamma
  = \E_{a,\V{\ell}}g(\V{L}_{\tau(\cosh s)}),
\end{equation}
uniformly in $\V{\ell}$ with $\ell_a \leq \cosh s$.

To conclude, we use \eqref{eq:RKhyp-delta} to take $\epsilon\to 0$
in \eqref{eq:RKhyp1}. More precisely, we use that
$\delta_{\theta_{s}u_{0}}$ concentrates the $u_a$ integral at
$z_{a}=\cosh s$ on the left-hand side, and hence the time integral at $t = 0$.
By the boost invariance of $\uspin{\cdot}_{\beta}$, 
this term produces  the left-hand side of \eqref{eq:hyp-rk}: 
  \begin{equation}
    \uspin{\delta_{\theta_su_0}(u_a)\E_{a,\V{z}}(g(\V{L}_{\tau(\cosh s)}))}_{\beta}
    =
    \uspin{\delta_{u_0}(u_a) \E_{a,\theta_s\V{z}}(g(\V{L}_{\tau(\cosh s)}))}_{\beta}
    =
    \uspin{\delta_{u_0}(u_a) g(\theta_{s}\V{z})}_{\beta} 
    .
  \end{equation}
  Again by \eqref{eq:RKhyp-delta}, the
  $\delta_{u_{0}}$ on the right-hand side of \eqref{eq:RKhyp1} concentrates the time integral at
  $\tau(\cosh s)$, which gives the right-hand side of \eqref{eq:hyp-rk}.
\end{proof}	

Finally, we prove a hyperbolic version of the Eisenbaum isomorphism
theorem, i.e., Theorem~\ref{thm:eucl-eis}. This concerns a killed
VRJP. The generator of this killed process
$(X_{t},\V{L}_{t})_{t\geq 0}$ acts on smooth functions
$f\colon \Lambda\times\R^{\Lambda}\to \R$ as
\begin{equation}
\label{eq:VRJP-gen-k}
(\cL^{h} f)(i,\V{\ell})
\bydef \cL
f(i,\V{\ell}) - h_{i}f(i,\V{\ell}) 
,\quad \text{i.e.,}\quad
\cL^h = \cL - \V{h},
\end{equation}
where $\cL$ is now the generator of the VRJP and $h_i$ are the killing
rates. We let $\E^{h}_{i,\V{\ell}}$ denote the corresponding deficient
expectation.  As for the SRW, the killing does not depend on the initial
local times, i.e.,
\begin{equation}
  \label{eq:VRJP-gen-k-1}
  \E^{h}_{i,\V{\ell}} \pb{g(X_{t},\V{L}_{t})} 
  = 
  \E_{i,\V{\ell}}\pb{g(X_{t},\V{L}_{t})e^{-\sum_{j\in\Lambda}h_{j}(L^{j}_{t}-\ell_j)}},
\end{equation}
and we can thus write
\begin{equation}
  \label{eq:VRJP-nokill}
  \E_{i,\V{\ell}}(g(X_t,\V{L}_t)e^{-\sum_{j\in\Lambda} h_j(L_t^j-1)})
  = \E_{i,\V{\ell}}^h(g(X_t,\V{L}_t))e^{-\sum_{j\in\Lambda} h_j (\ell_j-1)}
  = \E_{i,\V{\ell}}^h(g(X_t,\V{L}_t))e^{-(\V{h}, \V{\ell}-\V{1})}
  .
\end{equation}

\begin{theorem}
  \label{thm:EIS-VRJP}
  Let $\spin{\cdot}_{\beta,h}$ be the expectation of the $\HH^n$
  model, and let $\E_{i,\V{\ell}}^h$ be the expectation of the killed
  VRJP with $\V{h} \neq \V{0}$.  Let
  $g\colon\Lambda\times\R^{\Lambda}\to \R$ be of moderate growth, and
  let $s\in\R$. Then
  \begin{equation}
    \label{eq:EIS-VRJP-g}
    \sum_{i\in\Lambda}\spin{(\theta_{s}x_{i})
      g(i,\theta_{s}\V{z})}_{\beta,h} =
    \sum_{i\in\Lambda}h_{i}\spin{(\theta_{s}x_{i}-x_{i})
      \int_{0}^{\infty} \E_{i,\theta_{s}\V{z}}^{h}
      (g(X_{t},\V{L}_{t})
      )\,dt}_{\beta,h}.
  \end{equation}
\end{theorem}

\begin{proof}
  Analogously to the proof of Theorem~\ref{thm:eucl-eis}, we apply
  Lemma~\ref{lem:gen-hyp} with
  \begin{align}
      \rho(\V{u}) 
      &\bydef e^{(\V{h},\V{z}-\theta_{-s}\V{z})} = e^{(\V{h},\V{z}-\V{1})}(e^{-(\V{h},\theta_{-s}\V{z}-\V{1})})\\
      f(j,\V{\ell}) &\bydef g(\V{\ell})e^{-(\V{h},\V{\ell}-\V{1})} 1_{j=a},
  \end{align}
  and use that $(T_j\rho)(\V{u}) = h_j(x_j - \theta_{-s}x_j) \rho(\V{u})$ to obtain
  \begin{multline}
    \label{eq:hyp-eis-1}
    \sum_{j\in\Lambda} h_j\uspin{(x_j - \theta_{-s}x_j) \rho(\V{u})\int_0^\infty
    \E_{j,\V{z}}(
    g(\V{L}_t)1_{X_{t}=a}e^{-(\V{h},\V{L_t}-\V{1})})\,dt}_{\beta}
    \\
    =
      \sum_{j\in\Lambda}\uspin{\rho(\V{u})x_{j} 
    g(\V{z})1_{j=a}e^{-(\V{h},\V{z}-\V{1})}}_{\beta}
    = 
      \uspin{x_{a}^{1} 
      g(\V{z})e^{-(\V{h},\theta_{-s}\V{z}-\V{1})}}_{\beta}.
  \end{multline}
  Using \eqref{eq:VRJP-nokill} to substitute
  \begin{equation}
    \rho(\V{u})\E_{j,\V{z}}(
    g(\V{L}_t)1_{X_{t}=a}e^{-(\V{h},\V{L}_t-\V{1})}) 
    = \E_{j,\V{z}}^h(
    g(\V{L}_t)1_{X_{t}=a})e^{-(\V{h},\theta_{-s}\V{z}-\V{1})},
  \end{equation}
  and the boost invariance of the spin expectation
  $[\theta_s\cdot]_\beta = [\cdot]_\beta$, we can rewrite
  \eqref{eq:hyp-eis-1} as
  \begin{equation}
    \sum_{j\in\Lambda}h_j\uspin{(\theta_sx_j - x_j) \int_0^\infty
      \E_{j,\theta_s\V{z}}^h(g(\V{L}_t)1_{X_{t}=a})\,dt}_{\beta,h} 
    = 
    \uspin{(\theta_s x_a) g(\theta_s \V{z})}_{\beta,h}
  \end{equation}
  where we have absorbed the magnetic
  terms $e^{-(\V{h},\V{z}-\V{1})}$ into the measures.
  Normalising gives \eqref{eq:EIS-VRJP-g}. 
\end{proof}

\section{Isomorphism theorems for spherical geometry}
\label{sec:spherical}

In this section we describe analogues of the theorems of
Sections~\ref{sec:euclidean} and~\ref{sec:hyperbolic} for spherical
geometry.

\begin{figure}
\begin{center}
  \input{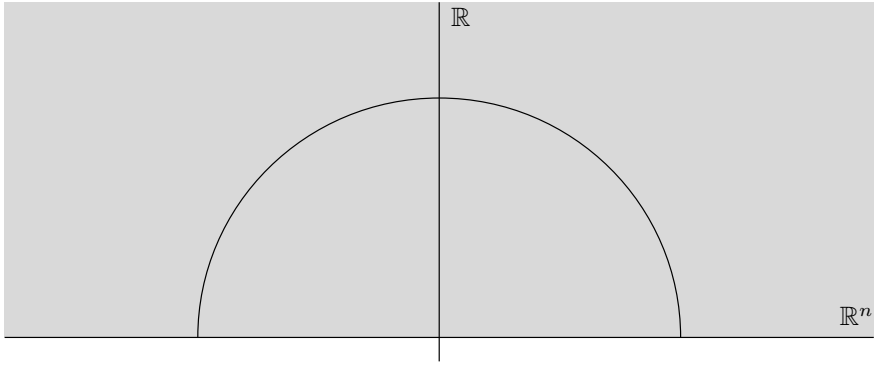}
\end{center}
\caption{The upper half-plane in Euclidean space
  $\R^{n+1}$ (shaded) and the upper hemisphere $\bbS^{n}_{+}$.}
\label{fig:spherical}
\end{figure}

\subsection{The vertex-diminished jump process}
\label{sec:vert-dimin-jump}

The \emph{vertex-diminished jump process} (VDJP)
$(X_{t},\V{L}_{t})$ with initial conditions
$(v,\V{L}_{0})\in\Lambda\times (0,1]^{\Lambda}$
is the Markov process with conditional jump rates
\begin{equation}
  \label{eq:VDJP}
  \P_{v,\V{L}_{0}}\cb{X_{t+dt}=j \mid (X_{s})_{s\leq t}, X_{t}=i} = \beta_{ij}L^{j}_{t}\,dt
\end{equation}
that is stopped at the time
$\death \bydef \inf \{s \mid \text{exists $j\in\Lambda$ s.t.\
  $L^{j}_{s}\leq 0$}\}$.  Here $\V{L}_{t}$ is the collection of local
times of $X_{t}$ defined by
\begin{equation}
  \label{eq:local-VDJP}
  L^{j}_{t} \bydef L^{j}_{0} - \int_{0}^{t}1_{X_{s}=j}\,ds,
\end{equation}
and $L^{j}_{0}>0$ is the \emph{initial local time} at $j$. It is
straightforward to see that $(X_{t},\V{L}_{t})$ is well-defined up to
$\death$ by a step-by-step construction as is done for the VRJP
in~\cite{MR2027294}.
Note that \eqref{eq:RW-models} with $\epsilon=-1$ describes the VDJP with
$\V{L}_{0}=\V{1}$.

The generator $\cL$ of the VDJP acts on smooth functions
$g\colon\Lambda\times (0,1]^{\Lambda}\to \R$ by
\begin{equation}
  \label{eq:VDJP-gen}
  (\cL g)(i,\V{\ell}) =
  \sum_{j\in\Lambda}\beta_{ij}\ell_{j}(g(j,\ell)-g(i,\ell))
  - \ddp{g(i,\ell)}{\ell_{i}}.
\end{equation}
We write $\P_{i,\V{L_{0}}}$ and $\E_{i,\V{L}_{0}}$ for the law
and expectation of the VDJP with initial condition $(i,\V{L}_{0})$.

\subsection{Rotational symmetry}

We consider the space $\R^{n+1}$ equipped with the Euclidean inner product 
$u \cdot v = u^{0}v^{0} + \dots + u^{n}v^{n}$,
which is preserved by the orthogonal group
$O(n+1)$.  In the next section we will
define an unnormalised expectation exactly as in
Section~\ref{sec:euclidean}, but we will investigate the
consequences of rotational symmetries instead of translational symmetries.

\subsection{The hemispherical spin model $\bbS^{n}_{+}$}
\label{sec:hemisph-spin-model}

\subsubsection{Hemispherical space}
\label{sec:hemispherical-space}

In this section we discuss a spin system that takes values in
$\bbS^{n}_{+}$, the open upper hemisphere of the sphere
$\bbS^{n}\subset \R^{n+1}$. See Figure~\ref{fig:spherical}.
For notational convenience we
 write $u = (u^0, \dots, u^n)\in \R^{n+1}$ and
let
$z=u^{0}$, and we will also often write
$x=u^{1}$. Then
\begin{equation}
  \label{eq:hemisphere}
  \bbS^{n}_{+} \bydef \{ u\in\R^{n+1} \mid u\cdot u = 1, z> 0 \},
\end{equation}
where the inner product is Euclidean.
$\bbS^{n}_{+}$ is parametrised by the open unit ball in $\R^{n}$, i.e., by
\begin{equation}
  \label{eq:ball}
  \ball^{n} = \{ (u^{1}, \dots, u^{n})\in \R^{n} \mid (u^{1})^{2} +
  \dots + (u^{n})^{2}< 1\}.
\end{equation}

\subsubsection{Symmetries}
In the flat and hyperbolic settings we considered the Euclidean group
$O(n) \ltimes \R^n$ and the restricted Lorentz group
$SO^{+}(n,1)$.  Unlike in these settings, the orthogonal group
$O(n+1)$ does not preserve the hemisphere. Our results, however, were
based on the \emph{infinitesimal} symmetries of flat and hyperbolic
space, and the hemisphere still possesses useful infinitesimal
symmetries. This section briefly explains this; the key identity
is~\eqref{eq:IBP-hem}.

The infinitesimal symmetries of the hemisphere form a representation
of the Lie algebra $\mathfrak{so}(n+1)$, see
Appendix~\ref{sec:symmetries}. 
The associated invariant measure $du$ on $\bbS^{n}_+$ can be written
in coordinates as
\begin{equation}
  \label{eq:Sn-haar}
  du = \frac{du^{1}\dots du^{n}}{z(u)}, \quad z(u) = 
  \sqrt{1- (u^{1})^{2} - \dots - (u^{n})^{2}}.
\end{equation}
This is the invariant measure on the full sphere
restricted to $\bbS^{n}_{+}$. We let $\theta_{s}$ denote a rotation by $s\in\R$ in the
$xz$-plane. Note that in the coordinates $(x,u^{2},\dots, u^{n})$
the infinitesimal generator of rotations in the $xz$-plane is
\begin{equation}
  \label{eq:Sn-Inf1}
  T \bydef z\ddp{}{x}, 
\end{equation}
which acts on the coordinate functions as
\begin{equation}
  \label{eq:Sn-Inf}
  Tz = -x,\qquad T x= z,  
  \quad Tu^{\alpha} = 0, \quad (\alpha =2,\dots, n).
\end{equation}
A consequence of $T$ being an infinitesimal symmetry of the
hemisphere is that for compactly supported smooth $f\colon
\bbS^{n}_{+}\to \R$,
\begin{equation}
  \label{eq:IBP-hem}
  \int_{\bbS^{n}_{+}}Tf\, du = 0,
\end{equation}
an identity which is also easily proven by rewriting the integral as
an integral over $\bbS^{n}$ and using the rotational invariance of the
full sphere.

\subsubsection{The $\bbS^{n}_{+}$ model}
\label{sec:hemisph-spin-model-1}

By a by now familiar abuse of notation, we write $\bbS^{n\Lambda}_+$ in
place of $(\bbS^{n}_+)^{\Lambda}$. Define,
for $\V{u}\in \bbS^{n\Lambda}_+$,
\begin{equation}
  \label{eq:hemi-H}
  H_{\beta}(\V{u}) \bydef \frac{1}{2}(\V{u},-\Delta_{\beta}\V{u}), 
  \quad H_{\beta,h}(\V{u}) \bydef H_{\beta}(\V{u}) + (\V{h},\V{1}-\V{z}),
\end{equation}
where as before $\beta$ and $\V{h}$ are collections of
non-negative edge and vertex weights, respectively.
For $F\colon \bbS^{n\Lambda}_+\to \R$ define the unnormalised
expectation
\begin{equation}
  \label{eq:hemi-un}
  \cb{F}_{\beta,h} \bydef
  \int_{\bbS^{n\Lambda}_{+}}F(\V{u})e^{-H_{\beta,h}(\V{u})}
  \,d\V{u}
  =
  \int_{\ball^{n\Lambda}} F(\V{u}) e^{-H_{\beta,h}(\V{u})}
  \prod_{i\in\Lambda} \frac{du_{i}^{1}\dots du_{i}^{n}}{z(u_{i})}\,
\end{equation}
where $d\V{u} \bydef \prod_{i\in\Lambda}du_{i}$, and each $du_{i}$ is
a copy of the invariant measure on $\bbS^{n}_+$.  The
\emph{$\bbS^{n}_{+}$ model} is the probability measure defined by the
normalised expectation
\begin{equation}
  \label{eq:hemi-norm}
  \avg{F}_{\beta,h} \bydef \frac{\cb{F}_{\beta,h}}{Z_{\beta,h}}, \quad
  Z_{\beta,h}\bydef \cb{1}_{\beta,h}.
\end{equation}
Unlike the GFF and $\HH^{n}$-model, the $\bbS^{n}_{+}$ model is well-defined if
$\V{h}=\V{0}$, and we will omit the subscripts $h$ to indicate $\V{h}=\V{0}$.

\begin{remark}
  \label{rem:Full-sphere}
  The \emph{spherical $O(n)$ models} are obtained by removing the
  restriction that spins lie in the upper hemisphere
  in~\eqref{eq:hemi-un}. See Remark~\ref{rem:sphericalint} below.
\end{remark}

\subsection{Isomorphism theorems}
\label{sec:isomorphism-theorems-hem}

The following isomorphism theorems are analogues of those in
Section~\ref{sec:euclidean} and~\ref{sec:hyperbolic}.  We again start
with a fundamental integration by parts identity, with the change that
now $\cL$ is the generator of the VDJP, $[\cdot]_\beta$ is the
unnormalised expectation of \eqref{eq:hemi-un}, and $T_j$ is the
infinitesimal rotation in the $xz$-plane in the $j$th coordinate
specified by \eqref{eq:Sn-Inf1}.

\begin{lemma}
  \label{lem:gen-VDJP}
  Let $[\cdot]_\beta$ be the unnormalised expectation of the
  $\bbS_+^n$ model, and let $\E_{i,\V{\ell}}$ be the expectation of
  the VDJP.  Let $f\colon \Lambda\times (0,1]^{\Lambda} \to \R$ be a
  smooth compactly supported function and let
  $\rho\colon \mathbb{S}^{n\Lambda}_+\to \R$ be smooth.  Then:
  \begin{equation}
    \label{eq:gen-VDJP-1}
    -\sum_{j\in\Lambda}\cb{\rho(\V{u})x_{j}  \cL f(j, \V{z})}_{\beta}
    =
    \sum_{j\in\Lambda} \cb{(T_j\rho)(\V{u})f(j,\V{z})}_{\beta}
    .
  \end{equation}
  In particular, the following integrated version holds for compactly
  supported $f\colon\Lambda\times(0,1]^{\Lambda}\to\R$:
  \begin{equation}
    \label{eq:generic-iso-sphere}
    \sum_{j\in\Lambda}\uspin{\rho(\V{u})x_{j} f(j, \V{z})}_{\beta}
    =
    \sum_{j\in\Lambda} \uspin{(T_j\rho)(\V{u})\int_0^\infty \E_{j,\V{z}}(f(X_t,\V{L}_t))\,dt}_{\beta}
    .
  \end{equation}
\end{lemma}

\begin{proof}
  By \eqref{eq:IBP-hem} we can integrate by parts. The proof is almost
  identical to that of Lemma \ref{lem:gen-hyp}, the only differences
  being $\HH^{n\Lambda}$ is replaced $\mathbb{S}^{n\Lambda}$, and
  $T_i = z_i \ddp{}{x_i}$ is the infinitesimal generator of a rotation
  in the $xz$-plane at $i$ instead of a Lorentz boost. This introduces
  a sign, i.e.,
  \begin{equation}
  T_if(\V{z}) = -x_i \ddp{f(\V{z})}{\ell_i}
  \end{equation}
  where the hyperbolic model had a factor of $+1$ in
  \eqref{eq:gen-hyp-8}, producing the VDJP generator instead of the
  VRJP generator. The remainder of the proof is essentially unchanged.
\end{proof}

\begin{remark}
  \label{rem:sphericalint}
  Analytically, \eqref{eq:gen-VDJP-1} holds for the spherical
  $O(n)$ model, although it is no longer obvious how to interpret
  $\cL$ as the generator of a Markov process since `jump rates' become
  negative. In particular, it is unclear how to obtain a formula
  like \eqref{eq:generic-iso-sphere}. A probabilistic
  interpretation of $\cL$
  for the $O(n)$ model, without restricting to the hemisphere, would be very
  interesting.
\end{remark}

The hemispherical BFS--Dynkin isomorphism theorem for the VDJP
is as follows:

\begin{theorem}
  \label{thm:BHS-hemi}
  Let $[\cdot]_\beta$ be the unnormalised expectation of the
  $\bbS_+^n$ model, and let $\E_{i,\V{\ell}}$ be the expectation of
  the VDJP.  Suppose $g\colon (0,1]^{\Lambda} \to \R$ is compactly
  supported. Then for $a,b\in\Lambda$,
  \begin{equation}
    \label{eq:BHS-hemi}
      \uspin{x_{a}x_{b}
        g(\V{z})}_{\beta} =
      \uspin{z_{a}\int_{0}^{\infty} \E_{a,\V{z}}(g(\V{L}_{t})1_{X_t = b})\,dt}_{\beta}.
  \end{equation}
\end{theorem}
\begin{proof}
  Apply Lemma~\ref{lem:gen-VDJP} with $\rho(\V{u}) = x_a$,
  $f(j,\V{\ell}) = g(\V{\ell})1_{j=b}$, and use 
  $T_j\rho(\V{u}) = 1_{j=a} z_j$.
\end{proof}
The fact that finite symmetries do not preserve the hemisphere leads
to slightly different formulations of the Eisenbaum and Ray--Knight
isomorphism theorems as compared to the GFF and $\HH^{n}$ models.  We
let $\cb{F(\V{u}) \delta_{u_{0}}(u_{a})}_{\beta}$ denote the
unnormalised expectation for the spin model in which the spin at
$u_{a}$ is fixed to be $u_{0}\in\bbS^{n}_{+}$.
\begin{theorem}
  \label{thm:RK-VDJP}
  Let $[\cdot]_\beta$ be the unnormalised expectation of the
  $\bbS_+^n$ model, and let $\E_{i,\V{\ell}}$ be the expectation of
  the VDJP.  Let $g\colon(0,1]^{\Lambda}\to \R$ be a smooth compactly supported function, let $a\in\Lambda$, and let
  $s\in(-\frac{\pi}{2},\frac{\pi}{2})$.  Then
  \begin{equation}
    \uspin{
      \E_{a,\V{z}}(g(\bm{L}_{\tau(\cos{s})})1_{\{\tau{(\cos{s})} <
        \zeta\}})\delta_{u_{0}}(u_a)}_\beta
    = \uspin{ g(\bm{z})\delta_{\theta_{s}u_0}(u_a)}_\beta
  \end{equation}
  where $\tau(\gamma) = \inf \{t \,| \,L_t^a \le \gamma\}$ and
  $u_0=(1,0,\dots,0) \in \bbS^n_+$.
\end{theorem}
\begin{proof}
  The proof is analogous to the proof of Theorem~\ref{thm:eucl-rk}.
  Since the identity is trivial if $s=0$, assume $s\neq 0$.  We begin
  by applying Lemma~\ref{lem:gen-VDJP} with
  $\rho(\V{u}) \bydef \rho_{\varepsilon}(u_{a})$,
  $f(j,\V{\ell}) \bydef g(\V{\ell})\eta_{\varepsilon}(\ell_a){1_{j=a}}$, with the functions
  $\rho_\varepsilon \colon \bbS^n_{+} \to \R$ and $\eta_{\varepsilon}
  \colon (0,1] \to\R$ chosen such that $T_{a}\rho$ and $\eta$ are smooth
  compactly supported approximations to
  $\delta_{u_{0}}(u_a)-\delta_{\theta_{s}{u_{0}}}(u_a)$ and
  $\delta_{\cos{s}}(\ell_a)$ subject to $\rho_{\varepsilon}(u_a)\eta_{\varepsilon}(z_a)=0$ for
  all $u_a\in \bbS^n_{+}$.  Since $s\neq 0$, these conditions can be shown
  to be satisfiable by explicit construction. Exactly as in the proof
  of Theorem~\ref{thm:eucl-rk} this yields
  \begin{equation}
    \label{eq:RKhyp2}
    \cb{T_{a}\rho_\varepsilon(u_{a})\int_{0}^{\infty}\E_{a,\V{z}}(g(\V{L}_{t})\eta_\varepsilon(L_{t}^{a}){1_{X_t=a}})\,dt}_{\beta}=0.
  \end{equation}
  To conclude, we use that $\theta_{s}{u_{0}}$ has $z$-coordinate $\cos s$,
  so the term with $\delta_{\theta_{s}{u_{0}}}(u_{a})$ concentrates the
  $u_a$ integral at $z_{a}=\cos s$, and hence the time integral at
  $t=0$. This gives the right-hand side of \eqref{eq:hyp-rk}.  The
  term with $\delta_{u_{0}}(u_{a})$ concentrates the time integral at
  $\tau(\cos s)$ and gives the left-hand side of \eqref{eq:hyp-rk} as
  the integrand is non-zero only if $\tau(\cos s)<\death$.
\end{proof}

The hemispherical Eisenbaum isomorphism theorem concerns a killed
VDJP. The generator of this killed process
$(X_{t},\V{L}_{t})_{t\geq 0}$ acts on smooth compactly supported
$f\colon \Lambda\times(0,1]^{\Lambda}\to \R$ by
\begin{equation}
  \label{eq:VDJP-gen-k}
  (\cL^{h} f)(i,\V{\ell})
  \bydef \cL f(i,\V{\ell})  - h_{i}f(i,\V{\ell}) 
  ,\quad \text{i.e.,}\quad
  \cL^h = \cL - \V{h},
\end{equation}
where $\cL$ is the generator of the VDJP and $h_i\geq 0$ are the
killing rates. We let $\E^{h}_{i,\V{\ell}}$ denote the corresponding
deficient expectation. As for the SRW, the killing does not depend on the
initial local times, i.e.,
\begin{equation}
  \label{eq:VDJP-gen-k-1}
  \E^{h}_{i,\V{\ell}} \pb{g(X_{t},\V{L}_{t})} 
  = 
  \E_{i,\V{\ell}}\pb{g(X_{t},\V{L}_{t})e^{-\sum_{j\in\Lambda}h_{j}(\ell_j-L^{j}_{t})}}.
\end{equation}
Notice that the sign in the killing term
$e^{-\sum_{j\in\Lambda}h_{j}(\ell_j-L^{j}_{t})}$ is reversed: this
because the local times of the VDJP are decreasing rather than
increasing by \eqref{eq:local-VDJP}. We can rewrite \eqref{eq:VDJP-gen-k-1} as
\begin{equation}
  \label{eq:VDJP-nokill}
  \E_{i,\V{\ell}}(g(X_t,\V{L}_t)e^{-\sum_{j\in\Lambda} h_j(1-L_t^j)})) 
  = \E_{i,\V{\ell}}^h(g(X_t,\V{L}_t))e^{-\sum_{j\in\Lambda} h_j (1-\ell_j)}.
\end{equation}

\begin{theorem}
  \label{thm:EIS-VDJP}
  Let $[\cdot]_\beta$ be the unnormalised expectation of the
  $\bbS_+^n$ model, and let $\E_{i,\V{\ell}}$ be the expectation of
  the killed VDJP.  Suppose that $g\colon(0,1]^{\Lambda}\to \R$ is
  compactly supported, and
  $s\in(-\frac{\pi}{2},\frac{\pi}{2})$. Then
  \begin{equation}
    \label{eq:EIS-VDJP-g}
    \uspin{x_{a}
      g(\V{z})e^{-(\V{h},\V{1}-\theta_{-s}\V{z})}}_{\beta} =
    \sum_{i\in\Lambda}h_{i}\uspin{(x_{i}-\theta_{-s}x_{i})
      \int_{0}^{\infty} \E_{i,\V{z}}^h(g(X_{t},\V{L}_{t}))\,dt\,e^{(\V{h},\V{1}-\theta_{-s}\V{z})}}_{\beta}.
  \end{equation}
\end{theorem}

\begin{proof}
  We apply Lemma~\ref{lem:gen-VDJP} with
  \begin{align}
      \rho(\V{u})
      &\bydef e^{(\V{h},\theta_{-s}\V{z} - \V{z})} = e^{(\V{h},\V{1}-\V{z})}(e^{-(\V{h},\V{1}-\theta_{-s}\V{z})})\\
      f(j,\V{\ell}) &\bydef g(\V{\ell})e^{-(\V{h},\V{1}-\V{\ell})}
      1_{j=a},
  \end{align}
  and use that
  $(T_j\rho)(\V{u}) = h_j(x_j - \theta_{-s}x_j) \rho(\V{u})$ to obtain
  \begin{align}
    \nonumber
      \sum_{j\in\Lambda} h_j\uspin{(x_j - \theta_{-s}x_j)
        \rho(\V{u})\int_0^\infty \E_{j,\V{z}}(
        g(\V{L}_t)1_{X_{t}=a}e^{-(\V{h},\V{1}-\V{L_t})})\,dt}_{\beta}
      &= \sum_{j\in\Lambda}\uspin{\rho(\V{u})x_{j}
        g(\V{z})1_{j=a}e^{-(\V{h},\V{1}-\V{z})}}_{\beta}\\
\label{eq:hemi-eis-1}
      &= 
      \uspin{x_{a}
        g(\V{z})e^{-(\V{h},\V{1}-\theta_{-s}\V{z})}}_{\beta}.
  \end{align}
  Using \eqref{eq:VDJP-nokill} to substitute
  \begin{equation}
    \rho(\V{u})\E_{j,\V{z}}(
    g(\V{L}_t)1_{X_{t}=a}e^{-(\V{h},\V{1}-\V{L_t})}) 
    = \E_{j,\V{z}}^h(
    g(\V{L}_t)1_{X_{t}=a})e^{-(\V{h},\V{1}-\theta_{-s}\V{z})}
  \end{equation}
  on the left hand side of \eqref{eq:hemi-eis-1} gives the desired
  result.
\end{proof}

\section{Isomorphism theorems for supersymmetric spin models}
\label{sec:SUSY}

In this section we introduce the supersymmetric $\R^{2|2}$,
$\HH^{2|2}$, and $\bbS^{2|2}_{+}$ spin models and derive isomorphism
theorems that relate them to the SRW, the VRJP, and the VDJP.  Readers who
are not familiar with the mathematics of supersymmetry may consult
Appendix~\ref{sec:SUSY-intro}, which contains an introduction to
supersymmetry as used in this article, before reading this section.

\subsection{Supersymmetric Gaussian free field}
\label{sec:susy-gaussian-free}

\subsubsection{Super-Euclidean space and the SUSY GFF}
\label{sec:SUSYGFF}

The \emph{supersymmetric Gaussian free field} (\emph{SUSY GFF} or
\emph{$\R^{2|2}$ model}) is defined in terms
of the algebra of observables $\Omega^{2\Lambda}(\R^{2\Lambda})
\bydef \Omega^{2|\Lambda|}(\R^{2|\Lambda|})$, see
Appendix~\ref{sec:SUSY-intro}.
This algebra replaces the algebra of
observables $C^\infty(\R^{n\Lambda})$ of the usual $n$-component
Gaussian free field.

Concretely, let $(\xi_{i})_{i\in\Lambda}$ and
$(\eta_{i})_{i\in\Lambda}$ be the generators of the Grassmann algebra
$\Omega^{2\Lambda}$, let $(x_{i},y_{i})_{i\in\Lambda}$ be coordinates
for $\R^{2\Lambda}$, and let $\Omega^{2\Lambda}(\R^{2\Lambda})$ be the
algebra with coefficients in $C^{\infty}(\R^{2\Lambda})$ generated by
$(\xi_{i})_{i\in\Lambda}$ and $(\eta_{i})_{i\in\Lambda}$ as in
Appendix~\ref{sec:SUSY-intro}.  We call elements $F$ of
$\Omega^{2\Lambda}(\R^{2\Lambda})$ \emph{forms}, and say that a form
is smooth, rapidly decaying, compactly supported, etc., if all of its
coefficient functions have this property.

We think of $\Omega^{2\Lambda}(\R^{2\Lambda})$ as the smooth functions
on a putative superspace $(\R^{2|2})^{\Lambda}$, though
$(\R^{2|2})^{\Lambda}$ has no formal meaning, i.e., we will only
work with the algebra $\Omega^{2\Lambda}(\R^{2\Lambda})$.
There are two ordinary (even) coordinates and two anticommuting (odd)
coordinates for each element $i \in \Lambda$, and by
analogy with the familiar representation of a vector $u_i \in \R^{2}$
in terms of its coordinate functions $u_i = (x_i,y_i)$, we will abuse
notation and write $u_i \in \R^{2|2}$ to refer to a \emph{supervector}
$u_i = (x_i,y_i,\xi_i,\eta_i)$, i.e., a tuple of of even and odd coordinates.
Further, we define the super-Euclidean `inner product' on $\R^{2|2}$ by
\begin{equation}\label{eq:susy-euc-ip}
  u_i\cdot u_j \bydef  x_ix_j + y_iy_j - \xi_i\eta_j + \eta_i\xi_j.
\end{equation}

Note that the `inner product' 
\eqref{eq:susy-euc-ip} defines a form, and is not an
inner product in the standard sense of the term. Similarly, we
write $\V{u} = (u_{i})_{i\in\Lambda}$ to denote the collection of the
$u_{i}$, and define $(\V{u},-\Delta_{\beta}\V{u})$ analogously, i.e.,
by
\begin{align}
\label{eq:GFF-Action}
(\V{u},-\Delta_{\beta}\V{u})
&\bydef
\sum_{i\in\Lambda}\sum_{j\in\Lambda} \beta_{ij}(x_{i}(x_{i}-x_{j}) +
y_{i}(y_{i}-y_{j}) -\xi_{i}(\eta_{i}-\eta_{j})
+\eta_{i}(\xi_{i}-\xi_{j}))
\\
&=\frac{1}{2}\sum_{i,j\in\Lambda} \beta_{ij}( u_{i}\cdot u_{i} +
u_{j}\cdot u_{j} - u_{i}\cdot u_{j} - u_j \cdot u_i),
\end{align}
where the second equality is a calculation. The
formal rules introduced above imply the last quantity is
$\frac14\sum_{i,j\in\Lambda}\beta_{ij} (u_{i}-u_{j})^{2}$ if we interpret
$u_{i}-u_{j}$ as
$(x_{i}-x_{j},y_{i}-y_{j},\xi_{i}-\xi_{j},\eta_{i}-\eta_{j})$ and use
\eqref{eq:susy-euc-ip} to compute
$(u_{i}-u_{j})^{2}\bydef   (u_{i}-u_{j})\cdot (u_{i}-u_{j})$.

For $F\in \Omega^{2\Lambda}(\R^{2\Lambda})$, the normalised Berezin integral is denoted
\begin{equation}
  \label{eq:R22int}
  \int_{(\R^{2|2})^{\Lambda}}F \bydef \frac{1}{(2\pi)^{|\Lambda|}}\int\, d\V{x}\,d\V{y}\, \partial_{\V{\eta}}\,\partial_{\V{\xi}}\;F,
\end{equation}
where $\partial_{\V{\eta}}\partial_{\V{\xi}}$ is defined by
$\partial_{\eta_{|\Lambda|}}\partial_{\xi_{|\Lambda|}}\dots \partial_{\eta_{1}}\partial_{\xi_{1}}$,
$d\V{x} = dx_{|\Lambda|}\dots dx_{1}$, and
$d\V{y}=dy_{|\Lambda|}\dots dy_{1}$ for some fixed ordering of the
$i\in\Lambda$ from $1$ to $|\Lambda|$.

To define the supersymmetric GFF, suppose $\V{h}\geq \V{0}$ and let
\begin{equation}
  \label{eq:SGFF-H}
  H_{\beta}(\V{u}) \bydef \frac{1}{2}(\V{u},-\Delta_{\beta}\V{u}), \quad
  H_{\beta,h}(\V{u}) \bydef H_{\beta}(\V{u}) + \frac{1}{2}(\V{h},|\V{u}|^2),
\end{equation}
where $|\V{u}|^{2} \bydef ( u_{i}\cdot u_{i})_{i\in\Lambda}$, and
hence $(\V{h},|\V{u}|^2) = \sum_{i\in\Lambda}h_{i} u_{i}\cdot u_{i}$.
Both $H_{\beta}$ and $H_{\beta,h}$ are elements of
$\Omega^{2\Lambda}(\R^{2\Lambda})$.  The superexpectation of the
\emph{supersymmetric Gaussian free field} is the linear map that
assigns to each $F\in \Omega^{2\Lambda}(\R^{2\Lambda})$ the value
\begin{equation}
  \label{eq:SGFF-E}
  \cb{F}_{\beta,h}
  \bydef \int_{(\R^{2|2})^{\Lambda}}F e^{-H_{\beta,h}},
\end{equation}
and we write $\cb{F}_{\beta}$ when $\V{h}=\V{0}$.  For $\V{h} \neq \V{0}$,
this superexpectation is indeed normalised; see the paragraph below
\eqref{e:localisation-bis}.

\subsubsection{Symmetries} 

In this section we describe the two aspects of the symmetries of the
SUSY GFF that we require. Further details about these symmetries,
which form a Lie superalgebra, can be found in
Appendix~\ref{sec:susy-symmetries}.

As for the GFF, the infinitesimal generator of translation in
the $x$-direction at $i\in\Lambda$ is
\begin{equation}
  T_i \bydef \ddp{}{x_i},
\end{equation}
and $T_{i}$ acts on coordinates as
\begin{equation}
  T_ix_j = 1_{i = j},\quad T_iy_j = 0, \quad T_i\eta_j = 0,\quad
  T_i\xi_j = 0, \qquad i,j\in\Lambda.
\end{equation}

Thus it is analogous to the operators $T_i$ for the ordinary GFF, and
it leads to analogous Ward identities, i.e., for forms $F$ with
sufficient decay,
\begin{equation}
  \label{eq:IBP-SUSY-GFF}
  \int_{(\R^{2|2})^{\Lambda}} (T_{i}F) = 0.
\end{equation}
For $s\in \R$ the finite symmetry associated to
$\sum_{i\in\Lambda} T_i$ will be denoted $\theta_s$, which acts by
\begin{equation}
  \theta_s x_i = x_i+s,\quad\theta_s y_i = y_i, \quad \theta_s \eta_i = \eta_i,\quad \theta_s \xi_i = \xi_i,
  \qquad i\in\Lambda.
\end{equation}

The second symmetry of importance is the \emph{supersymmetry generator}
\begin{equation}\label{eq:susy-gen}
Q \bydef \sum_{i\in\Lambda} Q_i\, 
\qquad 
Q_i \bydef \xi_i \ddp{}{x_i} + \eta_i \ddp{}{y_i} - x_i\ddp{}{\eta_i} + y_i \ddp{}{\xi_i},
\end{equation}
which acts on coordinates as
\begin{equation}
  Qx_i = \xi_i, \quad Qy_i = \eta_i, \quad Q\xi_i = -y_i, \quad
  Q\eta_i = x_i, \qquad i\in\Lambda.
\end{equation}
This supersymmetry generator is responsible for a very powerful Ward
identity known as the \emph{localisation lemma}: for any smooth
function $f\colon \R^{\Lambda\times\Lambda} \to \R$ with sufficient
decay,
\begin{equation} \label{e:localisation-bis}
  \int_{(\R^{2|2})^{\Lambda}} f(\V{u}\V{u}^T) = f(\V{0}),
\end{equation}
where $\V{u}\V{u}^T$ denotes the collection of forms
$(u_i \cdot u_j)_{i,j\in\Lambda}$; see Theorem~\ref{thm:F0} and
Corollary~\ref{cor:localisation}.  In particular, the expectation
\eqref{eq:SGFF-E} is normalised if $\V{h}\neq\V{0}$, i.e.,
$\uspin{1}_{\beta,h}=1$.

\subsubsection{Isomorphism theorems for the SUSY GFF}
\label{sec:isom-theor-susy}

This section presents isomorphism theorems for the SUSY GFF.
The statement of the following
fundamental Ward identity is formally identical to that of
Lemma~\ref{lem:gen-SRW}, but
now the expectation $\cb{\cdot}_{\beta}$ is that of a SUSY GFF.

\begin{lemma}
  \label{lem:gen-SUSY-SRW}
  Let $[\cdot]_\beta$ be the superexpectation of the $\R^{2|2}$ model, and let $\E_{i,\V{\ell}}$ be the expectation of the SRW.
  Let $f\colon \Lambda\times \R^{\Lambda} \to \R$ be a smooth function
  with rapid decay, and let $\rho\in \Omega^{2\Lambda}(\R^{2\Lambda})$
  have moderate growth.  Then:
  \begin{equation}
    \label{eq:gen-lem-susy}
    -\sum_{j\in\Lambda}\cb{\rho(\V{u})x_{j} \cL f(j, \frac{1}{2} |\V{u}|^{2})}_{\beta}
    =
    \sum_{j\in\Lambda} \cb{(T_j\rho)(\V{u})f(j,\frac{1}{2}|\V{u}|^{2})}_{\beta}
    .
  \end{equation}
  In particular, the following integrated version holds for all smooth
  $f\colon\Lambda\times\R^{\Lambda}\to\R$ with rapid decay:
  \begin{equation}
    \label{eq:generic-iso-susy}
    \sum_{j\in\Lambda}\uspin{\rho(\V{u})x_{j} f(j, \frac{1}{2} |\V{u}|^{2})}_{\beta}
    =
    \sum_{j\in\Lambda} \uspin{(T_j\rho)(\V{u})\int_0^\infty
      \E_{j,\frac{1}{2}|\V{u}|^2}(f(X_t,\V{L}_t))\,dt}_{\beta}
    .
  \end{equation}
\end{lemma}
\begin{proof}
  Starting from \eqref{eq:IBP-SUSY-GFF},
  the proof is identical to that of Lemma~\ref{lem:gen-SRW}.
\end{proof}

As a consequence, we obtain the same isomorphism theorems for the
supersymmetric GFF as for the non-supersymmetric one. However, for the
supersymmetric model, we may in addition use \emph{localisation} to
greatly simplify the left-hand side of \eqref{eq:generic-iso-susy}
when $T_j\rho(\V{u})$ is supersymmetric.

\begin{theorem}
  \label{thm:susy-BFSD}
  Let $[\cdot]_\beta$ be the superexpectation of the $\R^{2|2}$ model, and let $\E_{i,\V{\ell}}$ be the expectation of the SRW.
  Let $g\colon \R^{\Lambda}\to \R$ be a smooth function
  with rapid decay, and let $a,b\in\Lambda$. Then
  \begin{equation}
    \label{eq:susy-BFSD}
    \uspin{x_{a}x_{b}g(\frac{1}{2}|\V{u}|^{2})}_{\beta}
    = \int_{0}^{\infty} \E_{a,0}(g(\V{L}_{t})1_{X_t = b})\, dt.
  \end{equation}
\end{theorem}

\begin{proof}
  Apply Lemma~\ref{lem:gen-SUSY-SRW} with $\rho(\V{u}) = x_a$,
  $f(j,\V{\ell}) = g(\V{\ell})1_{j=b}$, and note
  $T_j\rho(\V{u}) = 1_{j=a}$. Thus
  the integrand on the right-hand side of \eqref{eq:generic-iso-susy}
  is a function of $|\V{u}|^2$, and hence is
  supersymmetric. By applying
  localisation, i.e., \eqref{e:localisation-bis}, we conclude
  \begin{equation*}
    \uspin{\int_0^\infty \E_{a,\frac{1}{2}|\V{u}|^{2}} \pa{1_{X_t=b}g(\V{L}_t)}\, dt }_{\beta} = \int_0^\infty \E_{a,0} \pa{1_{X_t=b}g(\V{L}_t)}\, dt.\qedhere
  \end{equation*}
\end{proof}

\begin{remark}
  \label{rem:SUSY-GFF-origin}
  Theorem~\ref{thm:susy-BFSD} has its origins in
  physics~\cite{PhysRevLett.43.744,MR594576,MR660000,MR713539}.
  A formulation similar to
  the one presented here was given in~\cite{MR1100240}, see
  also~\cite{MR941982}.
\end{remark}

The Ray--Knight isomorphism theorem applies to spin models in which
the spin at vertex $a$ is fixed; in the supersymmetric version the
constraint is $u_a=(0,0,0,0)$.  We write the corresponding
unnormalised expectation of an observable $F$ as
\begin{equation}
  [F \delta_{u_0}(u_a)]_\beta.
\end{equation}

\begin{theorem}
  \label{thm:susy-eucl-rk}
  Let $[\cdot]_\beta$ be the superexpectation of the $\R^{2|2}$ model,
  and let $\E_{i,\V{\ell}}$ be the expectation of the SRW.  Let
  $g\colon\R^{\Lambda}\to\R$ be smooth and compactly
    supported, 
  let $a\in\Lambda$, and let $s\in\R$.  Then
  \begin{equation}
    \uspin{ g(\frac{1}{2}|\theta_s\V{u}|^2)\delta_{u_0}(u_a)}_\beta = 
    \E_{a,0}g(\V{L}_{\tau(\frac{s^2}{2})}) 
  \end{equation}
  where $\tau(\gamma) \bydef \inf \{t \,| \,L_a^t \ge \gamma\}$ and $u_0 = (0,0,0,0)$.
\end{theorem}

\begin{proof}
  The proof is by applying Lemma~\ref{lem:gen-SUSY-SRW} with
  $\rho(\V{u}) \equiv \rho_{\varepsilon}(u_{a})$,
  $f(j,\V{\ell}) \bydef  g(\V{\ell})\eta_{\varepsilon}(\ell_a)1_{j=a}$, and the form
  $\rho_{\varepsilon} \in \Omega^{2}(\R^{2})$ and function $ \eta_{\varepsilon} \colon\R\to\R$ chosen such that
  $T_{a}\rho_{\varepsilon}$ and $ \eta_{\varepsilon}$ are smooth compactly supported approximations
  to $\delta_{u_0}(u_a)-\delta_{u_0}(\theta_{-s}u_a)$ and $\delta_{\frac{1}{2}s^{2}}$
  subject to $\rho_{\varepsilon}(u_a) \eta_{\varepsilon} (\frac{1}{2}|u_a|^{2})=0$.
  We refer to Appendix~\ref{app:delta} for smooth approximations to $\delta_{u_0}(u_a)$. 
     
  An argument identical to the one in the proof of
  Theorem~\ref{thm:eucl-rk} shows
  \begin{multline}
    \uspin{\delta_{u_0,\varepsilon}(\theta_{-(s-\varepsilon)}u_a)\int_{0}^{\infty}\E_{a,\frac{1}{2}|\V{u}|^2}(g(\V{L}_t)\eta_\varepsilon(L_t^a)1_{X_t = a}\,dt}_\beta
    \\
    =
    \uspin{\delta_{u_0,\varepsilon}(u_a)\int_{0}^{\infty}\E_{a,\frac{1}{2}|\V{u}|^2}(g(\V{L}_t)\eta_\varepsilon(L_t^a)1_{X_t = a})\,dt}_\beta
    .
  \end{multline}
  By choosing $\delta_{u_0,\varepsilon}(u_a)$ to be supersymmetric, i.e., $Q\delta_{u_0,\varepsilon} = 0$, the integrand on the right-hand side is a product of supersymmetric forms and is therefore supersymmetric.  Applying supersymmetric
  localisation (i.e., \eqref{e:localisation-bis}) hence shows
  \begin{equation}
   \uspin{\delta_{u_0,\varepsilon}(\theta_{-(s-\varepsilon)}u_a)\int_{0}^{\infty}\E_{a,\frac{1}{2}|u|^2}(g(\V{L}_t)\eta_\varepsilon(L_t^a)1_{X_t = a})\,dt}_\beta =
   \int_{0}^{\infty}\E_{a,0}(g(\V{L}_t)\eta_\varepsilon(L_t^a)1_{X_t = a})\,dt.
  \end{equation}
  Applying a global translation $\theta_{s-\varepsilon}$ on the left-hand side and then taking $\varepsilon \rightarrow 0$
  as in the proof of Theorem~\ref{thm:eucl-rk} gives the desired result
  \begin{equation*}
    \uspin{g(\frac{1}{2}|\theta_{s}\V{u}|^2)\delta_{u_0}(u_a)}_{\beta} =  \E_{a,0}g(\V{L}_{\tau(\frac{s^2}{2})}).\qedhere
  \end{equation*}
\end{proof}

The preceding two theorems are analogues of the BFS--Dynkin
and Ray--Knight isomorphisms for the SUSY GFF. While calculations
analogous to those leading to the Eisenbaum isomorphism can be
carried out for the SUSY GFF, it is not possible to apply
localisation, because the form $\frac{1}{2}|\theta_{s}\V{u}|^{2}$
that arises (recall~\eqref{eq:EIS-g}) is not supersymmetric.

\subsection{SUSY hyperbolic model $\HH^{2|2}$}
\label{sec:susy-hyperb-model}

In this section we introduce the supersymmetric analogue of
the $\HH^{2}$ model, and then obtain the 
associated isomorphism theorems.

\subsubsection{Super-Minkowski space $\R^{3|2}$ and the super-Minkowski model}
\label{sec:super-min}

Let $(\xi_{i},\eta_{i})_{i\in\Lambda}$ be the generators of the
Grassmann algebra $\Omega^{2\Lambda}$. The algebra of observables
$\Omega^{2\Lambda}(\R^{3\Lambda})$ is 
the algebra generated by $(\xi_{i},\eta_{i})_{i\in\Lambda}$ with
coefficients in $C^{\infty}(\R^{3\Lambda})$. Choosing orthonormal
coordinates $(z_i,x_i,y_i)_{i \in \Lambda}$ for $\R^{3\Lambda}$, a
supervector $u_i \in \R^{3|2}$ is a tuple of even and odd coordinates
$u_i = (z_i,x_i,y_i, \xi_i,\eta_i)$, and we say that $\R^{3|2}$ is a
\emph{super-Minkowski space} when equipped with the `inner product'
\begin{equation} 
  \label{eq:hypprod}
  u_i\cdot u_j \bydef -z_iz_j + x_ix_j + y_iy_j - \xi_i\eta_j + \eta_i\xi_j.
\end{equation}
We have written `inner product' to emphasise that $u_{i}\cdot
u_{j}$ is a form, and hence this is not an
inner product in the standard sense of the term.

\subsubsection{$\HH^{2|2}$ sigma model}
\label{sec:H22int}

To define a supersymmetric analogue of $\HH^{2}$, define the even form
\begin{equation}
  \label{eq:zform}
  z = z(x,y,\xi,\eta) \bydef \sqrt{1+x^{2}+y^{2} - 2\xi\eta} 
  = {\sqrt{1+x^{2}+y^{2}}}
  - \frac{\xi\eta}{\sqrt{1+x^{2}+y^{2}}}.
\end{equation}

Using the definition \eqref{eq:hypprod}, a short calculation shows that
$u_{i}\cdot u_{i} = -1$, just as for $\HH^{2}$.
The algebra of forms $\Omega^{2}(\HH^{2})$ is the algebra over $C^{\infty}(\HH^{2})$ generated by two Grassmann generators $\xi$ and $\eta$.
In coordinates, we have $F(u) = F(z,x,y,\xi,\eta) = F(\sqrt{1+x^{2}+{y}^{2} - 2{\xi}{\eta}}, x,y,\xi,\eta)$,
and hence every form $F\in \Omega^{2}(\HH^{2})$ can be identified with a form in $\Omega^{2}(\R^{2})$.
Using this correspondence 
we define the Berezin integral for $F\in \Omega^{2}(\HH^{2})$ as
\begin{equation}
  \label{eq:H22int}
  \int_{\HH^{2|2}}F \bydef \int_{\R^{2|2}} \frac{1}{z} F = \frac{1}{2\pi}
    \int dx\,dy\, \partial_{\xi}\, \partial_{\eta}  \frac{1}{z} F
\end{equation}
where on the right-hand side
we are viewing $F$ as a form in $\Omega^2(\R^{2})$. Similarly,
\begin{equation}
  \label{eq:H22-int-L}
  \int_{(\HH^{2|2})^{\Lambda}}F \bydef
  \int_{(\R^{2|2})^{\Lambda}} \frac{1}{\prod_{i\in\Lambda} z_{i}} F
  =
  \frac{1}{(2\pi)^{|\Lambda|}} \int\, d\V{x}\,d\V{y}\, \partial_{\V{\eta}}\,\partial_{\V{\xi}}\;\frac{1}{\prod_{i\in\Lambda} z_{i}}F
\end{equation}
where we note there is no ambiguity in the product of the $z_{i}$ as they
are even forms.

Define, for $\V{h}\geq \V{0}$,
\begin{equation}
  \label{eq:H22-def}
  H_{\beta}(\V{u}) \bydef \frac{1}{2}(\V{u},-\Delta_{\beta}\V{u}),
  \quad H_{\beta,h}(\V{u}) \bydef H_{\beta}(\V{u}) + (\V{h},\V{z}-\V{1}), 
\end{equation}
where
\begin{equation}
  \label{eq:H22-Lap}
  \begin{split}
    (\V{u},-\Delta_{\beta}\V{u}) 
    &\bydef
    \frac{1}{2}\sum_{i,j\in\Lambda} \beta_{ij}( u_{i}\cdot u_{i} + u_{j}\cdot u_{j} - u_{i}\cdot u_{j} - u_j \cdot u_i)
    =
    \frac{1}{2}\sum_{i,j\in\Lambda} \beta_{ij}( -2 - 2 u_{i}\cdot u_{j}),
    \\
    (\V{h},\V{z}-\V{1}) &\bydef \sum_{i\in\Lambda}h_{i}(z_{i}-1),
    \end{split}
\end{equation}
and each $u_{i}\cdot u_{j}$ is defined as in \eqref{eq:hypprod}.
The equality in the first line holds because $u_i\cdot u_i =-1$.
We define the \emph{$\HH^{2|2}$ model} superexpectation for $F\in \Omega^{2\Lambda}(\HH^{2\Lambda})$ by
\begin{equation}
    \label{eq:SGFF-Hyp}
  \uspin{F}_{\beta,h}\bydef
  \int_{(\HH^{2|2})^{\Lambda}}F e^{-H_{\beta,h}},
\end{equation}
and we write $\cb{F}_{\beta}$ in the case $\V{h}=\V{0}$.
For $\V{h} \neq \V{0}$, the superexpectation is normalised, i.e.,
$\uspin{1}_{\beta,h}=1$. This is a consequence of supersymmetry, 
see \eqref{eq:h22loc} below.

\subsubsection{Symmetries}
There are two symmetries necessary for what follows, and we
introduce them in this section. For a further discussion of the
Lie superalgebra of symmetries associated to the $\HH^{2|2}$ model see
Appendix~\ref{sec:susy-symmetries}.

The first relevant symmetry is the infinitesimal Lorentz boost
in the $xz$ plane at $i\in\Lambda$:
\begin{equation}
  \label{eq:H22-boost}
  T_i \bydef z_i \ddp{}{x_i} = \sqrt{1+x^2+y^2 - 2\xi\eta}\ddp{}{x_i},
\end{equation}
which acts on coordinates as
\begin{equation}
T_iz_j = x_j1_{i=j},\quad T_ix_j = z_j1_{i=j},\quad  T_iy_j = 0,\quad
T_i\xi_j = 0,\quad T_i\eta_j = 0 \qquad i,j\in\Lambda.
\end{equation}
As for the SUSY GFF, this leads to a Ward identity
for forms $F$ with rapid decay:
\begin{equation}
  \label{eq:IBP-SUSY-H22}
  \int_{(\HH^{2|2})^{\Lambda}} (T_{i}F) = 0  .
\end{equation}
For $s\in\R$ the finite symmetry associated to $\sum_{i\in\Lambda} T_i$ will be denoted
$\theta_s$, and acts as (for $j\in\Lambda$)
\begin{equation}
\label{eq:h22-boost}
\theta_s z_j = z_j\cosh{s} + x_j\sinh{s},\;\; \theta_sx_j =
x_j\cosh{s} + z_j\sinh{s},\;\;  \theta_sy_j = y_j,\;\; \theta_s\xi_j
= \xi_j,\;\; \theta_s\eta_j = \eta_j. 
\end{equation}

The second relevant symmetry is the supersymmetry generator $Q$, which
is defined by \eqref{eq:susy-gen}.  Note that $z_i$ can be written as
$z_i = \sqrt{1+ |\tilde u_i|^2}$, where
$\tilde u_i \bydef (x_i,y_i,\xi_i,\eta_i) \in \R^{2|2}$.  Thus,
$z_{i}$ is supersymmetric, i.e., $Qz_{i}=0$. This implies the same
localisation Ward identity applies for $\HH^{2|2}$ as for $\R^{2|2}$,
i.e., for smooth functions
$f\colon \R^{\Lambda}\times \R^{\Lambda\times\Lambda}\to \R$ with
sufficient decay,
\begin{equation}
  \label{eq:h22loc}
  \int_{(\HH^{2|2})^{\Lambda}} f(\V{z},\V{\tilde u}\V{\tilde u}^T) = f(\V{1},\V{0})
\end{equation}
where $\V{0}$ is the matrix indexed by $\Lambda$ with all entries $0$, and we
have written $\V{\tilde u}\V{\tilde u}^{T}$ to denote the set of
forms $(\tilde u_{i}\cdot \tilde u_{j})_{i,j\in\Lambda}$.

\subsubsection{Isomorphism theorems for the $\HH^{2|2}$ model}
\label{sec:isom-theor-hh22}

Let $\E_{i,\V{\ell}}$ denote the expectation for a VRJP started from
initial conditions $(i,\V{\ell})$. We begin with the SUSY analogue of
Lemma~\ref{lem:gen-hyp}.
\begin{lemma}
  \label{lem:gen-SUSY-H22}
  Let $[\cdot]_\beta$ be the superexpectation of the $\HH^{2|2}$ model, and let $\E_{i,\V{\ell}}$ be the expectation of the VRJP.
  Let $f\colon \Lambda\times \R^{\Lambda} 
  \to \R$ be a smooth function with rapid decay, and let
  $\rho\in \Omega^{2\Lambda}(\HH^{2\Lambda})$ have moderate growth.
  Then:
  \begin{equation}
    \label{eq:gen-lem-H22}
    -\sum_{j\in\Lambda}\cb{\rho(\V{u})x_{j} \cL f(j,\V{z})}_{\beta}
    =
    \sum_{j\in\Lambda} \cb{(T_j\rho)(\V{u})f(j,\V{z})}_{\beta}
    .
  \end{equation}
  In particular, the following integrated version holds for all
  smooth $f\colon\Lambda\times \R^{\Lambda} 
  \to\R$ with rapid decay:
  \begin{equation}
    \label{eq:generic-iso-susy-H22}
    \sum_{j\in\Lambda}\uspin{\rho(\V{u})x_{j} f(j, \V{z})}_{\beta}
    =
    \sum_{j\in\Lambda} \uspin{(T_j\rho)(\V{u})\int_0^\infty \E_{j,\V{z}}(f(X_t,\V{L}_t))\,dt}_{\beta}
    .
  \end{equation}
\end{lemma}
\begin{proof}
The proof is identical to that of Lemma~\ref{lem:gen-hyp}.
\end{proof}

The SUSY analogue of Theorem~\ref{thm:BHS} is the following.
\begin{theorem}
  \label{thm:SUSY-BHS}
  Let $[\cdot]_\beta$ be the superexpectation of the $\HH^{2|2}$
  model, and let $\E_{i,\V{\ell}}$ be the expectation of the VRJP.
  Let $g\colon \Lambda\times\R^{\Lambda}\to \R$ be a smooth function
  with rapid decay, and let $a,b\in\Lambda$. Then
  \begin{equation}
    \label{eq:H22-iso}
    \uspin{x_{a}x_{b}g(\V{z})}_{\beta} =
    \int_{0}^{\infty}\E_{a,\V{1}}(g(\V{L}_{t})1_{X_t = b})\,dt.
  \end{equation}
\end{theorem}

\begin{proof}
  Apply Lemma~\ref{lem:gen-SUSY-H22} with $\rho(\V{u}) = x_a$
  and $f(j,\V{\ell}) = g(\V{\ell})1_{j=b}$. Thus
  $T_j\rho(\V{u}) = 1_{j=a}z_a$. By applying localisation, i.e.,
  \eqref{eq:h22loc}, we obtain
  \begin{equation*}
    \uspin{x_{a}x_{b}g(\V{z})}_{\beta} =\uspin{z_a \int_{0}^{\infty}\E_{a,\V{z}}(g(\V{L}_{t})1_{X_t = b})\,dt}_{\beta} = \int_{0}^{\infty}\E_{a,\V{1}}(g(\V{L}_{t})1_{X_t = b})\,dt.\qedhere
  \end{equation*}
\end{proof}

\begin{theorem} \label{thm:H22-rk}
  Let $[\cdot]_\beta$ be the superexpectation of the $\HH^{2|2}$ model, and let $\E_{i,\V{\ell}}$ be the expectation of the VRJP.
  Let $g\colon\R^{\Lambda}\to\R$ be a smooth compactly supported function,
  let $a\in\Lambda$, and let $s\in\R$.  Then
  \begin{equation}\label{eq-susy-h22-rk}
    \uspin{ g(\theta_s\V{z})\delta_{u_0}(u_a)}_\beta =   \E_{a,\V{1}}g(\V{L}_{\tau(\cosh{s})})
  \end{equation}
  where $\tau(\gamma) = \inf \{t \,| \,L_a^t \ge \gamma\}$ and $u_0 = (1,0,0,0,0)$.
\end{theorem}
\begin{proof} Applying Lemma~\ref{lem:gen-SUSY-H22} with
  $\rho(\V{u}) \equiv \rho_{\varepsilon}(u_{a})$,
  $f(j,\V{\ell}) \bydef g(\V{\ell})\eta_{\varepsilon}(\ell_a)1_{j=a}$, and the form
  $\rho_{\varepsilon} \in \Omega^{2}(\HH^{2})$ and function $\eta_{\varepsilon} \colon\R_+\to\R$
  chosen such that $T_{a}\rho_\varepsilon$ and $\eta_\varepsilon$ are smooth compactly
  supported approximations to
  $\delta_{u_0}(u_a)-\delta_{\theta_{s}u_0}(u_a)$ and
  $\delta_{\cosh{s}}$ subject to $\rho_{\varepsilon}(u_a)\eta_\varepsilon(z_a)=0$, an argument
  identical to the proof of Theorem~\ref{thm:hyp-rk} shows

  \begin{equation}
  \uspin{\delta_{u_0,\varepsilon}(\theta_{-(s-\varepsilon)}u_a)\int_{0}^{\infty}\E_{a,\V{z}}(g(\V{L}_t)\eta_\varepsilon(L_t^a)1_{X_t = a})\,dt}_\beta =
  \uspin{\delta_{u_0,\varepsilon}(u_a)\int_{0}^{\infty}\E_{a,\V{z}}(g(\V{L}_t)\eta_\varepsilon(L_t^a)1_{X_t = a})\,dt}_\beta.
  \end{equation}
  As in the proof of Theorem~\ref{thm:susy-eucl-rk}, $\delta_{u_0,\varepsilon}(u_a)$ is chosen to be supersymmetric.
  The claim follows by applying localisation to the right-hand side, boosting the left-hand side by $\theta_{s-\varepsilon}$,
  and then taking $\varepsilon \rightarrow 0$ as in the proof of Theorem~\ref{thm:hyp-rk}:
  \begin{equation*}
    \uspin{\E_{a,\V{z}}g(\bm{L}_{\tau(\cosh{s})})\delta_{u_0}(u_a)}_\beta = \E_{a,\V{1}}g(\bm{L}_{\tau(\cosh{s})}).\qedhere
  \end{equation*}
\end{proof}

\begin{remark}
  \label{rem:H22-history} 
  The $\HH^{2|2}$ model was introduced in~\cite{MR1134935}; it serves
  as a toy model for Efetov's supersymmetric approach to studying
  random band matrices~\cite{MR708812}. The connection between random
  band matrices and hyperbolic symmetry goes back to Wegner and
  Sch\"{a}fer~\cite{MR600875,MR575503}, and Efetov made use of
  supersymmetry to avoid the use of the replica trick. For further
  discussion see~\cite{MR2728731}, and for other uses of
  supersymmetry in the study of random matrices see,
  e.g., \cite{MR2114358,MR3824956,DisertoriLohmannSodin}.
\end{remark}

\begin{remark}
\label{rem:H22-phenom}
  Unlike the $\HH^{n}$ models, the $\HH^{2|2}$ model captures the
  phenomenology of a localisation/delocalisation
  transition~\cite{MR2728731,MR2953867}.
\end{remark}

\subsection{SUSY hemispherical model $\bbS^{2|2}_{+}$}
\label{sec:susy-hemisph-model}

In this section we introduce the supersymmetric analogue of
the $\bbS^{2}_{+}$ model, and then obtain the associated isomorphism
theorems.

\subsubsection{Integrals over $\bbS_{+}^{2|2}$}
\label{sec:S22int}

In this subsection we work with smooth compactly supported forms in
$\Omega^{2\Lambda}(\bbS_+^{2\Lambda})$, which we denote
$\Omega^{2\Lambda}_{c}(\bbS_+^{2\Lambda})$.  Concretely, we will
identify such forms with compactly supported forms in
$\Omega^{2\Lambda}(\bbB^{2\Lambda})$, where $\bbB^2$ is the open unit
ball, by setting 
\begin{equation}
  \label{eq:zformS}
  z = z(x,y,\xi,\eta) \bydef \sqrt{1-x^{2}-y^{2} + 2\xi\eta} 
  = \sqrt{1-x^{2}-y^{2}}
  + \frac{\xi\eta}{\sqrt{1-x^{2}-y^{2}}},
\end{equation}
By considering $\bbB^2$ as a subset of $\R^{2}$, a compactly supported
form in $\Omega^{2\Lambda}(\bbB^{2\Lambda})$ can be trivially extended
to a form in $\Omega^{2\Lambda}(\R^{2\Lambda})$, and we may therefore
apply the results of Appendix~\ref{sec:SUSY-intro}.

Similarly to the notation introduced in Section~\ref{sec:H22int},
let $u_{i}=(z_{i},x_{i},y_{i},\xi_{i},\eta_{i})$, and let
\begin{equation}
  \label{eq:hemprod}
  u_{i}\cdot u_{j} \bydef z_{i}z_{j} +  x_{i}x_{j} + y_{i}y_{j}  -
  \xi_{i}\eta_{j} + \eta_{i}\xi_{j}, \qquad i,j\in\Lambda.
\end{equation}
With these definitions, $u_{i}\cdot u_{i} = 1$, just as for
$\bbS^{2}_{+}$. We define, for $F\in \Omega^{2}_c(\bbS^{2}_+)$,
\begin{equation}
  \label{eq:S22int}
  \int_{\bbS^{2|2}_{+}}F \bydef 
  \frac{1}{2\pi} \int dx\,dy\, \partial_{\xi}\, \partial_{\eta}  \frac{1}{z} F,
\end{equation}
and similarly, for $F\in\Omega^{2\Lambda}_{c}(\bbS^{2\Lambda}_{+})$,
\begin{equation}
  \label{eq:S22-int-L}
  \int_{(\bbS^{2|2}_{+})^{\Lambda}}F \bydef 
  \frac{1}{(2\pi)^{|\Lambda|}} \int d\V{x}\,d\V{y}\, \partial_{\V{\xi}}\, \partial_{\V{\eta}}\frac{1}{\prod_{i\in\Lambda} z_{i}} F,
\end{equation}
where we note there is no ambiguity in the product of the $z_{i}$ as
they are even forms.

\subsubsection{$\bbS^{2|2}_+$ model}
\label{sec:model-S22}

Define, for $\V{h}\geq \V{0}$,
\begin{equation}
  \label{eq:S22-def}
  H_{\beta}(\V{u}) \bydef \frac{1}{2}(\V{u},-\Delta_{\beta}\V{u}),
  \quad H_{\beta,h}(\V{u}) \bydef H_{\beta}(\V{u}) + (\V{h},\V{1}-\V{z}) ,
\end{equation}
where
\begin{equation}
  \label{eq:S22-Lap}
  \begin{split}
    (\V{u},-\Delta_{\beta}\V{u}) 
    &\bydef \frac{1}{2}\sum_{i,j\in\Lambda} \beta_{ij}( u_{i}\cdot u_{i} +
    u_{j}\cdot u_{j} - u_{i}\cdot u_{j} - u_j \cdot u_i)
    =
    \frac{1}{2}\sum_{i,j\in\Lambda} \beta_{ij}(2 - 2 u_{i}\cdot u_{j})
    , \\
     (\V{h},\V{1}-\V{z})  & \bydef\sum_{i\in\Lambda}h_{i}(1-z_{i}) \\
    \end{split}
\end{equation}
and $u_{i}\cdot u_{j}$ is defined as in \eqref{eq:hemprod}.  We
define the \emph{$\bbS^{2|2}_{+}$ model} superexpectation of
$F\in \Omega^{2\Lambda}_c(\bbS^{2\Lambda}_+)$ by
\begin{equation}
    \label{eq:SGFF-Hem}
	\uspin{F}_\beta
  \bydef \int_{(\bbS^{2|2}_{+})^{\Lambda}}F e^{-H_{\beta}}, \quad
  \uspin{F}_{\beta,h}\bydef
  \int_{(\bbS^{2|2}_{+})^{\Lambda}}F e^{-H_{\beta,h}}.
\end{equation}

\subsubsection{Symmetries}
As in the previous sections, there are two
symmetries of relevance to the following discussion.
For details on the Lie superalgebra associated to
$\bbS^{2|2}_{+}$, see Appendix~\ref{sec:susy-symmetries}.
The first symmetry of relevance is an
infinitesimal rotation in the $xz$-plane at $i\in\Lambda$, which has
generator
\begin{equation}
T_i \bydef z_i \ddp{}{x_i} = \sqrt{1-x_i^2-y_i^2+2\xi_i\eta_i}\ddp{}{x_i},
\end{equation}
and acts on coordinates as
\begin{equation}
T_iz_j = -x_j1_{i=j},\quad T_ix_j = z_j1_{i=j},\quad  T_iy_j = 0,\quad
T_i\xi_j = 0,\quad T_i\eta_j = 0, \qquad i,j\in\Lambda.
\end{equation}
As for the SUSY GFF, this leads to a Ward identity
for all sufficiently rapidly decaying forms $F$:
\begin{equation}
  \label{eq:IBP-SUSY-S22}
  \int_{(\bbS^{2|2}_{+})^{\Lambda}} (T_{i}F) = 0  .
\end{equation}
For $s\in\R$ the finite rotation associated to $\sum_{i\in\Lambda} T_i$ is denoted
$\theta_s$, and acts as, for $j\in\Lambda$,
\begin{equation}
\theta_s z_j = z_j\cos{s} - x_j\sin{s},\quad \theta_sx_j = x_j\cos{s} + z_j\sin{s},\quad  \theta_sy_j = y_j,\quad \theta_s\xi_j = \xi_j,\quad \theta_s\eta_j = \eta_j.
\end{equation}

The second symmetry of importance is the supersymmetry generator
$Q$ defined by \eqref{eq:susy-gen}.  Note that $z_i$ can be written as
$z_i = \sqrt{1- |\tilde u_i|^2}$, where
$\tilde u_i \bydef (x_i,y_i,\xi_i,\eta_i) \in \R^{2|2}$. It follows
that $z_i$ is supersymmetric, i.e., $Qz_i=0$. This implies the
same localisation Ward identity applies for $\bbS^{2|2}_+$ as for
$\R^{2|2}$, i.e., for $f\colon (0,1]^{\Lambda}\times
[-1,1]^{\Lambda\times\Lambda}\to \R$ that are smooth and compactly supported, 
\begin{equation}
  \label{eq:s22loc}
  \int_{(\bbS^{2|2}_{+})^{\Lambda}} f(\V{z},\V{\tilde  u}\V{\tilde u}^T) =  f(\V{1},\V{0}),
\end{equation}
where $\V{0}$ is the matrix indexed by $\Lambda$ with all entries $0$
and $\V{\tilde u}\V{\tilde u}^{T} \bydef (\tilde u_{i}\cdot \tilde u_{j})_{i,j\in\Lambda}$.

\subsubsection{Isomorphism theorems for the $\bbS^{2|2}_{+}$ model}
\label{sec:isom-theor-S22}

Let $\E_{i,\V{\ell}}$ denote the expectation for a VDJP started from
initial conditions $(i,\V{\ell})\in\Lambda\times(0,1]^{\Lambda}$. 
\begin{lemma}
  \label{lem:gen-SUSY-S22}
  Let $[\cdot]_\beta$ be the superexpectation of the $\bbS_+^{2|2}$ model, and let $\E_{i,\V{\ell}}$ be the expectation of the VDJP.
  Let $f\colon \Lambda\times (0,1]^{\Lambda} \to \R$ be a smooth
  compactly supported function and let
  $\rho \in \Omega_{c}^{2\Lambda}(\bbS_+^{2\Lambda})$.  Then:
  \begin{equation}
    \label{eq:gen-S22-1}
    -\sum_{j\in\Lambda}\cb{\rho(\V{u})x_{j}  \cL f(j, \V{z})}_{\beta}
    =
    \sum_{j\in\Lambda} \cb{(T_j\rho)(\V{u})f(j,\V{z})}_{\beta}
    .
  \end{equation}
  In particular, the following integrated version holds for smooth and 
  compactly supported $f\colon\Lambda\times(0,1]^{\Lambda}\to\R$:
  \begin{equation}
	\label{eq:generic-iso-S22}
	\sum_{j\in\Lambda}\uspin{\rho(\V{u})x_{j} f(j, \V{z})}_{\beta}
	=
	\sum_{j\in\Lambda} \uspin{(T_j\rho)(\V{u})\int_0^\infty
          \E_{j,\V{z}}(f(X_t,\V{L}_t))\,dt}_{\beta} 
	.
      \end{equation}
\end{lemma}
\begin{proof}
  The proof is identical to that of Lemma~\ref{lem:gen-VDJP}.
\end{proof}

The SUSY analogue of Theorem~\ref{thm:BHS-hemi} is the following. 

\begin{theorem}
  \label{thm:SUSY-hem}
  Let $[\cdot]_\beta$ be the superexpectation of the $\bbS_+^{2|2}$ model, and let $\E_{i,\V{\ell}}$ be the expectation of the VDJP.
  Let $g\colon (0,1]^{\Lambda}\to \R$ be a smooth compactly supported
  function, and let $a,b\in\Lambda$. Then
  \begin{equation}
    \label{eq:S22-iso}
    \uspin{x_{a}x_{b}g(\V{z})}_{\beta} =
    \int_{0}^{\infty}\E_{a,\V{1}}(g(\V{L}_{t})1_{X_t = b})\,dt.
  \end{equation}
\end{theorem}
\begin{proof}
  The proof is essentially identical to that of Theorem~\ref{thm:SUSY-BHS}.
\end{proof}

\begin{theorem} 
  \label{thm:S22-rk}
  Let $[\cdot]_\beta$ be the superexpectation of the $\bbS_+^{2|2}$ model, and let $\E_{i,\V{\ell}}$ be the expectation of the VDJP.
  Let $g\colon(0,1]^{\Lambda}\to \R$ be a smooth compactly supported function,
  let $a\in\Lambda$, and let $s\in(-\frac{\pi}{2},\frac{\pi}{2})$.  Then
  \begin{equation}
    \uspin{ g(\V{z})\delta_{\theta_s u_0}(u_a)}_\beta =
    \E_{a,\V{1}}(g(\V{L}_{\tau(\cos{s})})1_{\tau(\cos s)<\death})
  \end{equation}
  where $\tau(\gamma) = \inf \{t \,| \,L_t^a \le \gamma\}$ and
 $\theta_s u_0=(\cos{s},\sin{s},0,0,0)\in \bbS^{2|2}_{+}$.
\end{theorem}

\begin{proof}
  The proof is, \emph{mutatis mutandis}, identical to that of
  Theorem~\ref{thm:H22-rk}.
\end{proof}

\section{Application to limiting local times: the Sabot--Tarr\`{e}s limit}
\label{sec:ST}

In \cite{MR3420510}, Sabot and Tarr\`{e}s established the first connection
between the vertex-reinforced jump process and the SUSY hyperbolic
sigma model. Their result relates the asymptotic local time
distribution of a time-changed VRJP to a certain \emph{horospherical}
marginal of the $\HH^{2|2}$ model.  In this section we derive their
result (as stated in \cite[Appendix~B]{MR3729620}) from the
Ray--Knight isomorphism for the $\HH^{2|2}$ model.
The essence of the result is the following
corollary of Theorem~\ref{thm:H22-rk}.
Recall that we write $(z,x,y,\xi,\eta)\in \R^{3|2}$.

\begin{corollary}
  \label{cor:st-es}
  Let $[\cdot]_\beta$ be the superexpectation of the $\HH^{2|2}$
  model, and let $\E_{i,\V{\ell}}$ be the expectation of the VRJP.
  For $g\colon\R^{\Lambda}\to \R$ smooth and compactly
    supported, 
  \begin{equation}
    \label{eq:st-es}
    \lim_{\gamma \rightarrow \infty}\E_{a,\V{1}}\left(g(\frac{1}{\gamma}\V{L}_{\tau(\gamma)})\right)
    = \uspin{g(\V{z}+ \V{x})\delta_{u_{0}}(u_a)}_{\beta}
  \end{equation}
  where $\tau(\gamma) = \inf\{t | L_a^t \geq \gamma\}$ and
  $u_{0}=(1,0,0,0,0)$.
\end{corollary}
\begin{proof}
  We write $\gamma=\cosh s$. Then by Theorem~\ref{thm:H22-rk} applied
  to $g(\V{L}_{\tau(\cosh s)}/\cosh s)$, 
  \begin{align}
    \E_{a,\V{1}}\left( g( \frac{1}{\cosh s}\V{L}_{\tau(\cosh s}))\right)
    &= 
      \uspin{g(\frac{1}{\cosh s}\theta_s \V{z})\delta_{u_{0}}(u_a)}_{\beta}
      \nnb
    &= 
      \uspin{g(\V{z}+ \V{x}\tanh{s})\delta_{u_0}(u_a)}_{\beta},
  \end{align}
  by using~\eqref{eq:h22-boost} to compute $\theta_{s}\V{z} = \cosh s
  \V{z} + \sinh s \V{x}$.  Since
  $\tanh s \to 1$ as $s\to\infty$, by dominated convergence we obtain
  \begin{equation*}
    \lim_{s \rightarrow \infty}
    \E_{a,\V{1}}\left(g(\frac{1}{\cosh s}\V{L}_{\tau(\cosh s)})\right)
    = \uspin{g(\V{z}+ \V{x})\delta_{u_0}(u_a)}_{\beta}.
    \qedhere
\end{equation*}
\end{proof}

We now recover \cite[Theorem~2]{MR3420510} as stated
in~\cite[Theorem~B]{MR3904155}. Write $\log(\V{v}) =
(\log(v_i))_{i\in\Lambda}$. Applying Corollary~\ref{cor:st-es} to
a function $g\circ \log $ yields
\begin{equation}
    \label{eq:st-es-1}
    \lim_{\gamma \rightarrow \infty}\E_{a,\V{1}}\left(g(\log(\V{L}_{\tau(\gamma)}) - \log \gamma)\right)
    = \uspin{g(\log(\V{z}+ \V{x}))\delta_{u_{0}}(u_a)}_{\beta}
  \end{equation}
  where $\log \gamma = (\log \gamma)_{i\in\Lambda}$.
  To recover \cite[Theorem~2]{MR3420510} we
  rewrite the right-hand side of~\eqref{eq:st-es-1}. To do this,
  recall, e.g., from~\cite[Section~2.2]{MR2728731}, that horospherical
  coordinates for the $\HH^{2|2}$ model are given by the change of
  generators from $(x,y,\xi,\eta)$ to $(s,t,\psi,\bar\psi)$, where
  \begin{equation}
    \begin{split}
      x \bydef \sinh{t} - \frac{1}{2}(s^2 + 2\psi\psibar)e^t, \;\;\; y
      &\bydef se^{t},\;\;\;
      z \bydef \cosh{t} + \frac{1}{2}(s^2 + 2\psi\psibar)e^{t},\\
      \xi \bydef \psi e^{t},\;\;&\;\; \eta \bydef \psibar e^{t}.
    \end{split}
  \end{equation}
  Let
  \begin{equation}
    \label{eq:H1}
    H_1(\V{t}) \bydef 
    \frac{1}{2}\sum_{i,j\in\Lambda} \beta_{ij} (\cosh(t_i - t_j) - 1).
  \end{equation}
  The right-hand side of \eqref{eq:st-es-1} can be written
  explicitly in horospherical coordinates as
  \begin{equation} \label{e:SThoro} \uspin{g(\log(\V{z}+
      \V{x}))\delta(u_a)}_{\beta} =
    \frac{1}{\sqrt{2\pi}^{|\Lambda|-1}} \int_{
      \R^{|\Lambda|-1}}g(\V{t})e^{-H_1(\V{t})}\sqrt{\det{D(\beta,\V{t})}}
    \prod_{i\neq a} e^{-t_{i}}\,dt_{i},
  \end{equation}
  where $D(\beta,\V{t})$ is the $(|\Lambda|-1)\times(|\Lambda|-1)$
  matrix with entries
  \begin{equation}
    D_{ij}(\beta, \V{t}) \bydef
    \begin{cases} 
      -\beta_{ij}e^{t_i +t_j},\quad &i\ne j \\
      \sum_{k\ne a} \beta_{ik}e^{t_i + t_k} + \beta_{ai}e^{t_i} \quad
      &i = j
    \end{cases}
  \end{equation}
  indexed by $i,j \in \Lambda\backslash\{a\}$.
  This is \cite[Theorem~2]{MR3420510} as stated
  in~\cite[Theorem~B]{MR3904155}. In obtaining this formula we have
  used Theorem~\ref{thm:susy-cov} to perform the change of generators
  and then integrated out $s,\psi$ and $\bar\psi$, which can be done
  explicitly as conditioned on the $t$-variables these are Gaussian integrals,
  see~\cite[Section~2.3]{MR2728731}.

\begin{remark}
  \label{rk:SThoro}
  Qualitatively, the appearance of horospherical coordinates can
  be explained as follows.  The hyperbolic Ray--Knight isomorphism
  relates the time evolution of the VRJP by $\cosh s$ to the Lorentz
  boost by $s$ in the $xz$-plane.
  Since the asymptotics of Lorentz boosts in the
  $xz$-plane are captured by the
  $t$ marginal in horospherical coordinates,
  the formulation of the asymptotic local time distribution in terms
  of the $t$ marginal is quite geometrically natural.

  The Sabot--Tarr\`{e}s limit formula \cite[Theorem~2]{MR3420510} can
  also be derived from the hyperbolic BFS--Dynkin isomorphism.  More
  precisely, this can be done by using Corollary~\ref{c:Dynkin-fixedT}
  below, see \cite{SwanPhD}.  In this derivation the role of
  horospherical coordinates can be seen even more explicitly.

 For another explanation of the relation of horospherical
  coordinates to the VRJP, see \cite{MerklRollesTarres}.
\end{remark}

\section{Time changes and resolvent formulas}
\label{sec:TC-HK}

In this section we describe some useful variations and reformulations
of our theorems. For the sake of simplicity
we only consider the VRJP, but analogous results
also hold for the SRW and the VDJP.

\subsection{Time-changed and fixed-time formulas}
\label{sec:TC}

In the literature on the VRJP time changes
have played an important role; see, for example, \cite{MR3420510}.
For comparision with these references,
this section briefly explains how isomorphism theorems can be
translated to these time-changes.

For a Markov process $(X_{s},\V{L}_{s})$ on
$\Lambda \times \R^{\Lambda}$, let
$V\colon [\min_{ i\in\Lambda} L_0^i, \infty)\to [\min_{i\in\Lambda} V(L_0^i), \infty)$ be an increasing diffeomorphism and
define a random function $A\colon [0, \infty) \to [0, \infty)$ by
\begin{equation}
\label{e:Adef}
A(s) \bydef \int_{0}^{s} V'(L_{u}^{X_{u}})\,du = \sum_{i\in\Lambda}V(L_{s}^{i}) - V(L_0^i).
\end{equation}

We define $(\tilde X_{t},\V{\tilde L}_{t})$, the \emph{time-change by
  $V$} of $(X_{t},\V{L}_{t})$, by
\begin{equation}
\label{e:timechange}
\tilde X_{t}\bydef X_{A^{-1}(t)}, \quad \tilde L_t^i \bydef
V(L_{A^{-1}(t)}^i) = V(L_0^i) + \int_0^t 1_{\tilde X_u=i}\,du \,.
\end{equation}
Note that $A(0) = A^{-1}(0) = 0$, $\tilde X_{0}=X_{0}$ and
$\tilde L_0^i = V(L_{0}^i)$.  

In this section we will write
$V(\V{1}) \bydef (V(1))_{i\in\Lambda}$. The next corollary is an
example of an isomorphism theorem for a time-changed process.

\begin{corollary} 
  \label{c:DynkinV}
  Let $[\cdot]_\beta$ be the superexpectation of the $\HH^{2|2}$ model, and
  let $(\tilde X_t,\V{\tilde L}_t)$ be the time-change by $V$ of the VRJP
  with expectation $\E_{i,\V{\ell}}$.
  Then
  \begin{equation}
    \int_0^\infty \E_{a,V(\V{1})}(g(\tilde{X}_t,\V{\tilde L}_t)) \, dt
    =
    \sum_{i\in\Lambda} \q{x_a x_i V'(z_i) \, g(i,V(\V{z}))}_{\beta}
    .
  \end{equation}
\end{corollary}

\begin{proof}
  By \eqref{e:timechange} and the change of variable $s = A^{-1}(t)$,
  \begin{align}
    \int_{0}^{\infty} \E_{\tilde X_0, \V{\tilde L}_0}(g(\tilde X_t, \V{\tilde L}_t))\,dt 
    &= \int_{0}^{\infty} \E_{X_0, \V{L}_0}\left(g(X_{A^{-1}(t)}, V(\V{L}_{A^{-1}(t)}))\right)\,dt
      \nnb
    &= \int_{A^{-1}(0)}^{A^{-1}(\infty)} \E_{X_0, \V{L}_0}\left(g(X_{s}, V(\V{L}_{s})) A'(s)\right)\,ds
      \nnb
    &= \int_{0}^{\infty} \E_{X_0, \V{L}_0}\left(g(X_{s}, V(\V{L}_{s}))
      V'\left(L_{s}^{X_s}\right)\right)\,ds.
  \end{align}
  The claim now follows from Theorem~\ref{thm:SUSY-BHS} in the
  case that $g(i,\V{\ell})$ is of the form $\delta_{i,j}f(\V{\ell})$.
  The result for
  more general functions follows by summing (or by using the second
  part of Lemma~\ref{lem:gen-SUSY-H22}).
\end{proof}

The next corollary shows that supersymmetric isomorphism
theorems also give formulas for the local time distribution at
\emph{fixed} times.

\begin{corollary} 
  \label{c:Dynkin-fixedT}
  Let $[\cdot]_\beta$ be the superexpectation of the $\HH^{2|2}$ model, and
  let $(\tilde X_t,\V{\tilde L}_t)$ be the time-change by $V$ of the VRJP
  with expectation $\E_{i,\V{\ell}}$.
  Let $\delta_\epsilon\colon \R \to \R$ be a smooth and compactly supported
  approximation to $\delta_{0}$. 
  Then for $g\colon \R^{\Lambda}\to \R$ smooth and rapidly decaying
  and any $T>0$,
  \begin{equation*} \label{e:Dynkin-fixedT}
    \E_{a,V(\V{1})}g\left(\V{\tilde L}_T-\frac{T}{N}\right)
    =
    \lim_{\epsilon\to 0}\sum_{i\in\Lambda} \qa{x_ax_i V'(z_i) g\left(V(\V{z})-\frac{T}{N}\right) \delta_\epsilon\left(\sum_{i\in\Lambda} \left(V(z_i)-V(1)-\frac{T}{N}\right)\right)}_{\beta}.
  \end{equation*}
\end{corollary}
\begin{proof}
  The left-hand side can be written as
  \begin{align}
    \E_{a,V(\V{1})}(g(\V{\tilde L}_T-\frac{T}{N}))
    &=
      \sum_{i\in\Lambda} \E_{a,V(\V{1})}(g(\V{\tilde L}_T-\frac{T}{N})1_{X_T=i})
      \nnb
    &=
      \lim_{\epsilon\to 0}\sum_{i\in\Lambda} \int_0^\infty dt \,
      \E_{a,V(\V{1})}\pB{g(\V{\tilde L}_t-\frac{T}{N})1_{X_t=i}} \,
      \delta_\epsilon(t-T) 
      \nnb
    &=
      \lim_{\epsilon\to 0}\sum_{i\in\Lambda} \int_0^\infty dt \,
      \E_{a,V(\V{1})}\pB{g(\V{\tilde L}_t-\frac{T}{N})  1_{X_t=i} \delta_\epsilon(\sum_{i\in\Lambda} (\tilde L_t^i - V(1))-T)}
      \nnb
    &=
       \lim_{\epsilon\to 0}\sum_{i\in\Lambda} \qa{ x_a x_i V'(z_i) g\pa{V(\V{z})-\frac{T}{N}} \delta_\epsilon\pa{\sum_{i\in\Lambda} (V(z_i)-V(1)-\frac{T}{N})} }_\beta.
  \end{align}
  The second equality used that $t\mapsto
  \E_{a,V(\V{1})}(g(\V{\tilde L}_t-T/N)1_{X_t=i})$ is continuous,
  the third equality used that $\sum_{i\in\Lambda} (L_t^i - V(1))=t$ for any $t\geq 0$,
  and the fourth equality is Corollary~\ref{c:DynkinV}.
\end{proof}

By making use of an appropriate time-change, 
Corollary~\ref{c:Dynkin-fixedT} is the starting point for an
alternative derivation of the Sabot--Tarr\`{e}s limit formula \eqref{e:SThoro}, 
see Remark~\ref{rk:SThoro}.  Similar results have also been used as the
starting point for the study of large deviations of the local time of the
SRW~\cite[Theorem~1]{MR2330973}.

\subsection{Resolvent of the joint local time process}
\label{sec:HK}

The supersymmetric isomorphism theorems for the VRJP in
Section~\ref{sec:susy-hyperb-model} concern fixed initial local times for the joint
process $(X_t,\V{L}_t)$, i.e., $\V{L}_0=\V{1}$.
This initial condition arises from supersymmetric localisation at 
$(z,x,y,\xi,\eta)=(1,0,0,0,0)$ due to the sigma model constraint
$u\cdot u =-1$.  A more general and geometrically instructive
formulation can be obtained by considering the joint process
$(X_t, \V{L}_t)$ with a general initial condition. This
formulation involves the super-Minkowski space from
Section~\ref{sec:super-min} as opposed to the space $\HH^{2|2}$. 

\subsubsection{Super-Minkowski model}

Recall super-Minkowski space $\R^{3|2}$ from
Section~\ref{sec:super-min}.  We define the Berezin integral for an
observable $F \in \Omega^{2\Lambda}(\R^{3\Lambda})$ by
\begin{equation}
\label{eq:R32int}
\int_{(\R^{3|2})^{\Lambda}}F \bydef \frac{1}{(2\pi)^{|\Lambda|}} \int\, d\V{x}\,d\V{y}\,d\V{z}\, \partial_{\V{\eta}}\,\partial_{\V{\xi}}\;F,
\end{equation}
where $\partial_{\V{\eta}}\, \partial_{\V{\xi}}$ is defined by
$\partial_{\eta_{|\Lambda|}}\partial_{\xi_{|\Lambda|}}\dots \partial_{\eta_{1}}\partial_{\xi_{1}}$,
$d\V{x} = dx_{|\Lambda|}\dots dx_{1}$,
$d\V{y}=dy_{|\Lambda|}\dots dy_{1}$, and
$d\V{z}=dz_{|\Lambda|}\dots dz_{1}$ for some fixed ordering of the
$i\in\Lambda$ from $1$ to $|\Lambda|$.

For $u \in \R^{3|2}$, we write $u\cdot u< 0$ if the degree $0$ part
of the form $u\cdot u$ is negative,
where here $u\cdot u$ denotes the super-Minkowski inner product \eqref{eq:hypprod}.
For a spin configuration
$\V{u} \in (\R^{3|2})^\Lambda$ we write $\V{u} \cdot \V{u} < 0$ if
$u_i\cdot u_i< 0$ for all $i\in\Lambda$ and we then define
\begin{equation}
  r_i {\bydef} 
  \sqrt{-u_i\cdot u_i},
\end{equation}
and let $\V{r} = (r_{i})_{i\in\Lambda}$.
For such a spin configuration we consider the Hamiltonian
\begin{equation}
\label{eq:H22-unconstrained}
H_{\beta}(\V{u})
\bydef
\frac{1}{2}(\V{u},-\Delta_{\beta}\V{u})
+\frac{1}{2}(\V{r},-\Delta_{\beta}\V{r}),
\end{equation}
where the inner product for the $u_i$ is the one from \eqref{eq:hypprod}
and the $r_i$ are forms that are multiplied in the ordinary way:
$(\V{r},-\Delta_{\beta}\V{r}) = \sum_{i\in\Lambda}r_{i}(-\Delta_{\beta}r)_{i}$. 
Let $F\in \Omega^{2\Lambda}(\R^{3\Lambda})$ be a smooth form compactly
supported on $\{\V{u\cdot u}<0,\V{z}>0\}$, i.e., whose coefficient functions
vanish when the degree $0$ part of any form $u_{i}\cdot u_{i}$ is
non-negative or when $z_{i}\leq 0$ for any $i$.
We define an unnormalised superexpectation by
\begin{equation}
\label{eq:SUSY-Uhyp2}
[F]_\beta \bydef \int_{(\R^{3|2})^\Lambda} 
F(\V{u}) e^{-H_\beta(\V{u})} 1_{\V{u}\cdot \V{u}<0} 1_{\V{z}>0},
\end{equation}
with $\V{u}\cdot \V{u}<0$ as defined above. The assumption that
$F$ has compact support ensures the integrand is smooth.
We call this the \emph{super-Minkowski model}.
Note that $\{\V{u}\cdot \V{u} <0,\V{z}>0\}$ is a version of the
causal future for super-Minkowski space; see
Figure~\ref{fig:minkowski}.

\subsubsection{Symmetries and localisation}

Let
\begin{equation}
  T = x \ddp{}{z} + z \ddp{}{x}.
\end{equation}
Then $T$ represents an infinitesimal Lorentz boost in the $xz$-plane since
\begin{equation}
  T x = z, \qquad Tz = x, 
\end{equation}
and $T y = T \xi = T \eta = 0$. Note also that $T r = 0$.

The Hamiltonian $H_{\beta}$ is invariant under $T$, i.e.,
$\sum_{i\in\Lambda}T_i H_\beta(\V{u}) = 0$.  Here we have written
$T_i$ for the version of $T$ applying to the $i$-th coordinate.
Moreover the integral \eqref{eq:R32int} is invariant under $T$.  In
addition, the model is supersymmetric with supersymmetry generator $Q$
as in \eqref{eq:susy-gen}, and the following localisation statement
holds for all smooth $f\colon (0,\infty)^{2\Lambda} \to \R$ with
compact support:
\begin{equation} \label{e:Minkowski-localisation}
  [f(\V{z},\V{r})]_\beta = \int_{(0,\infty)^{\Lambda}} d\V{z}\, f(\V{z}, \V{z}).
\end{equation}
This can be seen by integrating over $\V{z}$ last when computing
the superexpectation, and using localisation for $(x,y,\eta,\xi)$, i.e.,
Corollary~\ref{cor:localisation}.

\subsubsection{Resolvent formula}
The super-Minkowski model is related to the resolvent of the VRJP.

\begin{theorem} \label{thm:VRJPresolvent} Let $[\cdot]_{\beta}$ be the
  superexpectation of the super-Minkowski model, and let
  $\pi = (\pi(i,\V{r}))$ be a smooth compactly support probability
  measure on $\Lambda \times (0,\infty)^\Lambda$.
  For all smooth $f\colon\Lambda\times\R^{\Lambda}\to\R$ with rapid decay,
  \begin{equation}
    \label{e:VRJPresolvent}
    \int_{0}^{\infty} \E_{\V{\pi}} f(X_{t},\V{L}_{t})\, dt
    =
    \sum_{i,j\in\Lambda}\cb{\frac{\pi(i,\V{r})}{r_i}x_{i}x_{j}  f(j,\V{z})}_{\beta}
  \end{equation}
  where we have written $\E_{\V{\pi}}$ to denote the expectation of
  a VRJP with initial condition $(X_{0},\V{L}_{0})$ distributed according to $\pi$.
\end{theorem}

\begin{remark}
  \label{rem:resolve-compact}
  In the notation of Remark~\ref{rem:srw-compact},
  Theorem~\ref{thm:VRJPresolvent} can be compactly rewritten as
  \begin{equation}
    \label{e:VRJPresolvent1}
    \int_{0}^{\infty} \E_{\V{\pi}} f(X_{t},\V{L}_{t})\, dt
    =
    \cb{(\V{x}, \frac{\V{\pi}(\V{r})}{\V{r}})(\V{x},
      \V{f}(\V{z}))}_{\beta}.
  \end{equation}
\end{remark}

The proof of Theorem~\ref{thm:VRJPresolvent} uses that
Lemma~\ref{lem:gen-SUSY-H22} remains true if $[\cdot]_\beta$ is
interpreted as the expectation of the super-Minkowski model,
and then follows the standard route as follows.

\begin{proof}
  Let
  $\rho(\V{u})\bydef \sum_{i\in\Lambda} \pi(i,\V{r})x_i/r_i$, 
  and let $T_{i}$ be the infinitesimal boost given by \eqref{eq:H22-boost}.
  Since $T_ir_i=0$ and $T_ix_i =z_i$ we have $T_j \rho = \pi(j,\V{r})z_j/r_j$.
  Since Lemma~\ref{lem:gen-SUSY-H22} holds
  for the super-Minkowski model,
  we apply \eqref{eq:generic-iso-susy-H22} to obtain
  \begin{equation}
    \sum_{i,j\in\Lambda}\uspin{\frac{\pi(i,\V{r})}{r_i}x_ix_{j} f(j, \V{z})}_{\beta}
    =
    \sum_{j\in\Lambda} \uspin{\frac{z_j}{r_j}\pi(j,\V{r})\int_0^\infty \E_{j,\V{z}}(f(X_t,\V{L}_t))\,dt}_{\beta}
    .
  \end{equation}
  By localisation, i.e., \eqref{e:Minkowski-localisation}, the right-hand side equals
  \begin{equation*}
    \int_{\R_+^\Lambda} d\V{z} \, \sum_{j\in\Lambda}\pi(j,\V{z})
    \int_0^\infty \E_{j,\V{z}}(f(X_t,\V{L}_t))\,dt
    =
    \int_0^\infty \E_{\V{\pi}}(f(X_t,\V{L}_t))\,dt.
    \qedhere
  \end{equation*}
\end{proof}

\section{Application to exponential decay of correlations in spin
  systems} 
\label{sec:applications}

In this section we prove theorems about the exponential decay of
spin-spin correlations. Let $d(i,j)$ denote the graph
distance between vertices $i$ and $j$ in the graph induced by the
edge weights $\beta$; this distance is finite since the induced
graph is finite and connected by assumption.

We first consider the $\HH^{2|2}$ model with constant and non-zero
external field.
\begin{theorem}
  \label{thm:exp-decay-H}
  Consider the $\HH^{2|2}$ model with
    $\sup_{i\in\Lambda} \sum_{j\in\Lambda} \beta_{ij} \leq \beta_*$ 
  and $h_{i}=h>0$ for all $i\in\Lambda$. Let $c(\beta_*,h) \bydef
  \log(1+h/\beta_*)$. Then for all $i,j\in \Lambda$,
  \begin{equation}
    \cb{x_ix_j}_{\beta,h} \leq \frac{1}{h} e^{-c(\beta_*,h)d(i,j)}.
  \end{equation}
\end{theorem}

\begin{proof}
  Let $\tau_{j}$ be the hitting time of $j$, i.e.,
  $\tau_{j} \bydef \inf \{s\geq 0 \mid X_{s}=j\}$.
  Then by choosing $g$ an exponential in Theorem~\ref{thm:SUSY-BHS},
  \begin{equation}
      \label{eq:edH1}
    \cb{x_ix_j}_{\beta,h}
    =\E_{i,\V{1}} \int_{0}^{\infty}1_{X_{s}=j}e^{-hs}\,ds
      =\E_{i,\V{1}}1_{\tau_{j}<\infty}\int_{\tau_{j}}^{\infty}1_{X_{s}=j} e^{-hs}\,ds
       \leq \frac{1}{h} \P_{i,\V{1}}(\tau_{j}<\infty).
  \end{equation}
  The inequality follows as the integral is at most
  $\int_{0}^{\infty}e^{-hs}ds = h^{-1}$.

  If $\tau_{j}<\infty$ then there are at least $d(i,j)$ times at which
  a rate $h$ exponential clock does \emph{not} ring before a rate
  $\beta_*$ clock, as there are at least $d(i,j)$ jumps to previously
  unvisited vertices on any path from $i$ to $j$. The probability of a
  rate $h$ clock ringing only after some rate $\beta_{ij}$ clock is
  at most $\beta_*/(\beta_*+h)$. Each of these events are independent by the
  memorylessness of the exponential, and hence
  \begin{equation*}
    \P_{i,\V{1}}(\tau_{j}<\infty) \leq (\frac{\beta_*}{\beta_*+h})^{d(i,j)}
    = e^{-c(\beta_*,h)d(i,j)}
  \end{equation*}
  Combined with \eqref{eq:edH1} this proves the theorem.
\end{proof}

\begin{remark}
  Theorem~\ref{thm:exp-decay-H} gives a positive rate
  $\log(1+h/\beta_*)\sim ch$ of exponential decay for some $c>0$ for
  any value of $\beta$. For small $\beta$, i.e., high temperatures, it
  is known that the rate stays uniformly bounded away from $0$ as
  $h\downarrow 0$ \cite{MR2736958,MR3189433}.  The rate is expected to
  be bounded away from $0$ for any $\beta$ when the graph $\Lambda$
  tends to $\Z^2$.  On the other hand, for $\Lambda \uparrow \Z^d$
  with $d\geq 3$ it is conjectured that the rate behaves
  asymptotically as $\sim c\sqrt{h}$ as $h\downarrow 0$.
\end{remark}

It would be interesting to obtain an analogue of
Theorem~\ref{thm:exp-decay-H} for the $\HH^{n}$ model by using
Theorem~\ref{thm:BHS}. This would require an appropriate estimate on
the $z$-field to control the initial local times of the VRJP. We
do not pursue this direction here.

For the hemispherical spin models, the estimates on the $z$-field are trivial
because $|z_i|\leq 1$, and we thus consider both the $\bbS^{n}_{+}$ model and the $\bbS^{2|2}_{+}$ model.
For $\bbS^{2|2}_+$ we have only defined the
superexpectation of compactly supported observables. To define the
superexpectation of non-compactly supported observables requires a
treatment of superintegrals with boundaries; since we do not need this
general treatment we instead define the two-point function
$\cb{{x_{i} x_{j}}}_{\beta,h}$ for the $\bbS^{2|2}_{+}$ model by
$\cb{x_{i} x_{j}}_{\beta,h} \bydef \lim_{n\to\infty} \cb{x_{i} x_{j}
  f_{n}(\V{z})}_{\beta,h}$ where $f_{n}$ is a sequence of smooth and
bounded approximations to $1_{\V{z}>0}$.  The proof of the following
theorem shows that this limit exists.

\begin{theorem}
  \label{thm:decay-hemi}
  Consider the $\bbS^{n}_{+}$  model with
  $\sup_{i\in\Lambda}\sum_{j\in\Lambda}\beta_{ij} \leq \beta_{*}$,
  and let $c(\beta_*) = \log (1-e^{-\beta_*})$.
  Then for all $i,j\in\Lambda$,
  \begin{equation}
    \label{eq:decay-hemi}
    \avg{x_{i}x_{j}}_{\beta,h}
    \leq e^{-c(\beta_*)d(i,j)}.
  \end{equation}
  The same result holds for the superexpectation
  $\cb{x_{i}x_{j}}_{\beta,h}$ of the $\bbS^{2|2}_{+}$ model.
\end{theorem}
\begin{proof}
  We first consider $\bbS^{2|2}_{+}$. Let $f_{n}$ be a sequence of
  smooth and bounded approximations to $1_{\V{z}>0}$. Letting
  $\E_{i,\V{1}}$ be the expectation for a VDJP with initial local time
  $\V{1}$, Theorem~\ref{thm:SUSY-hem} implies
  \begin{equation*}
    \cb{x_{i}x_{j}}_{\beta,h}
    = \lim_{n\to\infty} \cb{x_{i}x_{j} f_{n}(\V{z})}_{\beta,h}
    = \lim_{n\to\infty }\E_{i,\V{1}}\int_{0}^{\infty}f_{n}(\V{L})1_{X_{t}=j}e^{-\sum_{v}h_{v}L^{t}_{v}}\, dt,
  \end{equation*}
  To obtain upper bounds we may assume, without loss of generality,
  that $\V{h}=\V{0}$.
  By definition, $X_t$ dies once the local time at any vertex reaches
  $0$.  Since $f_{n}$ is asymptotically bounded above by one,
  it therefore suffices to bound the probability that $X_{t}$
  reaches $j$.

  By the definition of the VDJP, for each $r\in\Lambda$ the jump rate
  out of $r$ is bounded above by $\beta_*$.
  Thus for each
  $k\in\N$ there is probability at least $e^{-\beta_*}$ 
  the
  walk $X_{t}$ dies after its $k$th jump and before its $(k+1)$st
  jump. The probability $X_{t}$ reaches $j$ is at most the probability
  that $X_{t}$ does not die before taking $d(i,j)$ steps, and hence
  \begin{equation*}
    \cb{x_{i}x_{j}}_{\beta,h}
    \leq (1-e^{-\beta_*})^{d(i,j)}
    = e^{-c(\beta_*)d(i,j)}.
  \end{equation*}
  This completes the proof for $\bbS^{2|2}_{+}$.  For $\bbS^{n}_{+}$,
  we use (the normalised form of) Theorem~\ref{thm:BHS-hemi} in
  place of Theorem~\ref{thm:SUSY-hem}. The argument above applies
  pointwise in the initial local time, so using $0\leq z_{i}\leq 1$ we
  obtain the same conclusion.
\end{proof}

\begin{remark}
  \label{rem:exp-dec-mcb}
  A result closely related to Theorem~\ref{thm:decay-hemi} is given
  in~\cite[Theorem~2]{MR0441179}.
\end{remark}

\appendix

\section{Introduction to supersymmetric integration}
\label{sec:SUSY-intro}

This appendix gives a self-contained introduction to the mathematics
of supersymmetry that is relevant for this article.  For complementary
treatments, see in particular \cite{MR914369,MR1843511,MR2525670}.  In
Appendix~\ref{sec:fSUSY} we discuss some further aspects of
supersymmetry that are relevant to this article, but that are not
needed to understand the main text.

\subsection{Integration of differential forms}
\label{sec:IDF}

We begin by reviewing the important example of integration of
differential forms on Euclidean space $\R^N$. Let $x_{1},\dots, x_{N}$
be coordinates on $\R^{N}$.  A \emph{differential form} on $\R^N$ can
be written as
\begin{equation}
  \label{eq:F-exp}
  F = F_0 + \cdots + F_N
\end{equation}
where $F_0 \in C^\infty(\R^N)$ is a \emph{$0$-form,} i.e., an ordinary
function, and $F_p$ is a \emph{$p$-form}, i.e., a nonzero sum of terms of the form
\begin{equation}
  f_{i_{1},\dots, i_{p}}(x_1,\dots, x_N) \, dx_{i_1} \wedge \cdots \wedge
  dx_{i_p}, \qquad 1\leq i_{j}\leq N, \;\; 1\leq j\leq p,
\end{equation}
where $f_{i_{1},\dots, i_{p}} \in C^\infty(\R^N)$,
the coordinates are viewed as functions $x_i\colon \R^N \to \R$ in $C^\infty(\R^N)$,
and the differentials $dx_i$
are the generators of a Grassmann algebra. 
This means that the $dx_{i}$ are formal variables that are
multiplied with the anti-commuting wedge product:
\begin{equation} \label{e:dx-anti}
  dx_i \wedge dx_j = -dx_j \wedge dx_i.
\end{equation}
In particular, $dx_i \wedge dx_i = 0$.
Later, the $\wedge$ will often be omitted. By extending the wedge
product to differential forms by linearity, we obtain a unital
associative algebra over $C^{\infty}(\R^{N})$. This is the exterior algebra
of differential forms on $\R^N$, which we denote $\Omega(\R^N)$.

The form $F_p$ in \eqref{eq:F-exp} is the \emph{degree
  $p$} part of $F$. We say $F$ has \emph{degree $p$} or
\emph{is a $p$-form} if $F=F_p$.
Since $dx_{i}\wedge dx_{i}=0$,
there are no forms of degree greater than $N$. A form $F$ of
degree $N$ is said to be of \emph{top degree} and such an $F$
can be written as
\begin{equation}
  \label{eq:td}
  F(\V{x}) = f(\V{x}) \, dx_1 \wedge \cdots \wedge dx_N
\end{equation}
for some $f\in C^{\infty}(\R^{N})$, where we abbreviate
$\V{x} = (x_1,\dots, x_N)$.  The anticommutativity of the wedge
product implies that the order of the differentials determines an
overall sign in \eqref{eq:td}. Keeping this in mind, the integral of a
top degree form $F$ is defined by
\begin{equation}
  \int_{\R^N} F \bydef \int_{\R^N} f(\V{x}) \, dx_1 \cdots dx_N
\end{equation}
where the right-hand side is an ordinary integral with respect to
Lebesgue measure.  For $p<N$ the integral of a $p$-form $F_{p}$ is
defined to be zero: $\int_{\R^{N}} F_{p} \bydef 0$.  Having defined
the integral on $p$-forms for all $p$, we extend the definition of the
integral to the entire algebra $\Omega(\R^{N})$ of differential forms
by linearity.

\begin{example}[Change of variables]
\label{ex:CoV-df}
The differential notation and the use of the wedge product
is consistent with, and motivated by, the following change of variable formula.
Let $\Phi\colon \R^N\to\R^N$ be an orientation preserving
diffeomorphism. Then by the change of variables formula from
calculus
\begin{align}
  \int f(x_1, \dots, x_N) \, dx_1 \wedge \cdots \wedge dx_N
  &=
  \int f(\Phi_1(\V{x}), \dots, \Phi_N(\V{x})) (\det D\Phi) \, dx_1 \wedge \cdots \wedge dx_N
  \nnb
  &=
    \int f(\Phi_1(\V{x}), \dots, \Phi_N(\V{x})) \, d\Phi_1(\V{x}) \wedge \cdots \wedge d\Phi_N(x)
\end{align}
where $D\Phi$ is the Jacobian matrix of $\Phi$ and the second
equality has made use of the definition
\begin{equation}
  d\Phi_i(\V{x}) = \sum_{j=1}^{N} \ddp{\Phi_i(\V{x})}{x_j} dx_j,
\end{equation}
which leads, by a calculation, to the identity
\begin{equation}
  d\Phi_{1}(\V{x})\wedge \cdots \wedge d\Phi_{N}(\V{x})
  = (\det D\Phi) \;
  dx_{1}\wedge \cdots \wedge dx_{N}.
\end{equation}
\end{example}

\subsection{Odd and even forms}
A differential form is \emph{even} if it is a sum of $p$-forms with
all $p$ even and it is \emph{odd} if it is a sum of $p$-forms with all
$p$ odd. We say a form is \emph{homogeneous} if it is either even or
odd. We can decompose a general form $F$ as
\begin{equation}
  F = F_{\mathrm{even}} + F_{\mathrm{odd}},
  \qquad 
  \Omega(\R^N) = \Omega_{\mathrm{even}}(\R^N) \oplus \Omega_{\mathrm{odd}}(\R^N),
\end{equation}
where $F_{\mathrm{even}}$ is the sum of the degree
$p$ parts of $F$ with $p$ even, and similarly for $F_{\mathrm{odd}}$.
As the wedge product of a $p$-form with a $q$-form is either $0$
or a $(p+q)$-form, the exterior algebra
equipped with the wedge product is a $\Z_{2}$-graded algebra.
$\Z_{2}$-graded algebras are also called \emph{superalgebras}.
Formally, this means that if we define the parity of a homogeneous form as
\begin{equation}
  \alpha(F) \bydef \begin{cases} 0 \in \Z_2, \qquad F = F_{\mathrm{even}}\\
    1 \in \Z_2,\qquad F = F_{\mathrm{odd}}
  \end{cases}
\end{equation}
then $\alpha(F\wedge G) = \alpha(F) + \alpha(G)$ mod 2. 
A calculation shows that for homogeneous $F,G$
\begin{equation}
  F\wedge G = (-1)^{\alpha(F)\alpha(G)}G\wedge F,
\end{equation}
and in particular, even elements commute
with all other elements by linearity.

\subsection{Berezin integral}
\label{sec:BerezinInt}

In this section we introduce Grassmann algebras and the Berezin
integral. Integration of differential forms as introduced in the
previous sections constitute a special case.

\subsubsection{Grassmann algebras}
\label{sec:grassmann-algebras}

Let $\Omega^M$ be a Grassmann algebra with generators
$\xi_1, \dots, \xi_M$; as the subscripts suggest we will
always assume there is a fixed (but arbitrary) order on the
generators. Thus $\Omega^M$ is the unital associative algebra
generated by the $(\xi_i)_{i=1}^{M}$ subject to the anticommutation
relations
\begin{equation}
  \xi_i \xi_j + \xi_j\xi_i = 0, \quad 1\leq i\leq j\leq M.
\end{equation}
Let $\Omega^M(\R^N)$ be the algebra over $C^{\infty}(\R^{N})$
generated by the $(\xi_{i})_{i=1}^{M}$. Elements of this algebra can
be written as
\begin{equation}
  \label{eq:b-gf}
  \mathop{\sum_{I\subset \{1,\dots,M\}}}_{I = \{i_{1},\dots, i_{p}\}}
  f_{I}(\V{x}) \, \xi_{i_{1}}\cdots\xi_{i_{p}} 
\end{equation}
where $f_{I}\in C^{\infty}(\R^{N})$ for each
$I\subset \{1,\dots, M\}$, and we have arranged the product of
generators according to the given fixed order:
$i_{1}<i_{2}<\dots <i_{p}$.

\begin{example}
  The differentials $\xi_{i} = dx_i$
  are an instance of a Grassmann algebra, and the algebra of
  differential forms on $\R^N$ can be identified with $\Omega^N(\R^N)$.
\end{example}

We continue to use the term \emph{form} for elements of
$\Omega^M(\R^N)$ when $N\neq M$.  The notion of the degree of a form
and the $\Z_{2}$-grading that we defined for differential forms
extends to this more general context.

\subsubsection{Integration}
\label{sec:integration-forms}

For $i\in \{1,2,\dots, M\}$ the\emph{ left-derivative} $\ddp{}{\xi_i}\colon \Omega^M \to \Omega^M$
is the unique linear map determined by
\begin{equation}
  \ddp{}{\xi_i} (\xi_iF) = F \quad \text{if $\xi_i F \neq 0$}, \qquad \ddp{}{\xi_i} 1 = 0.
\end{equation}
We sometimes write $\partial_{\xi_i} = \ddp{}{\xi_{i}}$. Note that
$\partial_{\xi_i}$ is an \emph{anti-derivation}: if $F$ is a
homogeneous form, then
\begin{equation}
  \partial_{\xi_i} (FG) = (\partial_{\xi_i}F)G + (-1)^{\alpha(F)} F (\partial_{\xi_i}G).
\end{equation}
The left-derivative extends naturally to an anti-derivation on
$\Omega^M(\R^N)$ by defining
\begin{equation}
\partial_{\xi_{i}} ( f(\V{x})\xi_{i_{1}}\dots \xi_{i_{p}})
= f(\V{x}) \partial_{\xi_{i}} (\xi_{i_{1}}\dots \xi_{i_{p}}).
\end{equation}

\begin{example}
  The left-derivative gives a convenient formulation of the integral
  of a differential form. Let $F \in \Omega^N(\R^N)$ be a differential
  form and write $\xi_i = dx_i$. Then
  \begin{equation}
    \label{eq:exint}
    \int F = \int_{\R^N} dx_1 \cdots dx_N \, \partial_{\xi_N} \, \cdots\, \partial_{\xi_1} \, F
    = \int_{\R^N} d\V{x} \, \partial_{\V{\xi}} \, F
  \end{equation}
  where the left-hand side is the integral as a differential form in
  the sense of Section~\ref{sec:IDF}, and the last equality made use
  of the definition
  $\partial_{\V{\xi}} \bydef \partial_{\xi_{N}}\dots \partial_{\xi_{1}}$.
  Note that the order used in defining $\partial_{\V{\xi}}$ matters.
\end{example}

The notation on the right-hand side of \eqref{eq:exint} is called the
\emph{Berezin integral}.  This is a useful notion
because it is possible to change variables in $\V{x}$ and $\V{\xi}$
separately, as will be discussed below in Section~\ref{sec:cov}.  The
Berezin integral generalises to $N \neq M$ as follows.
\begin{definition}
  \label{def:BerezinInt}
  For $F\in\Omega^{M}(\R^{N})$, the \emph{Berezin integral of $F$} is
  \begin{equation}
    \label{eq:BerezinInt}
    \int F \bydef \int_{\R^{N}}dx_{1}\cdots dx_{N}
    \, \partial_{\xi_{M}}\cdots\partial_{\xi_{1}} \, F
    = \int_{\R^{N}} d\V{x} \, \partial_{\V{\xi}} \, F,
  \end{equation}
  where the last equality is by the definitions
  $d\V{x} = dx_{1}\dots dx_{N}$ and
  $\partial_{\V{\xi}} \bydef \partial_{\xi_{M}}\dots \partial_{\xi_{1}}$.  We say a
  form $F$ is \emph{integrable} if it can be written as a finite sum of forms
  of the type $f(\V{x}) \, \xi_{i_1}\dots\xi_{i_p}$ with $f$
  integrable on $\R^N$.
\end{definition}

The expression $d\V{x}\, \partial_{\V{\xi}}$ on the right-hand side
of \eqref{eq:BerezinInt} is an example of a \emph{superintegration form}.
More generally a superintegration form is given by
$d\V{x} \, \partial_{\V{\xi}}\, F$ for $F$ an even integrable form, and
integration with respect to this superintegration form is defined by
$\int G = \int_{\R^{N}} d\V{x}\,\partial_{\V{\xi}} \, FG$.

\subsubsection{Functions of forms}
\label{sec:functions-forms}

Suppose $g\in C^{\infty}(\R^{k})$. We will use
$\alpha = (\alpha_{1},\dots,\alpha_{k})$ to denote multiindices,
and we will also use the notation
\begin{equation*}
  g^{(\alpha)}(\V{x})
  \bydef \ddp{}{x_1^{\alpha_{1}}}\dots
  \ddp{}{x_k^{\alpha_{k}}}g(x), \qquad x^{\alpha} \bydef
  x_{1}^{\alpha_{1}}\cdots x_{k}^{\alpha_{k}}.
\end{equation*}
\begin{definition}
  Let $g \in C^\infty(\R^k)$ and $F^1, \dots F^k \in \Omega^M(\R^N)$
  be even forms. Then $g(F^1, \dots, F^k) \in\Omega^M(\R^N)$ is
  defined by the following formula, where the sum runs over all
  multiindices $\alpha$:
  \begin{equation}
    \label{eq:fof}
    g(F^1, \dots, F^k)
    \bydef \sum_{\alpha} \frac{1}{\alpha!} g^{(\alpha)}(F^1_0, \dots, F^k_0) (F-F_0)^\alpha.
  \end{equation}
\end{definition}
Note that the product defining $(F-F_{0})^{\alpha}$ is the wedge
product, i.e., this is shorthand for
$(F^{1}-F^{1}_{0})^{\alpha_{1}}\wedge \dots \wedge
(F^{k}-F^{k}_{0})^{\alpha_{k}}$, and $(F^{1}-F^{1}_{0})^{\alpha_{1}}$
is the $\alpha_{1}$-fold wedge product of this form with itself. There
is no ambiguity in the ordering since all forms are assumed even.  The
formal Taylor expansion in \eqref{eq:fof} is finite because forms of
degree greater than $N$ do not exist.  As a simple example of a
function of a form, the reader may wish to verify that
\begin{equation}
  e^{-x_1^2-\xi_1\xi_2} = e^{-x_1^2}(1-\xi_1\xi_2).
\end{equation}

\subsection{Gaussian integrals and localisation}

Let $A \in \R^{N\times N}$ be positive definite.  The $O(2)$-invariant
Gaussian measure on $\R^{2N}$ associated to the matrix $A$ has density
\begin{equation} 
  \label{e:Gauss-density}
  e^{-\frac12(\V{x},A\V{x})-\frac12(\V{y},A\V{y})} (\det A) \prod_{i=1}^N \frac{dx_i \, dy_i}{2\pi}.
\end{equation}
Let $\xi_1, \dots, \xi_N, \eta_1, \dots, \eta_N$ be generators of the
Grassmann algebra $\Omega^{2N}$, and define
  \begin{equation}
    \label{eq:det-G-pre}
    \partial_{\V{\eta}} \partial_{\V{\xi}}
    \bydef \partial_{\eta_N} \partial_{\xi_N}
    \cdots \partial_{\eta_1} \partial_{\xi_1}  \qquad
    (\V{\xi},A\V{\eta}) \bydef \sum_{i=1}^N A_{ij} \xi_i\eta_j.
  \end{equation}
A computation shows that
\begin{equation} \label{e:det-Grassmann}
  \partial_{\V{\eta}} \partial_{\V{\xi}} e^{(\V{\xi},A\V{\eta})}
  = \partial_{\V{\eta}} \partial_{\V{\xi}} \frac{1}{N!}
  (\sum_{i=1}^{N}A_{ij}\xi_{i}\eta_{j})^{N} = 
  \det A .
\end{equation}
\begin{remark}
  \label{rem:GG}
  The form
  $e^{(\V{\xi},A\V{\eta})} = e^{\frac12(\V{\xi},A\V{\eta})-\frac12
    (\V{\eta},A\V{\xi})} \in\Omega^{2N}$ is called a
  \emph{Grassmann Gaussian}.  The corresponding Grassmann Gaussian
  expectation $\avg{F} \bydef \q{F}/\q{1}$ where
  $\q{F}\bydef\partial_{\V{\eta}} \partial_{\V{\xi}}
  (e^{(\V{\xi},A\V{\eta})}F) \in \R$ for $F \in \Omega^{2N}$, and
  hence $\q{1}=\det A$ by \eqref{e:det-Grassmann}, behaves in many
  ways like a Gaussian integral.
\end{remark}
Using \eqref{e:det-Grassmann}, the Gaussian density
\eqref{e:Gauss-density} can be written as
\begin{equation}
  \label{eq:density}
  \prod_{i=1}^N \frac{dx_i \, dy_i \, \partial_{\eta_i} \partial_{\xi_i}} {2\pi}
  e^{-\frac12(\V{x},A\V{x})-\frac12(\V{y},A\V{y})+\frac12(\V{\xi},A\V{\eta})-\frac12 (\V{\eta},A\V{\xi})}
  .
\end{equation}
The form given by $(2\pi)^{-N}$ times the exponential
in~\eqref{eq:density} is called the \emph{super-Gaussian form}.  Thus
the Gaussian density is the coefficient of the top degree part of the
super-Gaussian form.
 
To lighten the notation, we will now write $u_i \bydef (x_i,y_i,\xi_i,\eta_i)$
and call $u_i$ a \emph{supervector}.
For supervectors $u_i$ and $u_j$ define a form
\begin{equation}
  u_i \cdot u_j
  \bydef
  x_ix_j+y_iy_j - \xi_i\eta_j + \eta_i\xi_j.
\end{equation}
We unite the supervectors $u_i$ into $\V{u} \bydef (u_i)_{i=1}^N$ and
introduce the following shorthand notation for the form
that occurs in the exponent of \eqref{eq:density}:
\begin{equation}
  \label{eq:ff}
  (\V{u}, A\V{u}) \bydef \sum_{i,j=1}^{N} A_{ij} u_{i}\cdot u_{j}.
\end{equation}
For a form $F$ we define \emph{the superintegral of $F$} by
\begin{equation}
  \label{eq:R22-int}
  \int_{(\R^{2|2})^{N}}F \bydef \frac{1}{(2\pi)^{N}}
  \int_{\R^{2N}}d\V{x} \, d\V{y} \, \partial_{\V{\eta}}\, \partial_{\V{\xi}} \, F,
\end{equation}
where $d\V{x} \bydef dx_{N}\dots dx_{1}$ and similarly for
$d\V{y}$. Then, since the coefficient of the top degree part of
\eqref{eq:density} is the density of a Gaussian,
\begin{equation}
  \label{eq:loc-baby}
  \int_{(\R^{2|2})^N}e^{-\frac12 (\V{u},A\V{u})} = 1
  .
\end{equation}
The fact that this superintegral is one is a
simple example of \emph{localisation} for superintegrals of
supersymmetric forms. The rest of this section describes this
phenomenon.

The \emph{supersymmetry generator}
$Q\colon \Omega^{2N}(\R^{2N}) \to \Omega^{2N}(\R^{2N})$ is defined as
\begin{equation}
  Q \bydef \sum_{i=1}^{N} Q_i\,,
  \qquad 
  Q_i \bydef \xi_i \ddp{}{x_i} + \eta_i \ddp{}{y_i} - x_i\ddp{}{\eta_i} + y_i \ddp{}{\xi_i}.
\end{equation}
Thus $Q$ formally exchanges the even and odd generators of $\Omega^{2N}(\R^{2N})$:
\begin{equation}
  Qx_i = \xi_i, \quad Qy_i = \eta_i, \quad Q\xi_i = -y_i, \quad Q\eta_i = x_i.
\end{equation}
A form $F\in \Omega^{2N}(\R^{2N})$ is defined to be
\emph{supersymmetric} if $QF=0$.
Note that $Q$ is an anti-derivation, and hence $Q(F_{1}F_{2})=0$ if $F_{1}$ and $F_{2}$
are both supersymmetric forms.

\begin{example}
  \label{ex:udotu}
  The following forms are supersymmetric:
  \begin{equation} 
    \label{e:udotudef}
    u_i \cdot u_j
    = x_ix_j+y_iy_j - \xi_i\eta_j + \eta_i\xi_j.
  \end{equation}
\end{example}

Much of the magic of supersymmetry is due to the fundamental
\emph{localisation theorem}:

\begin{theorem} \label{thm:F0}
  Suppose $F \in\Omega^{2N}(\R^{2N})$ is supersymmetric and integrable. Then
  \begin{equation}
    \int_{(\R^{2|2})^{N}} F = F_0(0)
  \end{equation}
  where the right-hand side is the degree-$0$ part of $F$ evaluated at $0$.
\end{theorem}

To keep this introduction to supersymmetry self-contained, we provide
the beautiful and instructive proof of this theorem in
Appendix~\ref{app:localisation}. To prove an important
corollary of the theorem we need the following chain rule, proven
in~\cite[p.59]{MR1843511} or \cite[Solution to Exercise~11.4.3]{rg-brief}.

\begin{lemma}
\label{lem:Qchain}
The supersymmetry generator $Q$ obeys the chain rule for even forms,
in the sense that if $K = (K_j)_{j=1}^{J}$ is a finite collection of
even forms, and if $f \colon \R^J \to \C$ is $C^\infty$, then
\begin{equation}
 \label{e:Qcr}
     Q(f(K)) = \sum_{j=1}^J f_j(K) QK_{j},
\end{equation}
where $f_j$ denotes the partial derivative of $f$ with respect to the
$j$th coordinate.
\end{lemma}

Let $\V{u}\V{u}^{T}$ denote
the collection $(u_{i}\cdot u_{j})_{i,j=1}^{N}$ of forms defined in \eqref{e:udotudef}.

\begin{corollary}
  \label{cor:localisation}
  For any smooth function
  $f\colon \R^{N\times N} \to \R$ with sufficient decay,
  \begin{equation}
    \int_{(\R^{2|2})^{N}} f(\V{u}\V{u}^T) = f(\V{0}).
  \end{equation}
\end{corollary}
\begin{proof}
  Let $F = f(\V{u}\V{u}^T)$. Then $F_0(\V{0}) = f(\V{0})$ and
  $QF = \sum_{ij} f_{ij}(\V{u}\V{u}^T) Q(u_i\cdot u_j) = 0$ by the
  chain rule of Lemma~\ref{lem:Qchain}, 
  where $f_{ij}$ denotes the partial derivative of $f$ with respect
  to the $ij$-th coordinate.  The claim follows from
  Theorem~\ref{thm:F0}.
\end{proof}

\subsection{Change of generators}
\label{sec:cov}

Recall the general expression~\eqref{eq:b-gf} for a form
$F\in\Omega^{M}(\R^{N})$. We will sometimes write $F(\V{x},\V{\xi})$ or
$F(x_{1},\dots, x_{N},\xi_{1},\dots, \xi_{M})$ to denote a form
written in this way. 

\begin{definition}
  \label{def:gens}
  A collection of even elements $(x_i)_{i=1}^N$ and odd elements $(\xi_j)_{j=1}^M$
  is a \emph{set of generators} for $\Omega^M(\R^N)$ if every $F \in
  \Omega^M(\R^N)$ can be written in the form \eqref{eq:b-gf}.
\end{definition}

Note that Example~\ref{ex:CoV-df} provided an example of a
change of generators
\begin{equation}
  \label{eq:CoV-df-revisit}
  y_{i} = \Phi_{i}(x_{1},\dots, x_{N}), \quad 
  \eta_i = dy_i = \sum_{j=1}^{N} \ddp{\Phi_i}{x_j}(x_{1}, \dots, x_{N}) \, dx_j
\end{equation}
along with a corresponding change of variables formula.

It is both possible and useful to
change between sets of generators in the sense of
Definition~\ref{def:gens} without the even and odd generators changing
together. Moreover, there is an extension of the usual change of
variables formula that applies in this setting.
This formula relies on the notion of \emph{superdeterminant} (or Berezinian)
  of a \emph{supermatrix} $M$:
\begin{equation}
  \label{eq:sdet}
  \sdet M \bydef \det (A-BD^{-1}C)\det D^{-1} \quad \text{for }
  M = \begin{pmatrix} A & B \\ C & D \end{pmatrix},
\end{equation}
where the entries of $M$ are elements of a Grassmann algebra, the
entries of the blocks $A$ and $D$ are even, the entries of the blocks
$B$ and $C$ are odd, and $D$ is invertible. Invertibility means
invertibility in the (commutative) algebra of even elements of the
Grassmann algebra.  The next result
is~\cite[Theorem~2.1]{MR914369}. In the theorem rapid decay
  means each of the coefficient functions of $F$ have rapid decay.
\begin{theorem}
  \label{thm:susy-cov}
  Suppose $y_{i}=y_{i}(\V{x},\V{\xi})$ and $\eta_{i} =
  \eta_{i}(\V{x},\V{\xi})$ are a set of generators. Then for any $F$
  with sufficiently rapid decay,
  \begin{equation}
    \label{eq:susy-cov}
    \int d\V{y} \, \partial_{\V{\eta}}\, F(\V{y},\V{\eta}) \,
    \sdet(M) = \int d\V{x}\, \partial_{\V{\xi}}\, F(\V{x},\V{\xi}),
  \end{equation}
  where $M$ is of the form in \eqref{eq:sdet} with entries
  $A_{ij} = \ddp{y_{i}}{x_{j}}$, $B_{ij} = \ddp{y_{i}}{\xi_{j}}$,
  $C_{ij} = \ddp{\eta_{i}}{x_{j}}$,
  $D_{ij} = \ddp{\eta_{i}}{\xi_{j}}$.
\end{theorem}

Implicit in Theorem~\ref{thm:susy-cov} is that a change of
generators always results in an invertible $D$, so the
superdeterminant is well-defined.

\begin{example} \label{ex:chvarmixed}
  Let $x,\xi_1,\xi_2$ be generators for $\Omega^2(\R)$. Then
  the set of forms
  $\{x + g(x)\xi_1\xi_2, \xi_{1}, \xi_{2}\}$
  is also a set of generators, and
  \begin{equation}
    \label{eq:gcov-1}
    \int dx\, \partial_{\xi_1}\partial_{\xi_2} \, F(x,\xi_1,\xi_2)
    =
    \int dx\, \partial_{\xi_1}\partial_{\xi_2} \,
    F(x+g(x)\xi_1\xi_2,\xi_1,\xi_2) (1+g'(x)\xi_1\xi_2)  .
  \end{equation}
\end{example}

It is instructive to verify the claims of the previous example
by hand, and we briefly do so.  To see the claim that these forms are a set
of generators, recall that by definition 
\begin{equation}
  \label{eq:gcov-2}
  F(x+g(x)\xi_1\xi_2,\xi_1,\xi_2) = F(x,\xi_1,\xi_2) + F'(x,\xi_1,\xi_2) g(x)\xi_1\xi_2.
\end{equation}
Letting $y\bydef g(x)\xi_{1}\xi_{2}$, a general form of $\{x +
  g(x)\xi_1\xi_2, \xi_{1}, \xi_{2}\}$ is thus, for some functions $a,b,c,d$,
\begin{equation*}
  a(x+y) + b(x+y)\xi_{1} + c(x+y)\xi_{2} + d(x+y)\xi_{1}\xi_{2}
  =
  a(x) + b(x)\xi_{1} + c(x)\xi_{2} +
  (d(x)+a'(x)g(x))\xi_{1}\xi_{2},
\end{equation*}
which clearly shows a general form in $\{x,\xi_{1},\xi_{2}\}$ can be
expressed as a form in $\{x+g(x)\xi_{1}\xi_{2},\xi_{1},\xi_{2}\}$.

To verify \eqref{eq:gcov-1} integrate \eqref{eq:gcov-2}. Integrating
the term containing $F'$ by parts yields
\begin{equation}
  \int dx \, \partial_{\xi_1}\partial_{\xi_2} \,
  F(x+g(x)\xi_1\xi_2,\xi_1,\xi_2)
  =
  \int dx \, \partial_{\xi_1}\partial_{\xi_2} \,
  F(x,\xi_1,\xi_2) (1-g'(x)\xi_1\xi_2).
\end{equation}
Since
$F(x+g(x)\xi_1\xi_2,\xi_1,\xi_2) g'(x)\xi_1\xi_2 = F(x,\xi_1,\xi_2)
g'(x)\xi_1\xi_2$, \eqref{eq:gcov-1} follows. This can
  alternately be verified by computing the superdeterminant of
  \begin{equation}
    \label{eq:exM}
    M = \begin{pmatrix} 1+g'(x)\xi_{1}\xi_{2} & \xi_{2} & -\xi_{1} \\
      0 & 1 & 0 \\ 0 & 0 & 1 \end{pmatrix}.
  \end{equation}

\section{Further aspects of symmetries and supersymmetry}
\label{sec:fSUSY}

This appendix discusses some additional aspects of
supersymmetry. First, we briefly introduce complex coordinates, which
have often been used in the literature  (see, e.g.,~\cite{MR2525670}).
Second, we prove Theorem~\ref{thm:F0}. The remaining sections
discuss symmetries and Ward identities, and in particular,
highlight how Theorem~\ref{thm:F0} is an example of a Ward
identity arising from an infinitesimal supersymmetry.

\subsection{Complex coordinates}
\label{sec:CC}

In Appendix~\ref{sec:SUSY-intro} we introduced Grassmann algebras over $\R$
and forms given by smooth functions with values in $\R$. Sometimes it
is convenient to work with Grassmann algebras over $\C$ and
complex-valued functions, and many discussions of supersymmetry do so,
see \cite{MR2525670} and references therein. To facilitate comparisons
with the literature we briefly introduce complex coordinates and
relate them to the presentation of Appendix~\ref{sec:SUSY-intro}.

To introduce complex coordinates we set
\begin{equation}
  z = \frac{1}{\sqrt{2}} ( x + i y), \quad \bar z = \frac{1}{\sqrt{2}} (x - i y), \quad
  \zeta = \frac{1}{\sqrt{2 i}} (\xi + i\eta), \quad \bar\zeta =
  \frac{1}{\sqrt{2i}} (\xi-i\eta).
\end{equation}
Correspondingly, define
\begin{equation}
  \ddp{}{z_i} = \frac{1}{\sqrt{2}} \pa{\ddp{}{x_i} -i \ddp{}{y_i}}, \quad
  \ddp{}{\bar z_i} = \frac{1}{\sqrt{2}} \pa{\ddp{}{x_i} +i \ddp{}{y_i}},
\end{equation}
and define $\partial_{\zeta_i}$ and $\partial_{\bar\zeta_i}$ to be the antiderivations on $\Omega^{2N}$ such that
\begin{equation}
  \ddp{}{\zeta_i} \zeta_j
  =
  \ddp{}{\bar\zeta_i} \bar\zeta_j = \delta_{ij}, \qquad
  \ddp{}{\zeta_i} \bar\zeta_j
  =
  \ddp{}{\bar\zeta_i} \zeta_j = 0.
\end{equation}
Up to an irrelevant factor of $\sqrt{i}$ (
a constant factor plays no role in  determining if a form is supersymmetric),
the supersymmetry generator can be written in complex coordinates as
\begin{equation} \label{e:Qdef}
  Q = \sum_{i=1}^N Q_i, \quad Q_i = \zeta_i \ddp{}{z_i} + \bar\zeta_i \ddp{}{\bar z_i} - z_i \ddp{}{\zeta_i} + \bar z_i \ddp{}{\bar\zeta_i}.
\end{equation}
Hence it acts on the complex generators by
\begin{equation}
  Qz_i = \zeta_i, \quad Q\bar z_i = \bar \zeta_i, \quad Q\zeta_i = -z_i, \quad Q\bar\zeta_i = \bar z_i.
\end{equation}
Writing $u_{i}=(z_{i},\zeta_{i})$ for $i=1,\dots, N$, the following
forms are supersymmetric:
\begin{equation} 
  u_i \cdot \bar u_j \bydef z_i \bar z_j + \zeta_i \bar \zeta_j.
\end{equation}

\paragraph{Realisation by differential forms}

Complex coordinates can be realised in terms of differential forms as
follows.  Denote the coordinates of $\R^{2}$ by $x$ and $y$ with
differentials $dx$ and $dy$, and set
\begin{equation}
  z = \frac{1}{\sqrt{2}} ( x + i y), \quad \bar z = \frac{1}{\sqrt{2}}
  (x - i y), \quad dz = \frac{1}{\sqrt{2i}} (dx + i dy), \quad
  d\bar z = \frac{1}{\sqrt{2i}}( dx - i dy).
\end{equation}

\subsection{Proof of Theorem~\ref{thm:F0}}
\label{app:localisation}

\renewcommand{\phi}{z}
\renewcommand{\phib}{\bar z}
\renewcommand{\psi}{\zeta}
\renewcommand{\psib}{\bar \zeta}

The proof of Theorem~\ref{thm:F0} will use the complex
coordinates introduced in Appendix~\ref{sec:CC}, and will also make use of the
following terminology and facts. A form is called \emph{$Q$-closed}
(\emph{supersymmetric}) if $QF=0$ and it is called \emph{$Q$-exact} if
$F=QG$ for some form $G \in \Omega^{2N}(\R^{2N})$. The $Q$-closed
forms $u_{i}\cdot u_{j}$ from Example~\ref{ex:udotu} are also
$Q$-exact, as can be verified by checking
\begin{equation}
  z_i\bar z_j + z_j\bar z_i + \zeta_i\bar\zeta_j - \bar\zeta_i\zeta_j= Q\lambda_{ij}, \quad \lambda_{ij} \bydef z_i\bar\zeta_j + z_j\bar\zeta_i.
\end{equation}
\begin{proof}[Proof of Theorem~\ref{thm:F0}]
Any integrable form $F$ can be written as
$K=\sum_{\alpha}F^{\alpha}\psi^{\alpha}$ with (i) $\psi^{\alpha}$  a
monomial in $\{\psi_{i},\psib_{i}\}_{i=1}^{N}$ and (ii)
$F^{\alpha}$ an integrable function of $\{\phi_{i},\phib_{i}\}_{i=1}^{N}$.
To emphasise this, we write $K=K(\V{\phi}, \V{\phib}, \V{\psi},\V{\psib})$.
To simplify notation we write $\int$ in place of $\int_{(\R^{2|2})^{N}}$.

\smallskip\noindent
\emph{Step 1.}
Let $S= \sum_{i=1}^{N} (\phi_i\phib_i + \psi_i \psib_i)$. 
We prove the following version of Laplace's Principle:
\begin{equation} \label{e:localisation-Laplace}
  \lim_{t\to\infty} \int e^{-t S}F = F_0(\V{0}) .
\end{equation}
Let $t>0$.
We make the change of generators
$\phi_{i} = \frac{1}{\sqrt{t}}\phi'_{i}$ and $\psi_{i} = \frac{1}{\sqrt{t}}\psi'_{i}$.
This transformation has unit Berezinian.
Let $\omega \bydef - \sum_{i=1}^{N}\psi_{i} \psib_{i}$. 
After dropping the primes, we obtain
\begin{equation}
    \int e^{-t S}F
    =
    \int e^{-\sum_{i=1}^{N}\phi_{i} \phib_{i} + \omega}
    F(
    \tfrac{1}{\sqrt{t}}\V{\phi},
    \tfrac{1}{\sqrt{t}}\V{\phib},
    \tfrac{1}{\sqrt{t}}\V{\psi},
    \tfrac{1}{\sqrt{t}}\V{\psib}) ,
\end{equation}
where $\tfrac{1}{\sqrt{t}}\V{\phi} \bydef \{\tfrac{1}{\sqrt{t}}\phi_i\}_{i=1}^N$,
and similarly for the other generators.
To evaluate the right-hand side, we expand $e^{\omega}$ and and obtain
\begin{align}
\label{e:etSK}
    \int e^{-t S}F
    &=
    \sum_{n=0}^{N}
    \int e^{-\sum_{i=1}^{N}\phi_{i} \phib_{i}} \frac{1}{n!}\omega^{n}
    F(
    \tfrac{1}{\sqrt{t}}\V{\phi},
    \tfrac{1}{\sqrt{t}}\V{\phib},
    \tfrac{1}{\sqrt{t}}\V{\psi},
    \tfrac{1}{\sqrt{t}}\V{\psib})
    .
\end{align}
We write $K=K^0+G$,  where $K^0$ is the degree zero part of $K$.
The contribution of $K^0$ to \eqref{e:etSK} involves only the
$n=N$ term and equals
\begin{equation} \label{e:etSK0}
  \int e^{-t S}F^0 =
  \int e^{-\sum_{i=1}^{N}\phi_{i} \phib_{i}} \frac{1}{N!} \omega^{N}
    F^0(
    \tfrac{1}{\sqrt{t}}\V{\phi},
    \tfrac{1}{\sqrt{t}}\V{\phib})
    ,
\end{equation}
so by the continuity of $F_0$,
\begin{align}
    \lim_{t\to\infty}
    \int e^{-t S}F_0
    &=
    F_{0}(\V{0})
    \int e^{-\sum_{i=1}^{N}\phi_{i} \phib_{i}} \frac{1}{N!}\omega^{N}
    =
    F_{0}(\V{0})
    \int e^{- S}
    .
\end{align}
By \eqref{eq:loc-baby} with $A$ the identity matrix,
this proves that
\begin{align}
    \lim_{t\to\infty}
    \int e^{-t S}F_0
    &=
    F_{0}(0)
    .
\end{align}

To complete the proof of \eqref{e:localisation-Laplace}, it remains to
show that $\lim_{t\to\infty}\int e^{-t S}G = 0$.  As above,
\begin{equation}
    \int e^{-t S}G
    =
    \sum_{n=0}^{N} \int e^{-\sum_{i=1}^{N}\phi_{i} \phib_{i}} \frac{1}{n!}
    \omega^{n}\,
    G \left(
    \tfrac{1}{\sqrt{t}}\V{\phi},
    \tfrac{1}{\sqrt{t}}\V{\phib},
    \tfrac{1}{\sqrt{t}}\V{\psi},
    \tfrac{1}{\sqrt{t}}\V{\psib}
    \right).
\end{equation}
Since $G$ has no degree-zero part, the term with $n=N$ is zero.  Terms
with smaller values of $n$ require factors $\psi_i\psib_i$ for some $i$ from $G$,
and these factors carry inverse powers of $t$.  They therefore vanish in the limit, and the
proof of \eqref{e:localisation-Laplace} is complete.

\smallskip
\noindent
\emph{Step 2.}
The Laplace approximation is exact:
\begin{equation} \label{e:Laplace-exact}
  \int e^{-t S}F \quad \text{is independent of $t \ge 0$.}
\end{equation}
To prove this, recall that $S = Q\lambda$.
Also, $Q e^{-S}=0$ by the chain rule of Lemma~\ref{lem:Qchain}, 
and $QF=0$ by assumption.  Therefore,
\begin{equation}
     \label{e:dlam}
    \frac{d}{dt} \int e^{-t S}F
    =
    -\int e^{-t S} S  F
    =
    -\int e^{-t S} (Q \lambda)  F
    =
    -\int Q\left(e^{-t S} \lambda F \right)
    =
    0 ,
\end{equation}
since the integral of any $Q$-exact form is zero, because it can be
written as a sum of derivatives (whose integral vanishes due to the
assumption of rapid decay) and a form of degree lower than the top
degree (whose integral vanishes by definition).

\smallskip\noindent \emph{Step 3.} Finally, we combine Laplace's
Principle \eqref{e:localisation-Laplace} and the exactness of the
Laplace approximation \eqref{e:Laplace-exact}, to obtain the desired
result
\begin{equation*}
    \int F
    =
    \lim_{t\rightarrow \infty}
    \int e^{-t S}F
    =
    F_{0} (\V{0}).\qedhere
\end{equation*}
\end{proof}

\subsection{Symmetries}
\label{sec:symmetries}

This appendix briefly reviews symmetries in the context of smooth
manifolds, to prepare the way for a discussion of symmetries of superalgebras.

\subsubsection{Infinitesimal symmetries}

For a smooth manifold $M$, 
infinitesimal symmetries are described by the infinite-dimensional Lie algebra of
smooth vector fields, $\mathrm{Vect}(M)$.  Vector fields act on
functions through the Lie derivative, which 
associates to every vector field $X \in \mathrm{Vect}(M)$ a \emph{derivation}
$T_X\colon C^{\infty}(M)\rightarrow C^{\infty}(M)$.
We recall that a derivation is
a linear map that obeys the Leibniz rule $T_X(fg) = T_X(f)g + fT_X(g)$. Concretely,
if $M$ is $n$-dimensional and $X$ is represented in local coordinates as
$X = \sum_{\alpha=1}^{n}g(u^1,\dots,u^n)\ddp{}{u^\alpha}$, then
$T_X(f) = \sum_{\alpha=1}^{n}g(u^1,\dots,u^n)\ddp{f}{u^\alpha}$.

In fact, every derivation on $C^{\infty}(M)$ arises from a vector
field, and hence there is an isomorphism
$\mathrm{Vect}(M) \simeq \mathrm{Der}(C^{\infty}(M))$.
Thus we can replace geometric objects (vector fields) with algebraic objects
(derivations).
The perspective will be useful for
superspaces, as their definition is fundamentally algebraic rather than geometric.

\subsubsection{Integral symmetries}
Rather than examining the entire Lie algebra
$\mathrm{Der}(C^{\infty}(M))$, it is often useful to consider
subalgebras that respect additional structures on the manifold. We
will be interested in the following case where $M$ carries a measure
$\mu$.  Let $\int_{M}f$ denote the integral of a function
$f\colon M\to \R$ with respect to the measure $\mu$. We call
$\int_{M}$ an \emph{integral on $M$}.
\begin{definition}\label{def:int-sym}
  Let $\int_M$ be an integral on a smooth manifold $M$.  A derivation
  $T\in \mathrm{Der}(C^\infty(M))$ is an \emph{infinitesimal symmetry} of the
  integral if for all $f \in C^\infty(M)$ with rapid decay
  \begin{equation}
    \label{e:T-symmetry}
    \int_M T f = 0.
  \end{equation}
\end{definition}
Infinitesimal symmetries lead to integration by parts formulas,
otherwise known as \emph{Ward identities}: 
suppose $T$ is a symmetry of $\int_M$, and that $f,g\in C^{\infty}(M)$ 
have rapid decay. Then 
\begin{equation}
\int_M T(fg) = 0,
\end{equation}
since $fg$ has rapid decay. Since $T$ acts as a derivation, we obtain
the Ward identity
\begin{equation}
\int_M (Tf) g = -\int_M f (Tg).
\end{equation}

For spin systems, different infinitesimal symmetries are obtained
depending on whether we examine the Gibbs measure $e^{-H_\beta}\, d\V{u}$ or 
the underlying 
measure $d\V{u}$. Ward identities for one lead to (anomalous) Ward identities
for the other. For instance, letting $\uspin{f}_\beta =
\int_{M^{\Lambda}}f e^{-H_\beta} \, d\V{u}$ denote an unnormalised
expectation, and letting $T$ be an infinitesimal symmetry of $d\V{u}$,
\begin{equation}
\int_{M^{\Lambda}} T(fe^{-H_\beta}) = 0, \quad \text{i.e.,} \quad
\int_{M^{\Lambda}} (Tf - f(TH_\beta))e^{-H_\beta} = 0
\end{equation}
and hence
\begin{equation}
\uspin{Tf}_\beta = \uspin{f(TH_\beta)}_\beta.
\end{equation}

\subsubsection{Global symmetries}

For spin system Gibbs measures
$\uspin{{F}}_\beta = \int_{M^{\Lambda}}{F} \, e^{-H_\beta}d\V{u}$, an
important role is played by derivations
$T \in \mathrm{Der}(C^{\infty}(M^\Lambda))$ which can be written in
the form
\begin{equation}
T \bydef \sum_{i \in \Lambda} T_i,
\end{equation}
where each $T_{i}$ is a copy of a single site derivation
\begin{equation}
T_i = \sum_{\alpha = 1}^{n}f_\alpha(u_i)\ddp{}{u_i^\alpha}
\end{equation}
with $f_\alpha$ independent of $i\in\Lambda$.
We call these \emph{diagonal derivations}.
If a diagonal derivation is an  infinitesimal symmetry of the
Gibbs measure, then we say that it is a \emph{global symmetry}. The
spin system Hamiltonians in this paper are of the form
$H_\beta(\V{u}) = \frac{1}{4}\sum_{i,j\in\Lambda}\beta_{ij} (u_i-u_j)^2$ with
$(u_{i}-u_{j})^{2}\bydef (u_{i}-u_{j})\cdot (u_{i}-u_{j})$ for some inner product. Hence
the global symmetries are equivalently
those diagonal derivations which satisfy
\begin{equation}
T(u_i-u_j)^2 = 0
\end{equation}
for all $i,j \in \Lambda$. These correspond to the infinitesimal
isometries of the target space, and form a representation of a finite
dimensional Lie algebra.

For the GFF on $\R^{n}$, the global symmetries are of the form
\begin{equation}\label{eq:euc-global}
T \bydef \sum_{i\in \Lambda}T_i,\quad T_i = \sum_{\alpha,\beta = 1}^n R_{\alpha\beta} u^\alpha_i \ddp{}{u^\beta_i} + \sum_{\gamma = 1}^{n}S_{\gamma}\ddp{}{u^\gamma_i},  
\end{equation}
where $R$ is an $n\times n$ real skew-symmetric matrix and $S$ is a
real vector in $\R^{n}$.  The global symmetries of $\R^{n}$ hence form
a representation of the Euclidean Lie algebra
$\mathfrak{so}(n) \ltimes \R^{n}$ under the Lie bracket of
derivations.  Global symmetries of Minkowski space $\R^{n,1}$ are of
the same form as \eqref{eq:euc-global}, but $R$ is now skew-symmetric
with respect to the Minkowski inner product, i.e.,
\begin{equation}
  R^{T}J + JR = 0, \quad J = \mathrm{diag}(-1,1,\dots,1).
\end{equation}
This gives a representation of the Poincare Lie algbera
$\mathfrak{so}(n,1) \ltimes\R^{n,1}$.

Global symmetries of the $\HH^{n}$ and
$\bbS^{n}_+$ spin models are induced from Lorentz/orthogonal
symmetries of $\R^{n,1}$ and
$\R^{n+1}$ respectively, i.e., global symmetries have
the form
\begin{equation}
T \bydef \sum_{i\in \Lambda}T_i,\quad   T_i = \sum_{\alpha,\beta} R_{\alpha\beta} u^\alpha_i \ddp{}{u^\beta_i}. 
\end{equation}
For the
$\HH^n$ model these form a representation of the Lorentzian Lie
algebra $\mathfrak{so}(n,1)$, and for the
$\bbS^{n}_+$ model these form a representation of the orthogonal Lie
algebra
$\mathfrak{so}(n+1)$. In coordinates, these
symmetries can be written as
\begin{equation}
T \bydef \sum_{i\in \Lambda}T_i,\quad T_i = \sum_{\alpha,\beta = 1}^n R_{\alpha\beta} u^\alpha_i \ddp{}{u^\beta_i} + \sum_{\gamma=1}^nS_{\gamma}z_i\ddp{}{u^\gamma_i}
\end{equation}
where $S_{\gamma} = R_{0\gamma}$ and $z = \sqrt{1 + (u^1)^2 + \dots +
  (u^n)^2}$ for $\HH^n$, while
$S_{\gamma} = R_{(n+1)\gamma}$ and  $z = \sqrt{1 - (u^1)^2 - \dots  - (u^n)^2}$ for $\bbS^n_+$ .

\subsection{Symmetries of supersymmetric spaces}
\label{sec:susy-symmetries}
Infinitesimal symmetries of Berezin integrals and the global
symmetries of supersymmetric spaces have descriptions similar to those
of the previous section.  The primary difference is that all
objects are graded.

\subsubsection{Superderivations and supersymmetries}
Let $A$ be a $\Z^2$-graded algebra (or superalgebra) such as
$A = \Omega^{n}(\R^{m})$.  Thus $A=A_0\oplus A_1$ where elements in $A_0$
are even and elements in $A_1$ are odd.  Using this decomposition, a
linear map $T\colon A \rightarrow A$ can be written in blocks as
\begin{equation}
Tf = \begin{bmatrix}
T_{00}& T_{01}\\
T_{10}&T_{11}
\end{bmatrix}
\begin{bmatrix}
f_0\\f_1
\end{bmatrix}.
\end{equation}
A linear map is \emph{even} if $T_{01} = T_{10} = 0$, and 
\emph{odd} if $T_{00} = T_{11} = 0$. As for functions, a \emph{homogeneous} linear map
is one that is even or odd.
We extend the parity function to homogeneous maps by
\begin{equation}
\alpha(T) =  \begin{cases} 0 \in \Z_2, \qquad T \text{ is even}
  \\
1 \in \Z_2,\qquad T \text{ is odd}
\end{cases},
\end{equation}
and for homogeneous $f$ we have $\alpha(Tf) = \alpha(T)+\alpha(f)$. A
\emph{homogeneous superderivation} is then defined as a
homogeneous linear map $T\colon A
\rightarrow A$ that obeys the super-Leibniz rule
\begin{equation}\label{eq:super-leib}
  T(fg) =  (Tf) g + (-1)^{\alpha(T)\alpha(f)}f(Tg).
\end{equation}

Thus even and odd superderivations are derivations and
antiderivations, respectively. 
A general superderivation is a sum of an even
and an odd superderivation. The collection of superderivations on $A$
forms a Lie superalgebra $\mathrm{SDer}(A)$ with the supercommutator
defined on homogeneous superderivations by
\begin{equation}
[T_1, T_2] = T_1\circ T_2 - (-1)^{\alpha(T_1)\alpha(T_2)}T_2\circ T_1,
\end{equation}
and extended to all superderivations by linearity. If $A= \Omega^{n}(M)$
is a superalgebra of forms on an $m$-dimensional manifold $M$, then
every superderivation $T \in \mathrm{SDer}(A)$ can be realised in
coordinates $(x^1,\dots, x^{m},\xi^1,\dots,\xi^{n})$ as
\begin{equation}
T = \sum_{\alpha = 1}^{m} F_\alpha\ddp{}{x^{\alpha}} + \sum_{\alpha =
  1}^{n} G_\alpha\ddp{}{\xi^{\alpha}} 
\end{equation}
where $F_\alpha,G_\alpha\in A$. If $T$ is an even/odd superderivation
then $F_\alpha$ are even/odd forms and $G_\alpha$ are odd/even forms.

\paragraph{Berezin integral symmetries and global symmetries}
We define a \emph{Berezin integral} $\int_{M}$ on a superalgebra
$\Omega^{n}(M)$ to be a linear map defined by integrating forms $F$
against an even Berezin integral form  $d\V{x}\,\partial_{\V{\xi}} \,
\rho(\V{x},\V{\xi})$, i.e., 
\begin{equation}
\int_M F \bydef \int_{\R^{m|n}} d\V{x}\,\partial_{\V{\xi}}\, \rho(\V{x},\V{\xi})F(\V{x},\V{\xi}). 
\end{equation}
\begin{definition}
  Let $\int_M$ be a Berezin integral 
  on a superalgebra $\Omega^{n}(M)$.  A superderivation
  $T\in \mathrm{SDer}(\Omega^{n}(M))$ is an \emph{infinitesimal
    symmetry} of $\int_{M}$ if for all
  $F \in \Omega^{n}(M)$ with rapid decay
  \begin{equation}
    \label{e:T-SUSY-symmetry}
    \int_M T F = 0.
  \end{equation}
\end{definition}

This leads to Ward identities in the same manner as the
non-supersymmetric case, the only difference coming from the
super-Leibniz rule: for homogeneous superderivations
$T\in \mathrm{SDer}{(\Omega^{n}(M))}$ and forms $F,G\in
\Omega^{n}(M)$
we have
\begin{equation}
\int_M TF = (-1)^{\alpha(T)\alpha(F)+1}\int_M TG.
\end{equation}

Global symmetries of supersymmetric spin systems are
infinitesimal symmetries of the form
\begin{equation}
T \bydef \sum_{i\in\Lambda} T_i,
\end{equation}
i.e., they are diagonal infinitesimal symmetries.
For the spin systems considered in this paper, which are defined in terms of
quadratic Hamiltonians
$\frac{1}{4}\sum_{i,j\in\Lambda}\beta_{ij}(u_{i}-u_{j})^{2}$, global symmetries
are those that
annihilate the appropriate super-Euclidean or super-Minkowski inner product
\begin{equation}
T(u_i-u_j)^2 = 0
\end{equation}
for all $i,j\in \Lambda$.
Here we have written
  $(u_{i}-u_{j})^{2}$ for the form $(u_{i}-u_{j})\cdot
  (u_{i}-u_{j})$. The following subsections briefly discuss this condition for
the $\R^{2|2}, \HH^{2|2}$, and $\bbS^{2|2}_{+}$ models.

\subsubsection{$\R^{2|2}$ model}
\label{sec:r22iso}
The inner product associated to the SUSY GFF is
\begin{equation}
u_i \cdot u_j = x_ix_j + y_iy_j - \xi_i\eta_j + \eta_i\xi_j,
\end{equation}
giving the global symmetries as diagonal superderivations $T\in
\mathrm{SDer}(\Omega^{2\Lambda}(\R^{2\Lambda}))$ satisfying
\begin{equation}
T(u_i - u_j)^2 = T\left((x_i-x_j)^2 + (y_i-y_j)^2 - 2(\xi_i-\xi_j)(\eta_i-\eta_j) \right) = 0
\end{equation}
for all $i,j\in \Lambda$.

Concretely, letting $u_i = (u^1_i,\dots,u^4_i) = (x_i,y_i,\xi_i,\eta_i)$, these are derivations of the form
\begin{equation}
T \bydef \sum_{i\in \Lambda}T_i,\quad T_i = \sum_{\alpha,\beta = 1}^4 R_{\alpha\beta} u^\alpha_i \ddp{}{u^\beta_i} + \sum_{\gamma = 1}^{4}S_{\gamma}\ddp{}{u^\gamma_i} 
\end{equation}
where $R$ is a real $4\times 4$ matrix (independent of $i\in\Lambda$) such that
\begin{equation}
R^{ST}J + JR = 0,
\end{equation}
where $R^{ST}$, the \emph{supertranspose} of $R$, and $J$ are given by
\begin{equation}\label{eq:infsymR22}
R^{ST} \bydef \begin{bmatrix}
A&B\\C&D
\end{bmatrix}^{ST} =  \begin{bmatrix}
A^T&C^T\\-B^T&D^T 
\end{bmatrix},
\qquad 
J \bydef \left[\begin{array}{cc|cc}
1 & 0 & 0 & 0 \\
0 & 1 & 0 & 0 \\\hline
0 & 0 & 0 & -1\\
0 & 0 & 1 & 0
\end{array}\right],
\end{equation}
and $S$ is a real vector. With the supercommutator of
superderivations,  these form a
representation of the super-Euclidean Lie superalgebra
$\mathfrak{osp}(2|2)\ltimes \R^{2|2}$ . In particular, the supersymmetry generator
\begin{equation}
Q \bydef \sum_{i\in\Lambda} Q_i = \sum_{i\in \Lambda} \xi_i \ddp{}{x_i} + \eta_i \ddp{}{y_i} - x_i\ddp{}{\eta_i} + y_i \ddp{}{\xi_i}
\end{equation}
and the infinitesimal global translation
\begin{equation}
T \bydef \sum_{i\in\Lambda} T_i =  \sum_{i\in \Lambda} \ddp{}{x_i}
\end{equation}
are global symmetries. 

A short computation shows that the individual $T_i$ and $Q_i$ are
symmetries of the flat Berezin--Lebesgue measure
$d\V{x}\,d\V{y}\,\partial_{\V{\xi}}\,\partial_{\V{\eta}}$. For
instance, if $F$ is a compactly supported form with top degree
component $F_{2\Lambda}(\V{x},\V{y})\V{\xi}\V{\eta}$,
\begin{equation}
\int_{(\R^{2|2})^\Lambda} (T_i F) = \int_{\R^{2\Lambda}}d\V{x}\,d\V{y}\,\partial_{\V{\xi}}\,\partial_{\V{\eta}} (T_i F) = \int_{\R^{2\Lambda}}d\V{x}\,d\V{y} \ddp{}{x_i}F_{2\Lambda}(\V{x},\V{y}) = 0
\end{equation}
where in the last step we have used the translation invariance of the usual Lebesgue measure. A particular case of this is formula \eqref{eq:IBP-SUSY-GFF}.

\subsubsection{Super-Minkowski space $\R^{3|2}$}
\label{sec:h22iso}
The inner product associated to the super-Minkowski model is the super-Minkowski inner product
\begin{equation}
u_i \cdot u_j = -z_iz_j + x_ix_j + y_iy_j - \xi_i\eta_j + \eta_i\xi_j,
\end{equation}
giving the global symmetries as diagonal superderivations
$T\in \mathrm{SDer}(\Omega^{2\Lambda}(\R^{3\Lambda}))$ satisfying
\begin{equation}
T(u_i - u_j)^2 = T\left(-(z_i-z_j)^{2} + (x_i-x_j)^2 + (y_i-y_j)^2 - 2(\xi_i-\xi_j)(\eta_i-\eta_j) \right) = 0
\end{equation}
for all $i,j\in \Lambda$.
Concretely, letting $u_i = (u_i^0, u_i^1,u_i^2,u_i^3,u_i^4) = (z_i,x_i,y_i,\xi_i,\eta_i)$, these are derivations of the form
\begin{equation}
T \bydef \sum_{i\in \Lambda}T_i,\quad T_i = \sum_{\alpha,\beta = 0}^4 R_{\alpha\beta} u^\alpha_i \ddp{}{u^\beta_i} + \sum_{\gamma = 1}^{5}S_{\gamma}\ddp{}{u^\gamma_i} 
\end{equation}
where $R$ is a real $5\times 5$ matrix such that
\begin{equation}
R^{ST}J + JR = 0
\end{equation}
with $J$ now the $5\times 5$ matrix
\begin{equation}
J = \left[\begin{array}{ccc|cc}
-1 & 0 & 0 & 0 &0\\
0 & 1 & 0 & 0 &0\\
0 &0 & 1 & 0 & 0 \\\hline
0 &0 & 0 & 0 & -1\\
0 &0 & 0 & 1 & 0
\end{array}\right],
\end{equation}
and  $S$
a real vector.
These global symmetries form a representation of the super-Poincare Lie
superalgebra $\mathfrak{osp}(2,1|2)\ltimes \R^{3|2}$ with the supercommutator of superderivations. In particular, the supersymmetry generator
\begin{equation}
Q \bydef \sum_{i\in\Lambda} Q_i = \sum_{i\in \Lambda} \pa{ \xi_i \ddp{}{x_i} + \eta_i \ddp{}{y_i} - x_i\ddp{}{\eta_i} + y_i \ddp{}{\xi_i} }
\end{equation}
and the global Lorentz boost
\begin{equation}
T \bydef \sum_{i\in\Lambda} T_i =  \sum_{i\in \Lambda} \pa{ z_i\ddp{}{x_i} + x_i \ddp{}{z_i} }.
\end{equation}
are global symmetries of the super-Minkowski spin model. As for the
$\R^{2|2}$ model, the individual $T_i$ and $Q_i$ are symmetries of the
Berezin--Lebesgue measure
$d\V{x}\,d\V{y}\,d\V{z}\,\partial_{\V{\xi}}\,\partial_{\V{\eta}}$.

\subsubsection{$\bbS^{2|2}_{+}$ and $\HH^{2|2}$ models}
\label{sec:sh22iso}
As for their standard counterparts, the global symmetries of the
$\bbS^{2|2}_{+}$ and $\HH^{2|2}$ models are induced from the ambient
super-Euclidean and super-Minkowski spaces. In both cases, the global
symmetries in ambient coordinates are
\begin{equation}
  T \bydef \sum_{i\in \Lambda}T_i,\quad T_i = \sum_{\alpha,\beta =
    0}^4 R_{\alpha\beta} u^\alpha_i \ddp{}{u^\beta_i}, 
\end{equation}
which form a representation of $\mathfrak{osp}(2,1|2)$ for the
$\HH^{2|2}$ model, and a representation of $\mathfrak{osp}(3|2)$ for
$\bbS^{2|2}_+$. In coordinates, the $T_i$ are written
\begin{equation}
  T_i = \sum_{\alpha,\beta = 1}^4 R_{\alpha\beta} u^\alpha_i
  \ddp{}{u^\beta_i} + \sum_{\gamma=1}^4S_{\gamma}z_i\ddp{}{u^\gamma_i} 
\end{equation}
with $z_i = \sqrt{1+x_i^2+y_i^2 -2\xi\eta}$ for $\HH^{2|2}$ and
$z_i = \sqrt{1-x_i^2-y_i^2 +2\xi\eta}$ for $\bbS^{2|2}_+$ and
$S_\gamma = R_{3\gamma}$ in both cases. As before, the supersymmetry
generator
\begin{equation}
  Q \bydef \sum_{i\in\Lambda} Q_i = \sum_{i\in \Lambda} \xi_i
  \ddp{}{x_i} + \eta_i \ddp{}{y_i} - x_i\ddp{}{\eta_i} + y_i
  \ddp{}{\xi_i} 
\end{equation}
is a global symmetry of both the $\HH^{2|2}$ and $\bbS^{2|2}_+$
models, as is the global Lorentz boost/rotation
\begin{equation}
  T \bydef \sum_{i\in\Lambda} T_i =  \sum_{i\in \Lambda} z_i\ddp{}{x_i}.
\end{equation}
A short computation also shows that the individual $T_i$ and $Q_i$ are
symmetries of the Berezin--Haar measure
$d\V{x}\,d\V{y}\,\partial_{\V{\xi}}\,\partial_{\V{\eta}}
\frac{1}{\prod_{i\in \Lambda} z_i}$.

\subsection{SUSY delta functions}
\label{app:delta}

We begin by defining Dirac delta functions to integrate against forms
$F$ in $\Omega^{2}(\R^{2})$. We will assume $F$ is given by a smooth
function of an even form.  Let $u_{0}=(0,0,0,0)\in \R^{2|2}$, and let
$G\in \Omega^{2}(\R^{2})$ be a smooth compactly supported form with
$\int_{\R^{2|2}}G = 1$. For $\epsilon>0$ define smooth forms
\begin{equation}
  \label{eq:delta}
  \delta^{(\epsilon)}_{u_{0}}(u) \bydef G(\frac{1}{\epsilon}u), \qquad
  \frac{1}{\epsilon}u = (\frac{x}{\epsilon}, \frac{y}{\epsilon},
  \frac{\xi}{\epsilon}, \frac{\eta}{\epsilon}). 
\end{equation}
We then define
\begin{equation}
  \label{eq:delta-1}
  \int_{\R^{2|2}} F(u) \delta_{u_{0}} \bydef \lim_{\epsilon\to 0}
  \int_{\R^{2|2}} F(u) \delta^{(\epsilon)}_{u_{0}}(u).
\end{equation}
The change of generators that rescales each generator
by $\epsilon^{-1}$ has unit Berezinian, and hence
\begin{equation}
  \label{eq:delta-2}
  \int_{\R^{2|2}} F(u)\delta_{u_{0}} = \lim_{\epsilon\to
    0}\int_{\R^{2|2}} F(\epsilon u) \delta^{(1)}_{u_{0}}(u) =
  F_{0}(0)\int_{\R^{2|2}} \delta^{(1)}_{u_{0}}(u) = F_{0}(0), 
\end{equation}
where we recall $F_{0}$ is the degree zero part of $F$.
In the third equality we have used that the degree $p$ parts of $F$
for $p\geq 1$ carry factors of $\epsilon$, and hence vanish in the
limit. The last equality follows
since $\int_{\R^{2|2}} \delta^{(1)}_{u_{0}} = \int_{\R^{2|2}} G = 1$.

Suppose
$\theta_{s}\colon (x,y,\xi,\eta)\mapsto
(\theta_{s}x,\theta_{s}y,\theta_{s}\xi,\theta_{s}\eta)$ is invertible
with inverse $\theta_{-s}$, and that $\theta_{s}u_{0}$ only has
non-zero even components. In this setting we define
$\delta_{\theta_{s}u_{0}}(u)$ by $\delta_{u_{0}}(\theta_{-s}u)$. If
the transformation $\theta_{s}$ has unit Berezinian, then we obtain
\begin{equation}
  \label{eq:delta-3}
  \int_{\R^{2|2}} F(u) \delta_{\theta_s u_0}(u) = \int_{\R^{2|2}}
  F(u) \delta_{u_0}(\theta_{-s}u) = \int_{\R^{2|2}} F(\theta_s u) \delta_{u_0}(u)
  = F_{0}(\theta_s u_0).
\end{equation}

It is often convenient to choose $G$ as a supersymmetric form. For
$\R^{2|2}$, this can be achieved by choosing any smooth compactly
supported function $g\colon \R \rightarrow \R$ with $g(0) = 1$, and
setting $G = g(|u|^2)$.

The definition of delta functions on $\Omega^{2N}(\R^{2N})$ is
analogous, but now based on a smooth compact form $G\in \Omega^{2N}(\R^{2N})$. 

For $\HH^{2|2}$ and $\bbS^{2|2}_{+}$, we define delta functions by
making using of the definition on $\R^{2|2}$. Namely, for $\HH^{2|2}$
in the coordinates $\tilde u = (x,y,\xi,\eta)$ with
$z(\tilde u) = \sqrt{1+ x^2+y^2-2\xi\eta}$, we set
\begin{equation}
  \delta^{(\epsilon, \HH^{2|2})}_{u_0}(u) = z(\tilde u)\delta^{(\epsilon)}_{\tilde u_0}(\tilde u)
\end{equation}
where $u_0=(1,0,0,0,0)
\in \HH^{2|2}$, $\delta^{(\epsilon)}_{\tilde u_0}(\tilde u)$
is a delta function for $\R^{2|2}$ as constructed above,
and $\tilde u_0 = (0,0,0,0) \in \R^{2|2}$. 
Then
\begin{equation}
  \lim_{\epsilon\to 0} \int_{\HH^{2|2}} F \delta^{(\epsilon, \HH^{2|2})}_{u_0}
  =
  \lim_{\epsilon\to 0} \int_{\R^{2|2}} F(z(\tilde u),x,y,\xi,\eta)
  \delta^{(\epsilon)}_{\tilde{u}_0}(\tilde u) =   F_{0}(1,0,0), 
\end{equation}
i.e., the zero-degree part of $F$ evaluated at the point $(z,x,y) = (1,0,0) \in \HH^{2}$.
The construction for $\bbS^{2|2}_+$ is analogous.

\section*{Acknowledgements}
\label{sec:acknowledgements}

We thank Christophe Sabot for pointing out an error in an earlier
version of this article. RB and TH would like to thank the Isaac
Newton Institute for Mathematical Sciences for support and hospitality
during the programme ``Scaling limits, rough paths, quantum field
theory'' when work on this paper was undertaken; this work was
supported by EPSRC grant no.\ EP/R014604/1. TH is supported by EPSRC
grant no.\ EP/P003656/1.

\bibliography{all}
\bibliographystyle{plain}

\end{document}